\documentclass[11pt,twoside,a4paper]{article}
\usepackage{float}
\usepackage[utf8]{inputenc}
\usepackage{import}
\usepackage{amsmath}
\usepackage{amsfonts}
\usepackage{amssymb}
\usepackage{amsthm}
\usepackage{centernot}
\usepackage{graphicx}
\usepackage{float}
\usepackage{calligra}
\usepackage{url}
\usepackage{indentfirst}
\usepackage{enumerate}
\usepackage{pdfpages}
\usepackage[lmargin=1.4cm,tmargin=1.7cm,bmargin=1.3cm,rmargin=1.4cm]
{geometry}
\usepackage[colorlinks=true,
	bookmarksnumbered=true,
	bookmarksopen=true,
	bookmarksopenlevel=3,
	pdfstartview=FitH,
	linkcolor=cyan,
	pdfmenubar=true,
	pdftoolbar=true,
	bookmarks=true,
	citecolor=cyan,
	urlcolor=magenta,
	filecolor=cyan,
	menucolor=black,
	plainpages=false,
	pdfpagelabels]{hyperref}
\usepackage{fancyhdr}

\newtheorem{df}{Definition}[section]

\newtheorem{thr}{Theorem}[section]
\newtheorem*{thr*}{Theorem}

\newtheorem{lem}{Lemma}[section]
\newtheorem{rem}{Remark}[section]

\def\fin { \vskip 0pt \hfill $\blacksquare$ \vskip 12pt}
\numberwithin{equation}{section}
\newcommand{\func}[1]{\operatorname{#1}}

\begin{document}

\title{ Well-posedness for the Navier-Stokes equations in \\
Morrey spaces on non-compact manifolds }
\author{\textbf{V\'{\i}ctor Chaves-Santos $^{1}$}, \textbf{Lucas C. F.
Ferreira $^{2}$}{\thanks{Email addresses: victor.ch.sant@gmail.com (VCS), lcff@ime.unicamp.br (LCFF, corresponding author).}} \\
{\small $^{1, \, 2}$ State University of Campinas (Unicamp),
IMECC-Department of Mathematics,} \\
{\small Rua S\'{e}rgio Buarque de Holanda, 651, CEP 13083-859, Campinas, SP,
Brazil.} }
\date{}
\maketitle

\begin{abstract}
We analyze the incompressible Navier-Stokes equations on a class of non-compact Riemannian manifolds within the framework of Morrey spaces.
Assuming bounded geometry together with negative Ricci and sectional curvature (e.g., hyperbolic spaces), we establish dispersive and smoothing
estimates for the heat semigroups associated with the Beltrami, Bochner and Hodge Laplacians in Morrey spaces, as well as for the Riesz transform. In particular, the presence of negative curvature yields improved large-time decay compared to the Euclidean setting. These estimates are of independent interest and enable us to construct solutions in time-weighted spaces of Kato type, leading to local-in-time well-posedness on a broad class of non-compact manifolds and global one in the case of Einstein manifolds. In the latter setting, we assume a smallness condition on the initial data in Morrey norms, which are weaker than $L^{p}$-norms and thus allow for certain classes of large $L^{p}$-data. We also discuss extensions to Ricci-flat manifolds. Our results introduce a new class of non-decaying and rough initial data for the Navier-Stokes equations on manifolds, extending previous works in Lebesgue and Sobolev spaces.


\end{abstract}

\tableofcontents


\section{Introduction}

We are concerned with the incompressible Navier-Stokes system on smooth,
complete, simply connected, non-compact Riemannian manifolds $M$ without
boundary, of dimension $m\geq 3.$ This system of equations describes the
motion of a viscous incompressible fluid in $M$ and, as explained by
Ebin-Marsden \cite{ebin} and Chang-Czubak-Disconzi \cite{chan}, its
appropriate formulation on manifolds is given by

\begin{equation}
\begin{cases}
\partial _{t}u+\nabla _{u}u+\dfrac{1}{\rho }\func{grad}\,p=\mathcal{L}_{\nu
}(u), \\
\func{div}(u)=0, \\
u(0,\cdot )=u_{0},%
\end{cases}
\label{ns1}
\end{equation}%
where%
\begin{equation*}
\mathcal{L}_{\nu }(u)=\nu \left( \overrightarrow{\Delta }u+r(u)\right) \text{
and }r(u)=\func{Ric}(u,\cdot )^{\sharp }.
\end{equation*}%
Here $u(t,\cdot )\in \Gamma (TM)$ denotes the velocity field, the function $%
p(t,\cdot )$ is the scalar pressure, and the constant $\rho >0$ represents
the fluid density, where $\Gamma (TM)$ is the space of vector fields on $M$.
The operator $\overrightarrow{\Delta }u=-\nabla ^{\ast }\nabla u$ stands for
the Bochner Laplacian. It is related to the Hodge Laplacian $\Delta
_{H}=-(dd^{\ast }+d^{\ast }d)$, where $d$ denotes the exterior derivative
and $d^{\ast }$ its formal adjoint; more precisely, one has $\overrightarrow{%
\Delta }u=\Delta _{H}+r(u)$ (Weitzenb\"{o}ck identity, see \eqref{weit}). A
complete derivation can be found in \cite{Taylor}. As usual, for a $1$-form $%
\omega $, we denote the vector field $\omega ^{\sharp }$ such that $g(\omega
^{\sharp },Y)=\omega (Y)$, for all $Y\in \Gamma (TM)$. Moreover, the symbol $%
\nabla $ denotes the covariant derivative, and for $u\in \Gamma \left(
TM\right) $ we have that $|u|=\sqrt{g(u,u)}$.

Following the approach in Lebesgue spaces by \cite{Pierfelice2017}, we apply
the divergence operator to the system above. Using the Weitzenb\"{o}ck
identity \eqref{weit} and the divergence free condition, we obtain $\func{div%
}(\overrightarrow{\Delta }u)=\func{div}(r(u))$ leading to the expression
\begin{equation*}
\frac{1}{\rho }\func{grad}\,p=\func{grad}(-\Delta _{g})^{-1}\func{div}\left(
\nabla _{u}u\right) -2\nu \,\func{grad}(-\Delta _{g})^{-1}\func{div}(r(u)),
\end{equation*}%
where $\Delta _{g}=\func{div}(\func{grad}\,p)$ is the so-called
Laplace-Beltrami operator (Beltrami Laplacian). Accordingly, we work with
the Navier-Stokes system in the following form
\begin{equation}
\begin{cases}
\partial _{t}u-\nu \left( \overrightarrow{\Delta }u+r(u)-B(u)\right) =-%
\mathbb{P}(\nabla _{u}u), \\
\func{div}(u)=0, \\
u(0,\cdot )=u_{0},%
\end{cases}
\label{ns-2}
\end{equation}%
where
\begin{equation*}
B(u)=-2\,\func{grad}(-\Delta _{g})^{-1}\func{div}(r(u)),\text{ }\mathbb{P}=I+%
\func{grad}(-\Delta _{g})^{-1}\func{div},\text{ and }\nabla _{u}u=\func{div}%
(u\otimes u).
\end{equation*}

The geometric setting characterized by bounded geometry, negative Ricci
curvature, and negative sectional curvature offers a rich framework for
analysis and differential geometry, in particular covering Einstein
manifolds with negative curvature and hyperbolic spaces. For example,
Pierfelice \cite{Pierfelice2017} developed heat kernel estimates and
established the well-posedness of the Navier-Stokes equations in $L^{p}$%
-spaces by means of energy-type methods. Related results in Lebesgue spaces
can be also found in \cite{NGUYEN2022}, \cite{Nguyen2022}. For results in
the case $m=2$, we refer to the works \cite{Pierfelice2017}, \cite{Czubak},
and the references therein.

On the other hand, the incompressible Navier-Stokes equations have been
extensively studied in the framework of Morrey spaces on the Euclidean
setting. Results on existence, uniqueness, and regularity of solutions in
this context can be found, for instance, in \cite{Giga1},\cite{Kato1992},%
\cite{Taylor1992}. Extensions of these results to compact manifolds were
also obtained in \cite{Taylor1992}, where, compared to the non-compact case,
the influence of geometric effects enters the analysis in a more controlled
manner thanks to the compactness property. Moreover, Mitrea and Taylor \cite%
{Mitrea2001} analyzed the steady Navier-Stokes equations on connected
Lipschitz domains in smooth compact Riemannian manifolds, relying on the
framework of Besov spaces.

The theory of Lebesgue and Sobolev spaces on manifolds can, in many
respects, be developed by adapting the corresponding theory in the Euclidean
setting (see, e.g., \cite{Hebey1996}, \cite{Hebey2000}). In this spirit, and
motivated by the Euclidean definition, Morrey spaces can also be naturally
defined on a non-compact manifold $(M,g)$, see Definition \ref{morrey}.
Here, we focus primarily on two such extensions: the formulation based on
geodesic balls, denoted by $\mathcal{M}_{p,\lambda }^{g}$, and a simplified
radius-based version, denoted by $\mathcal{M}_{p,\lambda }$. Adams \cite{ad}
studied aspects of the analytic theory of Morrey spaces in the Euclidean
setting $\mathbb{R}^{m}$ and discussed their extension to manifolds via the
space $\mathcal{M}_{p,\lambda }^{g}$. Some basic properties of the Morrey
spaces $\mathcal{M}_{p,\lambda }^{g}$ and $\mathcal{M}_{p,\lambda }$
resemble those of their Euclidean counterparts (see \cite{ad},\cite{Giga1},%
\cite{Kato1992},\cite{Taylor1992}).

Our intent is to consider the Navier-Stokes equations on non-compact
Riemannian manifolds within the framework of Morrey spaces, thereby
providing a new setting to the problem in non-compact geometries of
dimensions $m\geq 3$. In this framework, and considering bounded geometry
together with negative Ricci and sectional curvature, we establish both
local and global-in-time well-posedness results for initial data $u_{0}\in
\mathcal{M}_{p,\lambda }^{g}$ (see Theorems \ref{th-local} and \ref{th-ein}%
), introducing a novel class of rough initial data for the evolution of (\ref%
{ns-2}). By employing time-weighted spaces (\textit{\`{a} la Kato}) and a
contraction argument, we obtain solutions in the class $\mathcal{C}%
_{b}((0,T),\mathcal{M}_{p,\lambda }^{g})$ $\cap X_{T},$ where the
time-weighted space $X_{T}$ is defined by
\begin{equation*}
X_{T}=\left\{ u\in L_{loc}^{\infty }((0,T),\mathcal{M}_{q,\lambda
}^{g}(\Gamma (TM)))\,;\,\left[ d(t)^{\frac{m}{2}}t^{-\frac{\lambda }{2}}%
\right] ^{\left( \frac{1}{p}-\frac{1}{q}\right) }e^{\beta t}\Vert u(t)\Vert
_{g,q,\lambda }\in L^{\infty }(0,T)\right\} .
\end{equation*}%
Here $\mathcal{C}_{b}$ denotes the space of bounded and continuous
functions, $m-\lambda \leq p<q$, $d(t)=\min \{1,t\}$, and $\beta $ is chosen
in accordance with the large-time decay rate appearing in the smoothing
estimate (see Remark \ref{r-sm}). The existence time $T>0$ may be finite or
infinite, depending on whether the result is local or global in time. In
fact, depending on the conditions on the geometric setting or the equations,
we consider slight variants of the space $X_{T}$ defined above.

Morrey spaces lack global $L^{p}$-integrability and contain elements
exhibiting infinitely many singularities and non-decaying behavior at
infinity (i.e., nondecaying data), thereby accommodating a wider class of
functions. Originally introduced to study the local behavior of solutions to
elliptic PDEs, these spaces provide a natural setting for analyzing
singularities, concentration phenomena, partial regularity of weak
solutions, and the well-posedness of nonlinear equations at critical
regularity. They can describe singularities that are not detectable in
standard Lebesgue spaces. For example, considering the geometric setting $%
\func{Ric}\geq -K(m-1)$ with $K>0$ and denoting by $r(x,y)$ the geodesic
distance from $x$ to $y$ in $M$, functions exhibiting critical local
singularities of the form $r(\cdot ,y)^{-\eta }e^{-kr(\cdot ,y)}$, where $%
\eta =(m-\lambda )/p$ and $k=\tfrac{1}{p}(m-1)\sqrt{K}$, belong to Morrey
spaces $\mathcal{M}_{p,\lambda }^{g}$ and $\mathcal{M}_{p,\lambda }$ for $%
p\in \lbrack 1,\infty )$ and $\lambda \in (0,m),$ and for each fixed $y\in M$
(see Lemma \ref{ex-2a}). This contrasts with the Lebesgue spaces $L^{p}$,
for example on Hyperbolic space, which fail to contain some of these
functions (see Lemma \ref{ex-2a} and Remark \ref{rem-e}). Moreover, there is
a class of nondecaying data that belongs to Morrey spaces but not to
Lebesgue spaces $L^{p}$ for $p\neq \infty $. In fact, generally speaking,
Morrey spaces strictly contain both $L^{p}$ and weak-$L^{p}$ spaces,
extending the Lebesgue scale by measuring $L^{p}$-integrability in a
localized, controlled, and scale-sensitive manner, and thus capturing fine
local structure while allowing global non-integrability in $L^{p}$ (see
Remark \ref{rem-e} and Remark \ref{rem-e-2}). In view of the strict
inclusion, it is worth highlighting that smallness conditions in the weaker
norm of Morrey spaces allow one to consider some classes of large data in $%
L^{p}$-spaces.

In the case of Einstein manifolds ($Bu=0$), we obtain a global
well-posedness result with a smallness condition on the initial-data norm
(see Theorem \ref{th-ein}). Additionally, we also discuss a modified
Navier-Stokes system influenced by viscosity (see Section \ref{s.visc}) and
prove a global well-posedness result for more general non-compact manifolds
(see Theorem \ref{th-visc}). By making an analogy with the results in $%
\mathbb{R}^{n}$, note that these results ensure global well-posedness in the
critical ($\mathcal{M}_{m-\lambda ,\lambda }^{g}$) and subcritical ($%
\mathcal{M}_{p,\lambda }^{g}$, with $p>m-\lambda $) cases, due to the
exponential decay and the geometric context of our setting. However, the
formal definition of critical spaces for the Navier-Stokes equations in
general manifolds seems more complex and less clear.

In order to reach our aims, we need to develop dispersive and smoothing
estimates for the heat kernel in our setting. The analysis requires a
treatment based on fundamental solutions (Green functions) and pointwise
estimates of the heat kernel associated with the corresponding Laplacian
operator. One motivation for this approach is that energy-type estimates do
not appear to be directly effective in the framework of Morrey spaces. Here,
we provide estimates for the semigroups generated by the Beltrami, Bochner,
and Hodge Laplacians in the context of Morrey spaces, which are of
independent interest and may have broader applications to other problems in
analysis and partial differential equations.

It is known that the fundamental solution $G(t,x,y)$ can be construct (see,
e.g., \cite{dodziuk}, \cite{Buttig}) so that the solution of the heat
equation with initial data $u_{0}$ is given by
\begin{equation}  \label{green}
u(t,x)=%
\begin{cases}
\displaystyle\int_{M}G(t,x,y)u_{0}(y)\,dy,\hspace{0.2cm} & \text{for}\hspace{%
0.2cm}t>0; \\
u_{0}(x),\hspace{0.2cm} & \text{for}\hspace{0.2cm}t=0.%
\end{cases}%
\end{equation}%
Pointwise estimates for $G(t,x,y)$ are well established in the literature.
For instance, we refer to the results of Davies \cite{davies} (see %
\eqref{g-davies}), Grigor'yan \cite{Grigoryan1994} (see \eqref{es-h3}), and
Buttig and Eichhorn \cite{Buttig} (see \eqref{es-ho}). Moreover, we consider
some volume estimates from differential geometry for manifolds with bounded
Ricci curvature, and additional assumptions such as bounded geometry or
sectional curvature. Indeed, the condition $\func{Ric}\geq -K(m-1)$ yields,
via Bishop's comparison theorem (see \cite{chavel}, \cite{schoen}), a volume
comparison with geodesic balls in spaces of constant curvature, namely $%
|B_{x}(R)|\leq |B^{-K}(R)|$; in this way we obtain a upper bound of the
geodesic ball volume (see Lemma \ref{lem-v}). For lower volume bounds, a
similar result holds under the sectional curvature assumption $\kappa \leq
\delta $ (see \eqref{l-ball} and \eqref{e-vol}).

In deriving bounds for the heat semigroup in Morrey spaces, we adapt some
arguments found in \cite{Kato1992} and \cite{Pierfelice2017}. Extending
these arguments, however, is not straightforward and requires a careful
treatment of the underlying geometric conditions in combination with
Morrey-type norms. For that, we first establish the boundedness of $%
e^{t\Delta _{g}}:\mathcal{M}_{p,\lambda }^{g}\rightarrow L^{\infty }$ (see
Theorem \ref{th-mi}) and $e^{t\Delta _{g}}:\mathcal{M}_{p,\lambda
}^{g}\rightarrow \mathcal{M}_{p,\lambda }^{g}$ (see Theorem \ref{th-mm}), by
analyzing in stages the cases $0<R\leq 1$ and $R>1$ where $R$ denotes the
radius of the ball (see Definition \ref{morrey}). As a consequence of these
results, we obtain the boundedness of $e^{t\Delta _{g}}:\mathcal{M}%
_{p,\lambda }^{g}\rightarrow \mathcal{M}_{q,\lambda }^{g}$, for $1\leq p\leq
q<\infty $ (see Theorem \ref{th-d}). Next, for the heat semigroup associated
with the Bochner Laplacian, and under the assumption of negative Ricci
curvature, we treat the case $e^{t(\overrightarrow{\Delta }+r)}:\mathcal{M}%
_{p,\lambda }^{g}\rightarrow \mathcal{M}_{q,\lambda }^{g}$, for $1\leq p\leq
q<\infty $ (see Theorem \ref{th-cd}). For smoothing (gradient) estimates, we
establish the boundedness of $\nabla e^{t\Delta _{g}}:\mathcal{M}_{p,\lambda
}^{g}\rightarrow \mathcal{M}_{q,\lambda }^{g}$ (see Theorem \ref{es-s}) and $%
\nabla e^{t(\overrightarrow{\Delta }+r)}:\mathcal{M}_{p,\lambda
}^{g}\rightarrow \mathcal{M}_{q,\lambda }^{g}$, for $1\leq p\leq q<\infty $
(see Theorem \ref{th-s-ric}). Moreover, corresponding semigroup estimates
hold in the Morrey space $\mathcal{M}_{p,\lambda }$ (see Figure \ref{fig1}).
Notably, the presence of negative curvature leads to improved large-time
decay compared to the Euclidean case; namely, exponential rather than
polynomial decay. Finally, we address the operators $B$ and $\mathbb{P}$ in (%
\ref{ns-2}), as discussed in Theorem \ref{th-ri}, where we establishes the
boundedness of the Riesz transform $\nabla (-\Delta _{g})^{-1/2}$ in Morrey
spaces.

\begin{figure}[]
\centering  \includegraphics[height=10.2cm]{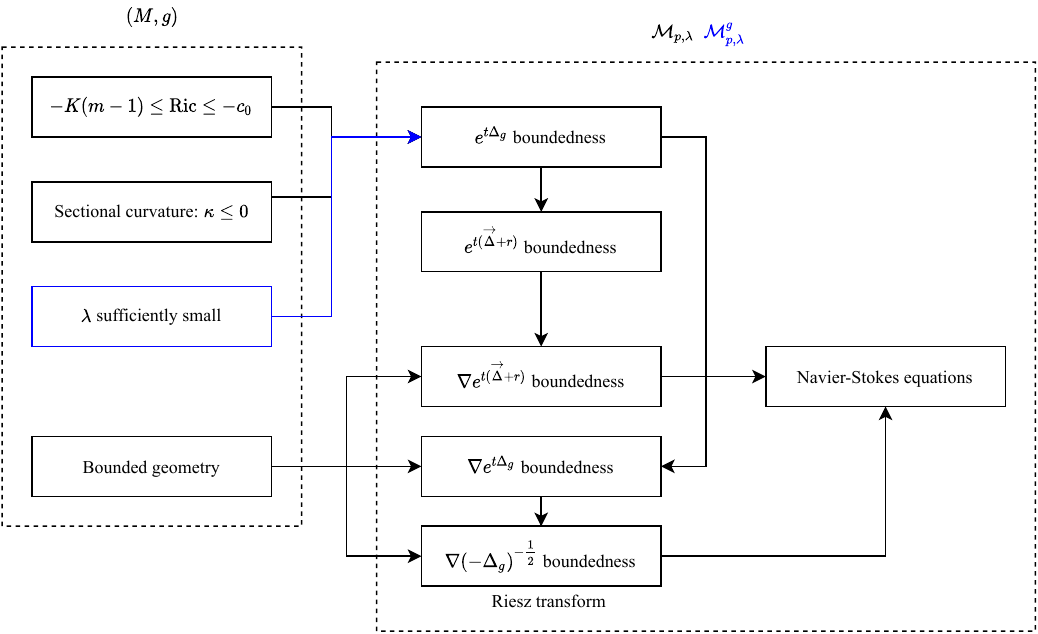} .
\caption{The blue arrow indicates the extra assumption for the case $%
\mathcal{M}^g_{p,\protect\lambda}$}
\label{fig1}
\end{figure}

A few remarks are in order regarding the case of non-negative Ricci
curvature, namely $\func{Ric}\geq 0$. In principle, one may also study
equations \eqref{ns1} under this condition. However, when the inequality is
strict, there is no clear evidence of favorable decay properties for the
associated semigroup. For this reason, we restrict our attention to
non-compact Riemannian manifolds that are Ricci-flat, i.e., $\func{Ric}=0$
(see Section \ref{s7}). In particular, this class includes the celebrated
Calabi-Yau manifolds. On Ricci-flat manifolds, we have nice estimates for
the heat kernel (see \cite{Grigoryan2014}, \cite{yau}), and these are
comparable to those in the Euclidean setting $\mathbb{R}^{m}$. By imposing
additional assumptions, such as bounded geometry or large volume growth (see
figure \ref{fig2}), we obtain the desired estimates in the functional spaces
of interest (see Theorems \ref{th-m0}, \ref{th-s0} and \ref{th-r0}), as well
as global well-posedness results (see Theorem \ref{th-w0}). Important
examples of Ricci-flat manifolds satisfying bounded geometry or exhibiting
large volume growth include the ALE (asymptotically locally Euclidean)
Ricci-flat (see, for example, \cite{NAKAJIMA1990385}). Other examples
include product manifolds of the form $\mathbb{R}^{m}\times C$, where $C$ is
a simply connected compact manifold with $\func{Ric}=0$. In particular, this
class includes the case in which $C$ is a simply connected compact
Calabi-Yau manifold.

\begin{figure}[]
\includegraphics[height=6.5cm]{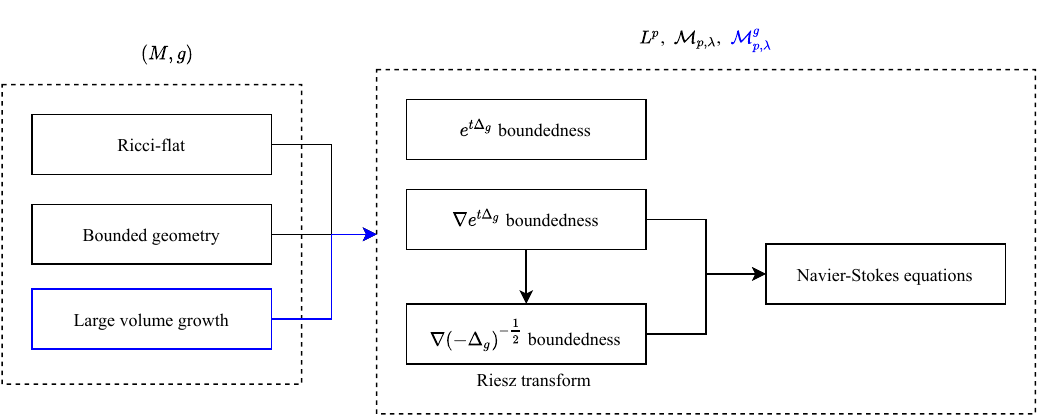}
\caption{The blue arrow indicates the extra assumption for the case $%
\mathcal{M}^{g}_{p,\protect\lambda}$}
\label{fig2}
\end{figure}

In the remainder of the introduction, for the reader convenience, we provide
a detailed overview of the organization of the manuscript.

\subsection{Organization of the Manuscript}

The structure of the present paper is outlined as follows. Section \ref{s1}
is devoted to volume estimates of geodesic balls on manifolds, where we
present upper bounds in Section \ref{s1.1} and lower bounds in Section \ref%
{s1.2}.

In Section \ref{s2}, we address different versions of Morrey spaces on
non-compact manifolds (Definition \ref{morrey}) and provide examples of
useful functions (Lemma \ref{ex-2a}).

Section \ref{s3} is dedicated to establishing semigroup estimates for $%
e^{t\Delta _{g}}$ in the Morrey norm. The main objective is to obtain the
dispersive bound $e^{t\Delta _{g}}:\mathcal{M}_{p,\lambda }^{g}\rightarrow
\mathcal{M}_{q,\lambda }^{g}$, for $1\leq p\leq q<\infty $ (Theorem \ref%
{th-d}).

In Section \ref{s4}, we treat the semigroup $e^{t(\overrightarrow{\Delta }%
+r)}:\mathcal{M}_{p,\lambda }^{g}\rightarrow \mathcal{M}_{q,\lambda }^{g}$.
Dispersive results are presented in Section \ref{s4.1}, while Section \ref%
{s4.2} is devoted to smoothing (gradient) estimates, including bounds for
the operators $\nabla e^{t{\Delta }_{g}}$ (Theorem \ref{es-s}) and $\nabla
e^{t(\overrightarrow{\Delta }+r)}$ (Theorem \ref{th-s-ric}). Section \ref{s5}
addresses the Riesz transform $\nabla (-\Delta _{g})^{-\frac{1}{2}}$, and in
particular the projection operator $\mathbb{P}$. The boundedness of $\nabla
(-\Delta _{g})^{-\frac{1}{2}}$ is established in Theorem \ref{th-ri}.

In Section \ref{s6}, based on Definition \ref{df-m}, we obtain local
well-posedness for a general class of non-compact manifolds. Modifications
of the model with viscosity are discussed in Section \ref{s.visc}. Global
well-posedness results for Einstein manifolds are proved in Section \ref%
{s6.1} (Theorem \ref{th-ein}).

Finally, Section \ref{s7} is devoted to Ricci-flat manifolds. We obtain
estimates for $e^{t\Delta _{g}}$ (Theorem \ref{th-m0}), $\nabla e^{t\Delta
_{g}}$ (Theorem \ref{th-s0}), the Riesz transform (Theorem \ref{th-r0}), and
then establish global well-posedness results (Theorem \ref{th-w0}).

\section{Review of volume estimates}

\label{s1}

\subsection{Upper bounds}

\label{s1.1}

In this work, we suppose that $M$ is a smooth, complete, simply connected,
non-compact Riemannian manifold of dimension $m$, without boundary. Note
that if $M$ satisfies these assumptions and has non-positive sectional
curvature $\kappa \leq 0$, then, by the Cartan-Hadamard theorem, the
exponential map is a diffeomorphism. In particular, the injectivity radius
satisfies $r_{x}=\infty $ for all $x\in M$.

From results in differential geometry (see \cite[Theorem III.3.1]{chavel}),
for manifolds with $r_{x}=\infty $, we can express the Riemannian integral
of an integrable function $f$ on $M$ in spherical coordinates centered at $%
z\in M$, namely
\begin{equation}
\int_{M}f\,dV=\int_{S_{z}}\int_{0}^{\infty }f(\exp _{z}(tv))J(t,v)\,dt\,dv.
\label{eq-int}
\end{equation}%
By Bishop's results (see \cite[Theorem III.4.3]{chavel}) we can estimate $%
J(t,v)$ using the case of hyperbolic space. Indeed, if $\func{Ric}\geq
-K(m-1)$, for some $K>0$, then
\begin{equation*}
J(t,v)\leq \dfrac{\sinh ^{m-1}(\sqrt{-K}t)}{\sqrt{-K}}.
\end{equation*}%
In particular, for $n\in (0,m-1]$ there exists a $C>0$ such that
\begin{equation}
J(t,v)\leq Ct^{n}e^{(m-1)\sqrt{K}t}.  \label{eq-j}
\end{equation}

Moreover, by the Bishop volume comparison theorem (\cite[Theorem III.4.4]%
{chavel} and \cite[Theorem 1.3]{schoen}) if $\func{Ric}(M)\geq -K(m-1)$, $%
K\geq 0$, then for any $x\in M$ and $r>0,$ we have that
\begin{equation}
|B_{x}(R)|\leq |B^{-K}(R)|\hspace{0.3cm}\text{and}\hspace{0.3cm}\frac{%
|B_{x}(R)|}{|B^{-K}(R)|}\,\text{is non-increasing in}\,R,  \label{prop-bk}
\end{equation}%
where $B_{x}(R)$ is the geodesic ball, with center $p$ and radius $R$, and $%
B^{-K}(R)$ is the geodesic ball in the space with constant curvature $-K$.

By \cite[Lemma 2.1]{li} we can compare volumes of geodesic balls with
different radius. It is the subject of the next lemma.

\begin{lem}
\label{l-comp} Let $(M,g)$ be a complete Riemannian manifold of dimension $m$%
. If $\func{Ric}\geq -K(m-1)$ for $K>0$, then there exist constants $C(m)>0$
such that
\begin{equation}
\frac{|B_{x}(R_{2})|}{|B_{x}(R_{1})|}\leq C(c,m)e^{k(R_{2}-R_{1})},
\label{comp}
\end{equation}%
for $0<c<R_{1}<R_{2}$ with $k=\sqrt{K}(m-1)$ and
\begin{equation*}
\frac{|B_{x}(R_{2})|}{|B_{x}(R_{1})|}\leq C(m)\left( \frac{R_{2}}{R_{1}}%
\right) ^{m}e^{k(R_{2}-R_{1})},
\end{equation*}%
for $0<R_{1}<R_{2}$.
\end{lem}

Following the ideas of \cite[Lemma 2.1]{li}, we can estimate more clearly
the volume of the ball in the case $\func{Ric}\geq -K(m-1)$.

\begin{lem}
\label{lem-v} Let $(M,g)$ be a complete, simply connected, non-compact
Riemannian manifold of dimension $m$. If $\func{Ric}\geq -K(m-1)$ for $K>0$,
then for $s>0$ and $0\leq n\leq m$ there exist constants $C=C(K,n,m,s)>0$
such that
\begin{equation*}
|B_{x}(R)|\leq
\begin{cases}
CR^{n}, & \text{for}\hspace{0.3cm}0\leq R\leq s \\
Ce^{(m-1)\sqrt{K}R}, & \text{for}\hspace{0.3cm}R>s%
\end{cases}%
.
\end{equation*}%
Moreover,
\begin{equation}
|B_{x}(R)|\leq CR^{n}e^{(m-1)\sqrt{K}R}.  \label{eq-ball}
\end{equation}
\end{lem}

\noindent {\textbf{Proof.}} For $K>0$, using the Bishop volume comparison
\begin{equation*}
|B_{x}(R)|\leq |B^{-K}(R)|=\int_{0}^{R}\sinh ^{m-1}(\sqrt{K}\rho )\,d\rho =%
\dfrac{1}{\sqrt{K}}\int_{0}^{\sqrt{K}R}\sinh ^{m-1}(\rho )\,d\rho .
\end{equation*}%
If $0\leq \sqrt{K}R\leq \sqrt{K}s$, taking $C>0$ such that $\sinh (\rho
)\leq C\rho ^{\frac{1}{\alpha }}$ with $\alpha =\frac{m-1}{n-1}$, then
\begin{equation*}
|B_{x}(R)|\leq \dfrac{C}{\sqrt{K}}\int_{0}^{\sqrt{K}R}\rho ^{\frac{m-1}{%
\alpha }}\,d\rho \leq CR^{n}.
\end{equation*}%
Now, note that for every $R\geq 0$,
\begin{equation*}
|B_{x}(R)|\leq \dfrac{1}{\sqrt{K}}\int_{0}^{\sqrt{K}R}\sinh ^{m-1}(\rho
)\,d\rho \leq \dfrac{1}{\sqrt{K}}\int_{0}^{\sqrt{K}R}e^{(m-1)\rho }\,d\rho
\leq Ce^{(m-1)\sqrt{K}R}.
\end{equation*}%
Taken together, the above estimates establish polynomial growth for small $r$
and exponential growth for arbitrary $R$, as desired.\fin

\begin{rem}
\ \label{rem-h-esp-1}

\begin{enumerate}
\item[$(i)$] In the hyperbolic space $\mathbb{H}_{\kappa }^{m}$ with
constant sectional curvature $-\kappa $ (where $\kappa >0$), the volume of
geodesic balls grows exponentially. More precisely, for all $R\geq R_{0}>0$,
there exist positive constants $\widetilde{C}$ and $C$ such that
\begin{equation*}
\widetilde{C}e^{(m-1)\sqrt{K}R}\leq |B_{x}(R)|\leq Ce^{(m-1)\sqrt{K}R}.
\end{equation*}

\item[$(ii)$] We also note that alternative upper bounds for the volume of
geodesic balls in manifolds equipped with a weighted measure are available
in the literature; see, for instance, the work of Munteanu and Wang \cite%
{MUNTEANU}.
\end{enumerate}
\end{rem}

\subsection{Lower bounds}

\label{s1.2}

From \cite[Theorem III.4.1 and Theorem III.4.2]{chavel} and the Hadamard
theorem, we have a lower estimate of the geodesic ball for manifolds with
non-positive sectional curvature. Indeed, assume that the $M$ is complete,
simply connected with sectional curvature $\kappa \leq 0$. Considering %
\eqref{eq-int}, then for every $p\in M$, $t>0$ and $v\in S_{p},$ we have
\begin{equation*}
J(t,v)\geq Ct^{m-1}.
\end{equation*}%
Furthermore, if $\kappa \leq \delta $ then
\begin{equation}
|B_{p}(R)|\geq |B^{\delta }(R)|,  \label{l-ball}
\end{equation}%
where $B^{\delta }(R)$ is the geodesic ball in the space with constant
curvature $\delta $. Therefore, for $\kappa \leq 0$, it follows that $%
|B_{p}(R)|\geq c\,R^{m}$. More generally, it is very common to suppose the
hypothesis $\func{Ric}\geq -K(m-1)$ with polynomial growth
\begin{equation}
|B_{x}(R)|\geq \alpha R^{n},\hspace{0.2cm}\forall \,R>0,\hspace{0.2cm}%
\forall \,x\in M,  \label{n-vol}
\end{equation}%
for $n\in (0,m]$ and $\alpha =\alpha (n)>0$ constants. When $n=m$, the
condition \eqref{n-vol} is known in the literature as large volume growth
(see \cite{xia1999open} and \cite{Li2011}).

Moreover, by \cite[Corollary 1]{qing} we have exponential growth for the
volume of the geodesic ball. Indeed, for a complete non-positive curved $m$%
-dimensional Riemannian manifold $(m\geq 2)$, assume that the $\func{Ric}%
\leq -c_{0}$ for some constant $c_{0}>0$, then for every $R_{0}>0$ there is
a constant $C=C(c_{0},R_{0})>0$ such that
\begin{equation}
|B_{p}(R)|\geq Ce^{\sqrt{c_{0}}R}  \label{e-vol}
\end{equation}%
holds for $R\geq R_{0}$. For further details on lower bounds, see the work
of Schoen \cite{schoen}.


\section{Morrey spaces on manifolds}

\label{s2}

The study of partial differential equations in Morrey spaces has been
extensively developed in the Euclidean setting $\mathbb{R}^{n}$. Concerning
the Navier-Stokes equations, Giga and Miyakawa \cite{Giga1}, Kato \cite%
{Kato1992}, and Taylor \cite{Taylor1992} established well-posedness results
in these spaces defined on $\mathbb{R}^{n}$. A detailed analysis of Morrey
spaces in the Euclidean framework can be found in \cite{ad}. The extension
of Morrey spaces to more general manifolds is also discussed in \cite{ad};
indeed, the most natural generalization involves using the volume of
geodesic balls. In this work, we present some possible approaches to this
generalization.

\begin{df}
\label{morrey}Let $(M,g)$ be a complete $m$-dimensional Riemannian manifold
without boundary. Recall that the norm of a vector $f\in \Gamma (TM)$ is
defined by $|f|=\sqrt{g(f,f)}.$ For $p\in \lbrack 1,\infty )$ and $\lambda
\in \lbrack 0,m)$, we define
\begin{align*}
\mathcal{M}_{p,\lambda }^{A}(\Gamma (TM))=\bigg\{& f\in L_{loc}^{p}(\Gamma
(TM))\,:\, \\
& \Vert f\Vert _{A,p,\lambda }=\sup_{x_{0}\in M,R>0}\left( A(x_{0},R)^{-%
\frac{\lambda }{m}}\int_{B_{x_{0}}(R)}|f|^{p}\right) ^{\frac{1}{p}}<\infty %
\bigg\},
\end{align*}%
where $B_{x_{0}}(R)$ denotes the geodesic ball centered at $x_{0}$ with
radius $R$, and $A(x_{0},R)$ is a prescribed quantity representing an
estimate of the volume of a ball of radius $R$. Similarly, one can define $%
\mathcal{M}_{p,\lambda }^{A}(M)$ for scalar functions. Typical choices of $%
A(x_{0},R)$ give rise to the following spaces:

\begin{itemize}
\item $(\mathcal{M}^g_{p,\lambda}, \|\cdot\|_{g,p,\lambda})$: if $A(x_0,R) =
|B_{x_0}(R)|$, where $B_{x_0}(R)$ is the geodesic ball with center $x_0$ and
radius $R$;

\item $(\mathcal{M}^K_{p,\lambda}, \|\cdot\|_{K,p,\lambda})$: if $A(x_0,R) =
|B^{K}(R)|$, where $B^{K}(R)$ is some ball with radius $R$ of space form
with curvature $K$;

\item $(\mathcal{M}_{p,\lambda}, \|\cdot\|_{p,\lambda})$: if $A(x_0,R) = R^m$%
;

\item $(\mathcal{M}^{e}_{p,\lambda}, \|\cdot\|_{e,p,\lambda})$: if $A(x_0,R)
= e^{\sqrt{K}(m-1) r}$.
\end{itemize}
\end{df}

Note that when $\func{Ric}(M)\geq -K(m-1)$, $K\geq 0$, we have $\mathcal{M}%
_{p,\lambda }^{g}\subset \mathcal{M}_{p,\lambda }^{K}\subset \mathcal{M}%
_{p,\lambda }^{e}.$ If, in addition, $M$ has large-volume growth (which
includes the case $\kappa \leq 0$), then $\mathcal{M}_{p,\lambda }\subset
\mathcal{M}_{p,\lambda }^{g}$. If $K=0$, we obtain the equivalence $\mathcal{%
M}_{p,\lambda }^{0}=\mathcal{M}_{p,\lambda }$. Moreover, for $M=\mathbb{R}%
^{m}$, $\mathcal{M}_{p,\lambda }(\mathbb{R}^{m})=\mathcal{M}_{p,\lambda
}^{0}(\mathbb{R}^{m})=\mathcal{M}_{p,\lambda }^{g}(\mathbb{R}^{m}).$

For simplicity, this paper is devoted to the study of initial data in the
spaces $\mathcal{M}_{p,\lambda }^{g}$ and $\mathcal{M}_{p,\lambda }^{e}$. As
in the Euclidean case, these spaces are Banach space. Moreover, they satisfy
the following inclusion property, stated in the next lemma.

\begin{lem}
\label{prop-inck} Let $p,q\in \lbrack 1,\infty )$ and $\lambda ,\mu \in
\lbrack 0,m)$ be such that $\frac{m-\lambda }{p}=\frac{m-\mu }{q}$ and $%
p\leq q$. Then, we have the continuous inclusion $\mathcal{M}_{q,\mu
}^{g}\subset \mathcal{M}_{p,\lambda }^{g}$.
\end{lem}

\noindent {\textbf{Proof.}} Using H\"{o}lder's inequality in $L^{p},$ with $%
p^{-1}=r^{-1}+q^{-1}$, together with \eqref{prop-bk}, we can estimate
\begin{align*}
\Vert f\Vert _{p;x_{0},R}& \leq \Vert 1\Vert _{r;x_{0},R}\Vert f\Vert
_{q;x_{0},R}=|B_{x_{0}}(R)|^{\frac{1}{r}}\Vert f\Vert _{q;x_{0},R} \\
& =(|B_{x_{0}}(R)|^{\frac{1}{m}})^{\frac{m}{p}-\frac{m}{q}}\Vert f\Vert
_{q;x_{0},R},
\end{align*}%
where $\Vert f\Vert _{p;x_{0},R}=\left( \int_{B_{x_{0}}(R)}|f|^{p}\right)
^{1/p}$. Then, it follows that
\begin{align*}
(|B_{x_{0}}(R)|^{\frac{1}{m}})^{-\frac{\lambda }{p}}\Vert f\Vert
_{p;x_{0},R}& \leq (|B_{x_{0}}(R)|^{\frac{1}{m}})^{\frac{m-\lambda }{p}-%
\frac{m}{q}}\Vert f\Vert _{q;x_{0},R} \\
& \leq (|B_{x_{0}}(R)|^{\frac{1}{m}})^{-\frac{\mu }{q}}\Vert f\Vert
_{q;x_{0},R},
\end{align*}%
which yields $\Vert f\Vert _{g,p,\lambda }\leq \Vert f\Vert _{g,q,\mu }.$\fin

It is not difficult to establish H\"{o}lder's inequality in the framework of
Morrey spaces on manifolds.

\begin{lem}[H\"{o}lder's inequality]
\label{holder}Let $p,q,r\in \lbrack 1,\infty )$ and $\lambda ,\mu ,\tau \in
\lbrack 0,m)$ satisfy $\frac{1}{r}=\frac{1}{p}+\frac{1}{q}$ and $\frac{\tau
}{r}=\frac{\lambda }{p}+\frac{\mu }{q}$. If $f\in \mathcal{M}_{p,\lambda
}^{g}$ and $h\in \mathcal{M}_{q,\mu }^{g}$, then $fh\in \mathcal{M}_{r,\tau
}^{g}$ and
\begin{equation*}
\Vert fh\Vert _{g,r,\tau }\leq \Vert f\Vert _{g,p,\lambda }\Vert h\Vert
_{g,q,\mu }.
\end{equation*}
\end{lem}

For manifolds with bounded Ricci curvature, one can adapt the classical
examples of functions of the form $1/|x|^{\alpha }$ in $\mathbb{R}^{m}$ to
construct functions belonging to Morrey spaces on such manifolds. These
examples will play an important role in the semigroup estimates from $%
\mathcal{M}_{p,\lambda }^{g}$ to $\mathcal{M}_{p,\lambda }^{g}$ (see Theorem %
\ref{th-mm}). We therefore state the following lemma.

\begin{lem}
\label{ex-2a} Let $(M,g)$ be a complete, simply connected, non-compact,
boundaryless Riemannian manifold of dimension $m$ whose Ricci curvature
satisfies $\func{Ric}\geq -K(m-1)$ for some $K>0$. Moreover, assume that $M$
has large volume growth (e.g., when $\kappa \leq 0$) and let $p\in \lbrack
1,\infty )$.

\begin{enumerate}
\item[$(i)$] We have that the function $r(\cdot ,y)^{-l}e^{-kr(\cdot ,y)}$
belongs to $\mathcal{M}_{p,\lambda }^{g}$ and $\mathcal{M}_{p,\lambda }$ for
$0\leq l\leq \frac{m-\lambda }{p}$, $k\geq \tfrac{1}{p}(m-1)\sqrt{K}$, and $%
\lambda \in (0,m)$. Moreover,
\begin{equation*}
\Vert r(\cdot ,y)^{-l}e^{-kr(\cdot ,y)}\Vert _{g,p,\lambda },\Vert r(\cdot
,y)^{-l}e^{-kr(\cdot ,y)}\Vert _{p,\lambda }\leq C,
\end{equation*}%
for some constant $C>0$ that does not depend on $y\in M$.

\item[$(ii)$] In the case $\lambda =0$, the function $r(\cdot
,y)^{-l}e^{-kr(\cdot ,y)}$ belongs to $L^{p}$ provided that $0\leq l<\frac{m%
}{p}$ and $k>\tfrac{1}{p}(m-1)\sqrt{K}$.
\end{enumerate}
\end{lem}

\noindent {\textbf{Proof.}} We restrict ourselves to proving the lemma for
the spaces $\mathcal{M}_{p,\lambda }^{g}$ and ($\lambda =0$) $L^{p}$, as the
corresponding result for the space $\mathcal{M}_{p,\lambda }$ follows by
analogous arguments. To do this, we need to verify that, for every $z\in M$,
\begin{equation*}
\int_{B_{z}(R)}r(x,y)^{-lp}e^{-kpr(x,y)}\,dx\leq C|B_{z}(R)|^{\frac{\lambda
}{m}}.
\end{equation*}%
First, consider the case $r(z,y)\geq 2R$. For any $x\in B_{z}(R)$,
\begin{equation*}
2R\leq r(z,y)\leq R+r(x,y)\Rightarrow r(x,y)\geq R\geq r(x,z).
\end{equation*}%
We then apply the spherical coordinates \eqref{eq-int}, together with
estimate \eqref{eq-j}, to obtain
\begin{align*}
\int_{B_{z}(R)}r(x,y)^{-lp}e^{-kpr(x,y)}\,dx& \leq
\int_{B_{z}(R)}r(x,z)^{-lp}e^{-kpr(x,z)}\,dx \\
& =\int_{S_{z}}\int_{0}^{R}|tv|^{-lp}e^{-kp|tv|}J(t,v)\,dt\,dv \\
& \leq C\int_{0}^{R}t^{-lp}t^{lp+\lambda -1}e^{-kpt+(m-1)\sqrt{K}t}\,dt \\
& \leq C\int_{0}^{R}t^{\lambda -1}\,dt \\
& \leq CR^{\lambda }\leq C|B_{z}(R)|^{\frac{\lambda }{m}}.
\end{align*}%
For the case $\lambda =0$, taking $kp>(m-1)\sqrt{K}$ and $lp<m$, we get
\begin{align*}
\int_{S_{z}}\int_{0}^{R}|tv|^{-lp}e^{-kp|tv|}J(t,v)\,dt\,dv& \leq
C\int_{0}^{R}t^{-lp}t^{m-1}e^{-kpt+(m-1)\sqrt{K}t}\,dt \\
& \leq C\int_{0}^{R}t^{-lp+m-1}e^{-kpt+(m-1)\sqrt{K}t}\,dt\leq C.
\end{align*}%
In the case $r(z,y)<2R$, we have $B_{z}(R)\subset B_{y}(3R),$ and hence
\begin{align*}
\int_{B_{z}(R)}r(x,y)^{-lp}e^{-kpr(x,y)}\,dx& \leq
\int_{B_{y}(3R)}r(x,y)^{-lp}e^{-kpr(x,y)}\,dx \\
& =\int_{S_{y}}\int_{0}^{3R}|tv|^{-lp}e^{-kp|tv|}J(t,v)\,dt\,dv \\
& \leq C\int_{0}^{3R}t^{-lp}t^{lp+\lambda -1}e^{-kpt+(m-1)\sqrt{K}t}\,dt \\
& \leq C\int_{0}^{3R}t^{\lambda -1}\,dt \\
& \leq CR^{\lambda }\leq C|B_{z}(R)|^{\frac{\lambda }{m}}.
\end{align*}%
Again, for $\lambda =0$, it follows that
\begin{align*}
\int_{S_{y}}\int_{0}^{3R}|tv|^{-lp}e^{-kp|tv|}J(t,v)\,dt\,dv& \leq
C\int_{0}^{3R}t^{-lp}t^{m-1}e^{-kpt+(m-1)\sqrt{K}t}\,dt \\
& \leq C\int_{0}^{3R}t^{-lp+m-1}e^{-kpt+(m-1)\sqrt{K}t}\,dt\leq C.
\end{align*}%
Combining the above estimates yields the desired result.\fin

\begin{rem}
\label{rem-e} \hspace{0.1cm}

\begin{enumerate}
\item[$(i)$] Note that, just as in Lemma \ref{ex-2a}, we can establish
similar properties for the functions $r(\cdot ,y)^{-l}e^{-kr(\cdot ,y)^{2}}$.

\item[$(ii)$] Taking $\mu =0$ in Lemma \ref{prop-inck}, we obtain the
inclusion $L^{q}\subset \mathcal{M}_{p,\lambda }^{g}$ for $p<q$ and $\frac{%
m-\lambda }{p}=\frac{m}{q}$. Moreover, if we additionally assume $\kappa
\leq 0$, then this inclusion is proper. To see this, observe that by Lemma %
\ref{ex-2a}, the function $r(\cdot ,y)^{-\eta }e^{-kr(\cdot ,y)}\in \mathcal{%
M}_{p,\lambda }^{g}$ with $\eta =\frac{m-\lambda }{p}$. However, for $R>1,$
we have
\begin{align*}
\int_{B_{y}(R)}r(x,y)^{-\eta q}e^{-kqr(x,y)}\,dx&
=\int_{S_{z}}\int_{0}^{R}|tv|^{-\eta q}e^{-kq|tv|}J(t,v)\,dt\,dv \\
& \geq \int_{S_{z}}\int_{0}^{1}t^{-\eta q}e^{-kqt}J(t,v)\,dt\,dv \\
& \geq C\int_{0}^{1}t^{-\eta q+m-1}\,dt,
\end{align*}%
where the last integral diverges when $\eta =\frac{m-\lambda }{p}\geq \frac{m%
}{q}$. Hence the $L^{q}$-norm is not finite, and thus $\mathcal{M}%
_{p,\lambda }^{g}\nsubseteq L^{q}$ for $p<q$ and $\frac{m-\lambda }{p}\geq
\frac{m}{q}$. Due to the above strict inclusion $L^{q}\varsubsetneq \mathcal{%
M}_{p,\lambda }^{g}$, note that imposing smallness conditions in the weaker
norm of Morrey spaces allows the treatment of some classes of large data in $%
L^{p}$-spaces.

\item[$(iii)$] For the hyperbolic space, we have for $kp=\sqrt{K}(m-1)$ and $%
0\leq lp\leq 1$ that $r(\cdot ,y)^{-l}e^{-kr(\cdot ,y)}\in \mathcal{M}%
_{p,\lambda }^{g}(\mathbb{H}_{\kappa }^{m})$, but $r(\cdot
,y)^{-l}e^{-kr(\cdot ,y)}\notin L^{p}(\mathbb{H}_{\kappa }^{m})$. Indeed,
the first claim follows from Lemma \ref{ex-2a}. For the second claim, using
Remark \ref{rem-h-esp-1}), observe that for $R>1$
\begin{align*}
\int_{B_{y}(R)}r(x,y)^{-lp}e^{-kpr(x,y)}\,dx&
=\int_{S_{z}}\int_{0}^{R}|tv|^{-lp}e^{-kp|tv|}J(t,v)\,dt\,dv \\
& \geq \int_{S_{z}}\int_{1}^{R}t^{-lp}e^{-kpt}J(t,v)\,dt\,dv \\
& \geq C\int_{1}^{R}t^{-lp}e^{-kpt+\sqrt{K}(m-1)t}\,dt \\
& =C\int_{1}^{R}t^{-lp}\,dt.
\end{align*}%
Since $\mathbb{H}_{\kappa }^{m}=\cup _{R>0}B_{y}(R)$, this shows that the
function $r(\cdot ,y)^{-l}e^{-kr(\cdot ,y)}$ does not belong to $L^{p}(%
\mathbb{H}_{\kappa }^{m})$. Consequently, $M_{p,\lambda }^{g}(\mathbb{H}%
_{\kappa }^{m})\nsubseteq L^{p}(\mathbb{H}_{\kappa }^{m})$. Moreover, for
fixed $p$ and $\lambda $, consider indexes $p_{0}\leq p$ and $\lambda
_{0}\geq \lambda $ such that $\frac{m-\lambda }{p}=\frac{m-\lambda _{0}}{%
p_{0}}$. By the inclusion property, we then obtain that $\mathcal{M}%
_{p,\lambda }^{g}(\mathbb{H}_{\kappa }^{m})\subset \mathcal{M}%
_{p_{0},\lambda _{0}}^{g}(\mathbb{H}_{\kappa }^{m})\nsubseteq L^{p_{0}}(%
\mathbb{H}_{\kappa }^{m})$ for $p\geq p_{0}$.
\end{enumerate}
\end{rem}

In the following remark, we consider both pure power-law data and
nondecaying data.

\begin{rem}
\label{rem-e-2} \hspace{0.1cm}

\begin{enumerate}
\item[$(i)$] (Power-law distance data) For manifolds with $\func{Ric}\geq 0$%
, we have that both the inclusion property (Lemma \ref{prop-inck}) and H\"{o}%
lder's inequality (Lemma \ref{holder}) are also valid for the space $%
\mathcal{M}_{p,\lambda }$. Moreover, Lemma \ref{ex-2a} can be adapt to show
that the power-law distance function $r(\cdot ,y)^{-\eta }$ belongs to $%
\mathcal{M}_{p,\lambda }$ with $\eta =\frac{m-\lambda }{p}$. If, in
addition, $M$ has large volume growth, then this function does not belong to
$L^{q}(M)$ for $p<q$ and $\frac{m-\lambda }{p}\leq \frac{m}{q}$. Indeed,
assume $\func{Ric}\geq 0$. First considering the case $r(z,y)\geq 2R$, we
note that
\begin{align*}
\int_{B_{z}(R)}r(x,y)^{-\eta p}\,dx& \leq \int_{B_{z}(R)}r(x,z)^{-\eta
p}\,dx=\int_{S_{z}}\int_{0}^{R}|tv|^{-\eta p}J(t,v)\,dt\,dv \\
& \leq C\int_{0}^{R}t^{-\eta p}t^{m-1}\,dt\leq C\int_{0}^{R}t^{\lambda
-1}\,dt\leq CR^{\lambda }.
\end{align*}%
For the remaining case $r(z,y)<2R$, we have%
\begin{align*}
\int_{B_{z}(R)}r(x,y)^{-\eta p}\,dx& \leq \int_{B_{y}(3R)}r(x,y)^{-\eta
p}\,dx=\int_{S_{y}}\int_{0}^{3R}|tv|^{-\eta p}J(t,v)\,dt\,dv \\
& \leq C\int_{0}^{3R}t^{-\eta p}t^{\eta p+\lambda -1}\,dt\leq
C\int_{0}^{3R}t^{\lambda -1}\,dt\leq CR^{\lambda }.
\end{align*}%
Thus $r(\cdot ,y)^{-\eta }\in \mathcal{M}_{p,\lambda }$. On the other hand,
let $1\leq q<\infty $ and denote $\rho (t)=|B_{y}(t)|.$ Then
\begin{align*}
\int_{B_{y}(R)}r(x,y)^{-\eta q}\,dx& =\int_{S_{y}}\int_{0}^{R}|vt|^{-\eta
q}J(t,v)\,dt\,dv \\
& =\int_{0}^{R}t^{-\eta q}\int_{S_{y}}J(t,v)\,dv\,dt \\
& =\int_{0}^{R}t^{-\eta q}\frac{d}{dt}\rho (t)\,dt \\
& =R^{-\eta q}\rho (R)-\lim_{R\rightarrow 0}R^{-\eta q}\rho (R)+\eta
q\int_{0}^{R}t^{-\eta q-1}\rho (t)\,dt \\
& \geq CR^{-\eta q+m}-C\lim_{R\rightarrow 0}R^{-\eta q+m}+C\eta
q\int_{0}^{R}t^{-\eta q+m-1}\,dt.
\end{align*}%
Since the last integral diverges whenever $\eta =\frac{m-\lambda }{p}\geq
\frac{m}{q}$, we conclude that $r(\cdot ,y)^{-\eta }$ does not belong to $%
L^{q}(M)$; in particular, when $\frac{m-\lambda }{p}=\frac{m}{q}$, the
inclusion $L^{q}\subset \mathcal{M}_{p,\lambda }$ is proper.

\item[$(ii)$] (Nondecaying data) Let $M$ be a Riemannian manifold such that
either $(i)$ $\func{Ric}\geq -K(m-1),$ for some $K>0,$ and $\kappa \leq 0$;
or $(ii)$ $\func{Ric}\geq 0$ and $M$ has large volume growth. For $x_{0}\in
M $ and $n\in \mathbb{N}$, consider the functions
\begin{equation*}
\phi _{n}(x,x_{0})=%
\begin{cases}
\sin (2\pi nr(x,x_{0})), & \text{if }r(x,x_{0})\leq 1\text{;} \\
0, & \text{otherwise,}%
\end{cases}%
\end{equation*}%
where $r(x,x_{0})$ denotes the distance between $x$ and $x_{0}$ in $M$. We
have that the function $\phi (x)=\sum_{n=1}^{\infty }\phi _{n}(x,a_{n})$
belongs to the Morrey spaces $\mathcal{M}_{p,\lambda }^{g}(M)$ and $\mathcal{%
M}_{p,\lambda }(M)$, provided that $\{a_{n}\}\subset M$ is a sequence of
points such that $r(a_{n+1},a_{n})>2$ and $r(a_{n},x_{0})$ grows
sufficiently fast. Moreover, $\phi \notin L^{p}(M)$ for $p\neq \infty $, $%
\phi \in L^{\infty }(M),$ and it does not decay as $r(x,x_{0})\rightarrow
\infty ;$ that is, $\phi $ is a nondecaying function.
\end{enumerate}
\end{rem}


\section{Heat equation on non-compact manifolds}

\label{s3}

Estimates of the heat kernel on non-compact manifolds have been a central
topic of research. Among the most important results are those of Li and Yau
\cite{yau} for manifolds with bounded Ricci curvature. Improvements to these
results can be found in more recent works, such as Davies \cite{davies} and
Li and Zhen \cite{li}. These estimates have better decay with an exponential
term depending on the bottom of the Laplacian. Let us denote by $\lambda
_{1}\geq 0$ the bottom of the $L^{2}$-spectrum of $-\Delta _{g}\,$, defined
as
\begin{equation*}
\lambda _{1}=\inf_{f\in \mathcal{C}_{0}^{\infty }(M)}\dfrac{\int_{M}|\nabla
f|^{2}}{\int_{M}f^{2}}.
\end{equation*}

If $M$ has non-negative Ricci curvature, then $\lambda _{1}=0$. Furthermore,
for a complete manifold with $\func{Ric}\geq -K(m-1)$ for some $K>0$, we
have $\lambda _{1}\leq \frac{(m-1)^{2}K}{4}$ (see \cite{Cheng1975}). In the
case of Ricci curvature bounded from above, for a complete, simply
connected, $m$-dimensional Riemannian manifold $M$ with $\kappa \leq -\delta
\leq 0$ and $\func{Ric}\leq -c_{0}^{2}<0$, we have that (by \cite[Theorem 1]%
{setti})
\begin{equation*}
\lambda _{1}\geq \dfrac{1}{4}\left( c_{0}^{2}+(m-1)(m-2)\delta \right) .
\end{equation*}

The result of \cite[Theorem 3]{davies} currently provides the sharpest known
estimate of the heat kernel. More precisely, let $M$ be a complete
Riemannian manifold with $\func{Ric}\geq -K(m-1)$, for some $K\geq 0$. If $%
x,y\in M$ and $t>0$, then
\begin{equation}
0\leq G(t,x,y)\leq C(m,K)|B_{x}(s)|^{-\frac{1}{2}}|B_{y}(s)|^{-\frac{1}{2}%
}e^{-\lambda _{1}t-\frac{r^{2}}{4t}},  \label{g-davies}
\end{equation}%
where $r=r(x,y)$ denotes the Riemannian distance and $s=\min \left\{ 1,\sqrt{%
t},\frac{t}{r}\right\} .$ Moreover, under the additional assumption of
volume growth \eqref{n-vol}, the estimate can be simplified to
\begin{equation}
G(t,x,y)\leq C(m,K)s^{-n}e^{-\lambda _{1}t-\frac{r^{2}}{4t}}\leq C(m,K)\max {%
\left\{ 1,t^{-\frac{n}{2}},t^{-n}r^{n}\right\} }e^{-\lambda _{1}t-\frac{r^{2}%
}{4t}}.  \label{es-h1}
\end{equation}%
In particular, when $\kappa \leq 0$, we obtain \eqref{es-h1} with $n=m$. In
the case where $-K(m-1)\leq \func{Ric}\leq -c_{0}$ and $\kappa \leq 0$, it
follows from \eqref{e-vol} that
\begin{equation}
G(t,x,y)\leq Ce^{-\sqrt{c_{0}}s}e^{-\lambda _{1}t-\frac{r^{2}}{4t}}\leq
Ce^{-\lambda _{1}t-\frac{r^{2}}{4t}},  \label{es-h2}
\end{equation}%
for $t\geq t_{0}>0$ and some constant $C=C(m,t_{0},c_{0},K).$

From Grigor'yan \cite{Grigoryan1999},\cite{Grigoryan1994}, we have good
estimates for Cartan-Hadamard manifolds. Since the Riemannian manifold is
complete, non-compact, and simply connected with $\kappa \leq 0$
(Cartan-Hadamard manifold), then the heat kernel $G(t,x,y)$ satisfies the
estimate
\begin{equation}
G(t,x,y)\leq C\max \{1,t^{-\frac{m}{2}}\}e^{-\lambda _{1}t-\frac{r^{2}}{2Dt}%
},  \label{es-h3}
\end{equation}
for some $D>2$ and $C=C(m,D).$ For the case $-K(m-1)\leq \func{Ric}\leq
-c_{0}$ and $\kappa \leq 0$, we can combine \eqref{es-h2} with $t_{0}=1$ and %
\eqref{es-h3} for $0<t<1$, which yields
\begin{equation}
G(t,x,y)\leq C(m,c_{0},K)\max \{1,t^{-\frac{m}{2}}\}e^{-\lambda _{1}t-\frac{%
r^{2}}{4t}}.  \label{es-h4}
\end{equation}%
Also, if $\kappa \leq -\kappa _{0}$, then
\begin{equation*}
G(t,x,y)\leq C\max \{1,t^{-\frac{m}{2}}\}e^{-\frac{(m-1)^{2}}{4}\kappa _{0}t-%
\frac{r^{2}}{2Dt}}.
\end{equation*}

For the hyperbolic space $\mathbb{H}_{\kappa }^{m}$ with constant negative
curvature $-\kappa $, we have
\begin{equation}
G_{\kappa }(t,x,y)\asymp t^{-\frac{m}{2}}(1+r+t)^{\frac{m-3}{2}}(1+r)e^{-%
\frac{(m-1)^{2}}{4}\kappa t-\frac{r^{2}}{4t}-\frac{m-1}{2}\sqrt{\kappa }r}.
\label{es-hp}
\end{equation}%
By Debiard, Gaveau and Mazet \cite{Debiard1976}, if $M$ is simply connected
and $\kappa \leq -\kappa _{0}\leq 0$, then
\begin{equation}
G(t,x,y)\leq G_{\kappa _{0}}(t,x,y),\text{ for all }x,y\in M\text{.}
\label{sg-k0}
\end{equation}%
Furthermore, for
\begin{equation*}
\partial _{t}u^{\flat }+\Delta _{H}u^{\flat }=0\text{, }u\in \Gamma (TM),
\end{equation*}%
we have that (see \cite[Lemme 2.3]{Lohou})
\begin{equation}
|e^{t\Delta _{H}}u^{\flat }|\leq e^{t\Delta _{g}}|u|.  \label{hb}
\end{equation}%
Here, for a vector field $X\in \Gamma (TM)$, we denote by $X^{\flat }$ the $%
1 $-form defined by $X^{\flat }(Y)=g(X,Y)$ for all $Y\in \Gamma (TM)$. We
also use the convention $g(X^{\flat },Y^{\flat })=g(X,Y)$. Consequently,
when deriving dispersive estimates, it suffices to consider solutions for
functions $u:\mathbb{R}\times M\rightarrow \mathbb{R}$. For simplicity, in
the following results we will consider the estimates \eqref{es-h4} and %
\eqref{sg-k0}. To obtain the heat kernel estimate from $\mathcal{M}%
_{p,\lambda }^{g}$ to $\mathcal{M}_{q,\lambda }^{g}$ for $1\leq p\leq
q<\infty $, we first develop the $\mathcal{M}_{p,\lambda }^{g}\rightarrow
L^{\infty }$ estimate, using the pushforward measure, adapting an idea found
in \cite{Kato1992} to the setting of manifolds.

\begin{thr}[Dispersive: $\mathcal{M}_{p,\protect\lambda }^{g}\rightarrow
L^{\infty }$]
\label{th-mi} Let $(M,g)$ be a complete, simply connected, non-compact, $m$%
-dimensional Riemannian manifold without boundary. Let $p\in \lbrack
1,\infty )$ and $\lambda \in \lbrack 0,m)$ and denote by $u_{0}$ the initial
data for the heat equation.

\begin{enumerate}
\item[$(i)$] Assume that $-K(m-1)\leq \func{Ric}\leq -c_{0}<0$ and $\kappa
\leq 0,$ and let $u_{0}\in \mathcal{M}_{p,\lambda }^{g}(\Gamma (TM))$. Then,
for all $t>0$, the following estimate holds:
\begin{equation}
\Vert e^{t\Delta _{g}}u_{0}\Vert _{\infty }\leq Cd(t)^{-\frac{m}{2p}}t^{%
\frac{\lambda }{2p}}e^{-\frac{1}{p}\left[ \lambda _{1}-k\frac{\lambda }{m}%
\gamma \right] t}\Vert u_{0}\Vert _{g,p,\lambda },  \label{Thm-Disp-aux-1}
\end{equation}%
where $d(t)=\min \{1,t\}$, $\gamma =k\frac{\lambda }{m}\left( \frac{1}{4}%
-c\right) ^{-1}$, $k=\sqrt{K}(m-1)$, $0<c<\frac{1}{4},$ and $C=C(m,c_{0},K)$.

\item[$(ii)$] Assume that $\kappa \leq 0$ and let $u_{0}\in \mathcal{M}%
_{p,\lambda }(\Gamma (TM))$. Then, for all $t>0$, the following estimate
holds:
\begin{equation}
\Vert e^{t\Delta _{g}}u_{0}\Vert _{\infty }\leq Cd(t)^{-\frac{m}{2p}}t^{%
\frac{\lambda }{2p}}e^{-\frac{\lambda _{1}}{p}t}\Vert u_{0}\Vert _{p,\lambda
},  \label{Thm-Disp-aux-2}
\end{equation}%
where $d(t)=\min \{1,t\}$ and $C=C(m)$.
\end{enumerate}
\end{thr}

\noindent {\textbf{Proof.}} Denote $u=e^{t\Delta _{g}}u_{0}.$ For the first
item, we apply H\"{o}lder's inequality to the Green representation formula %
\eqref{green}, obtaining
\begin{equation*}
|u(t,x)|^{p}\leq \int_{M}G(t,x,y)|u_{0}(y)|^{p}\,dy.
\end{equation*}%
Using the heat kernel estimate \eqref{es-h4} and passing to the radial
(pushforward) measure, we deduce
\begin{equation}
|u(t,x)|^{p}\leq C\max \left\{ 1,t^{-\frac{m}{2}}\right\} \int_{0}^{\infty
}e^{-\lambda _{1}t-\frac{r^{2}}{4t}}\,d\rho (r),  \label{aux-proof-1}
\end{equation}%
where $\rho (r)=\int_{B_{x}(r)}|u_{0}(y)|^{p}\,dy$ and $B_{x}(r)$ denotes
the geodesic ball of radius $r$ and centered at $x$.

Integrating by parts, observing that $|B_{x}(r)|^{-\frac{\lambda }{m}}\rho
(r)\leq \Vert u_{0}\Vert _{g,p,\lambda }^{p}$, using the volume estimate for
geodesic balls \eqref{eq-ball}, and inequality $-\frac{r^{2}}{4t}+k\frac{%
\lambda }{m}r\leq -c\frac{r^{2}}{t}$ for $r\geq \gamma t$, we can estimate
\begin{align*}
\int_{0}^{\infty }e^{-\lambda _{1}t-\frac{r^{2}}{4t}}\,d\rho (r)& \leq
e^{-\lambda _{1}t}\int_{0}^{\infty }e^{-\frac{r^{2}}{4t}}\,d\rho (r) \\
& =e^{-\lambda _{1}t}\left[ e^{-\frac{r^{2}}{4t}}\rho (r)\bigg|_{0}^{\infty
}+\int_{0}^{\infty }\dfrac{r}{2t}e^{-\frac{r^{2}}{4t}}\rho (r)\,dr\right] \\
& \leq Ce^{-\lambda _{1}t}\Vert u_{0}\Vert _{g,p,\lambda
}^{p}\int_{0}^{\infty }\dfrac{r^{\lambda +1}}{2t}e^{-\frac{r^{2}}{4t}+k\frac{%
\lambda }{m}r}\,dr \\
& \leq Ce^{-\lambda _{1}t}\Vert u_{0}\Vert _{g,p,\lambda }^{p}\left[
\int_{0}^{\gamma t}\dfrac{r^{\lambda +1}}{2t}e^{-\frac{r^{2}}{4t}+k\frac{%
\lambda }{m}r}\,dr+\int_{\gamma t}^{\infty }\dfrac{r^{\lambda +1}}{2t}e^{-%
\frac{r^{2}}{4t}+k\frac{\lambda }{m}r}\,dr\right] \\
& \leq Ce^{-\lambda _{1}t}\Vert u_{0}\Vert _{g,p,\lambda }^{p}\left[ e^{k%
\frac{\lambda }{m}\gamma t}\int_{0}^{\gamma t}\dfrac{r^{\lambda +1}}{2t}e^{-%
\frac{r^{2}}{4t}}\,dr+\int_{\gamma t}^{\infty }\dfrac{r^{\lambda +1}}{2t}%
e^{-c\frac{r^{2}}{t}}\,dr\right] \\
& \leq Ct^{\frac{\lambda }{2}}e^{-\lambda _{1}t+k\frac{\lambda }{m}\gamma
t}\Vert u_{0}\Vert _{g,p,\lambda }^{p}.
\end{align*}%
In this way, we obtain
\begin{equation*}
|u(t,x)|\leq Cd(t)^{-\frac{m}{2p}}t^{\frac{\lambda }{2p}}e^{-\frac{1}{p}%
\left[ \lambda _{1}-k\frac{\lambda }{m}\gamma \right] t}\Vert u_{0}\Vert
_{g,p,\lambda },\text{ with }d(t)=\max \{1,t\}.
\end{equation*}%
Taking the essential supremum over $x\in M$ on the left-hand side then
yields the estimate stated in (\ref{Thm-Disp-aux-1}).

For the second item, let $u_{0}\in \mathcal{M}_{p,\lambda }.$ Using the
estimate \eqref{es-h3}, we have
\begin{equation}
|u(t,x)|^{p}\leq C\max \left\{ 1,t^{-\frac{m}{2}}\right\} \int_{0}^{\infty
}e^{-\lambda _{1}t-\frac{r^{2}}{2Dt}}\,d\rho (r).  \label{aux-proof-2}
\end{equation}%
Moreover, the integral can be estimated as follows:
\begin{align*}
\int_{0}^{t}e^{-\lambda _{1}t-\frac{r^{2}}{2Dt}}\,d\rho (r)& \leq
\int_{0}^{\infty }e^{-\lambda _{1}t-\frac{r^{2}}{2Dt}}\,d\rho (r) \\
& \leq C\dfrac{e^{-\lambda _{1}t}}{Dt}\Vert u_{0}\Vert _{p,\lambda
}^{p}\int_{0}^{\infty }r^{\lambda +1}e^{-\frac{r^{2}}{2Dt}}\,dr \\
& \leq Ce^{-\lambda _{1}t}t^{\frac{\lambda }{2}}\Vert u_{0}\Vert _{p,\lambda
}^{p}.
\end{align*}%
Finally, combining this estimate with (\ref{aux-proof-2}), we obtain the
desired pointwise bound%
\begin{equation*}
|u(t,x)|\leq d(t)^{-\frac{m}{2p}}t^{\frac{\lambda }{2p}}e^{-\frac{\lambda
_{1}}{p}t}\Vert u_{0}\Vert _{p,\lambda }.
\end{equation*}%
\fin

Different from the Euclidean case (\cite{Kato1992}), we do not have the
convolution property to get the estimate between the Morrey spaces.
Furthermore, we need exponential decay to prove the boundedness of the Riesz
operator (see Theorem \ref{th-ri}); consequently, in our setting, we need to
analyze the heat kernel more deeply.

\begin{thr}[Dispersive: $\mathcal{M}_{p,\protect\lambda }^{g}\rightarrow
\mathcal{M}_{p,\protect\lambda }^{g}$]
\label{th-mm} Let $(M,g)$ be a complete, simply connected, non-compact, $m$%
-dimensional Riemannian manifold without boundary. Let $p\in \lbrack
1,\infty )$ and $\lambda \in \lbrack 0,m)$ and denote by $u_{0}$ the initial
data for the heat equation.

\begin{enumerate}
\item[$(i)$] Assume that $-K(m-1)\leq \func{Ric}\leq -c_{0}<0$ and $\kappa
\leq 0,$ and let $u_{0}\in \mathcal{M}_{p,\lambda }^{g}(\Gamma (TM))$. Then,
for all $t>0$, the following estimate holds:
\begin{equation*}
\Vert e^{t\Delta _{g}}u_{0}\Vert _{g,p,\lambda }\leq Ce^{-\alpha _{p}t+k%
\frac{\lambda }{mp}\gamma _{g}t}\Vert u_{0}\Vert _{g,p,\lambda },
\end{equation*}%
where $k=\sqrt{K}(m-1),$ $\alpha _{p}=\min \left\{ \frac{4\delta _{m}(p-1)}{%
p^{2}},\frac{\lambda _{1}}{p}\right\} $, $\delta _{m}\geq \frac{1}{4}\left(
c_{0}^{2}+(m-1)(m-2)\kappa ^{\ast }\right) ,$ $\kappa ^{\ast
}=\sup_{M}\kappa ,$ $\gamma _{g}=k\left( 1+\frac{\lambda }{m}\right) \left(
\frac{1}{4}-c\right) ^{-1},$ $0<c<\frac{1}{4}$, and $C=C(m,c_{0},K)$.

\item[$(ii)$] Assume that $\func{Ric}\geq -K(m-1)$ and $\kappa \leq 0$, and
let $u_{0}\in \mathcal{M}_{p,\lambda }(\Gamma (TM)).$ Then, for all $t>0$,
the following estimate holds:
\begin{equation*}
\Vert e^{t\Delta _{g}}u_{0}\Vert _{p,\lambda }\leq C\left( 1+\gamma t\right)
^{\frac{\lambda }{p}}e^{-\alpha _{p}t}\Vert u_{0}\Vert _{p,\lambda },
\end{equation*}%
where $\gamma =k\frac{\lambda }{m}\left( \frac{1}{4}-c\right) ^{-1}$, $0<c<%
\frac{1}{4}$, and $C=C(m,K).$
\end{enumerate}
\end{thr}

\noindent {\textbf{Proof.}} We begin by noting that for $u_{0}\in \mathcal{M}%
_{p,\lambda }^{g}(\Gamma (TM))$ (item $(i)$), we employ the heat kernel
estimate \eqref{es-h4}, whereas for $u_{0}\in \mathcal{M}_{p,\lambda
}(\Gamma (TM))$ (item $(ii)$), we use the estimate \eqref{es-h3}.

For the case $R\leq 1$, we analyze separately the terms $e^{t\Delta
_{g}}u_{0}\chi _{B_{x}(R+\varepsilon R)}$ and $e^{t\Delta _{g}}u_{0}\left(
1-\chi _{B_{x}(R+\varepsilon R)}\right) $ for some fixed $\varepsilon >0.$
It follows that
\begin{align*}
\left( \int_{B_{x_{0}}(R)}|u(t,x)|^{p}\,dx\right) ^{\frac{1}{p}}& \leq
\left( \int_{B_{x_{0}}(R)}\bigg|\int_{M}G(t,x,y)u_{0}(y)\,dy\bigg|%
^{p}\,dx\right) ^{\frac{1}{p}} \\
& \leq \left( \int_{B_{x_{0}}(R)}\bigg|\int_{M}G(t,x,y)u_{0}(y)\chi
_{B_{x}(\varepsilon R)}(y)\,dy\bigg|^{p}\,dx\right) ^{\frac{1}{p}}+ \\
& \hspace{2cm}+\left( \int_{B_{x_{0}}(R)}\bigg|\int_{M\smallsetminus
B_{x}(\varepsilon R)}G(t,x,y)u_{0}(y)\,dy\bigg|^{p}\,dx\right) ^{\frac{1}{p}}
\\
& =I_{1}+I_{2}.
\end{align*}%
For the first integral $I_{1}$, using the $L^{p}$-estimates for the
semigroup $e^{t\Delta _{g}}$ (see \cite[Theorem 4.1 and Proposition 4.6]%
{Pierfelice2017}), we obtain
\begin{align}
I_{1}\leq \Vert e^{t\Delta _{g}}u_{0}\chi _{B_{x_{0}}(R+\varepsilon R)}\Vert
_{p}& \leq Ce^{-\frac{4\delta _{m}(p-1)}{p^{2}}t}\Vert u_{0}\chi
_{B_{x_{0}}(R+\varepsilon R)}\Vert _{p}  \notag \\
& \leq Ce^{-\frac{4\delta _{m}(p-1)}{p^{2}}t}R^{\frac{\lambda }{p}%
}(1+\varepsilon )^{\frac{\lambda }{p}}\Vert u_{0}\Vert _{g,p,\lambda }
\notag \\
& \leq Ce^{-\frac{4\delta _{m}(p-1)}{p^{2}}t}B_{x_{0}}(R)^{\frac{\lambda }{mp%
}}(1+\varepsilon )^{\frac{\lambda }{p}}\Vert u_{0}\Vert _{g,p,\lambda },
\label{eq-mm-00}
\end{align}%
where the parameter $\delta _{m}\geq \frac{1}{4}\left(
c_{0}^{2}+(m-1)(m-2)\kappa ^{\ast }\right) $ with $\kappa ^{\ast
}=\sup_{M}\kappa $. Note that in the case $u_{0}\in \mathcal{M}_{p,\lambda }$%
, we alternatively have
\begin{equation}
I_{1}\leq Ce^{-\frac{4\delta _{m}(p-1)}{p^{2}}t}R^{\frac{\lambda }{p}%
}(1+\varepsilon )^{\frac{\lambda }{p}}\Vert u_{0}\Vert _{p,\lambda }.
\label{eq-mm-00-item2}
\end{equation}

For the second integral $I_{2}$, we introduce the functions $%
r(x,x_{0})^{-(m-\lambda )}$ and $r(x,x_{0})^{m-\lambda }$. Observing that $%
r(x,y)>\varepsilon R>\varepsilon r(x,x_{0})$, we can estimate, for all $%
t\leq 1,$
\begin{align}
I_{2}^{p}& \leq \int_{B_{x_{0}}(R)}r(x,x_{0})^{-(m-\lambda
)}\int_{M\smallsetminus B_{x}(\varepsilon R)}r(x,x_{0})^{m-\lambda
}G(t,x,y)|u_{0}(y)|^{p}\,dy\,dx  \notag \\
& \leq C\varepsilon ^{-(m-\lambda
)}\int_{B_{x_{0}}(R)}r(x,x_{0})^{-(m-\lambda )}\int_{M\smallsetminus
B_{x}(\varepsilon R)}r(x,y)^{m-\lambda }G(t,x,y)|u_{0}(y)|^{p}\,dy\,dx
\notag \\
& \leq C\varepsilon ^{-(m-\lambda )}d(t)^{-\frac{m}{2}}e^{-\lambda
t}\int_{B_{x_{0}}(R)}r(x,x_{0})^{-(m-\lambda )}\int_{\varepsilon R}^{\infty
}r^{m-\lambda }e^{-\frac{r^{2}}{4t}}\,d\rho _{x}(r)\,dx  \notag \\
& \leq C\varepsilon ^{-(m-\lambda )}d(t)^{-\frac{m}{2}}e^{-\lambda
t}\int_{B_{x_{0}}(R)}r(x,x_{0})^{-(m-\lambda )}\int_{0}^{\infty }\left(
\dfrac{r^{m-\lambda +1}}{2t}-(m-\lambda )r^{m-\lambda -1}\right) e^{-\frac{%
r^{2}}{4t}}\,\rho _{x}(r)\,dr\,dx  \notag \\
& \leq C\varepsilon ^{-(m-\lambda )}d(t)^{-\frac{m}{2}}e^{-\lambda
t}\int_{B_{x_{0}}(R)}r(x,x_{0})^{-(m-\lambda )}\int_{0}^{\infty }\dfrac{%
r^{m-\lambda +1}}{2t}e^{-\frac{r^{2}}{4t}}\,\rho _{x}(r)\,dr\,dx  \notag \\
& \leq C\varepsilon ^{-(m-\lambda )}d(t)^{-\frac{m}{2}}e^{-\lambda t}\Vert
u_{0}\Vert _{g,p,\lambda }^{p}\int_{B_{x_{0}}(R)}r(x,x_{0})^{-(m-\lambda
)}\,dx\int_{0}^{\infty }\dfrac{r^{m+1}}{2t}e^{-\frac{r^{2}}{4t}+k\frac{%
\lambda }{m}r}\,dr  \notag \\
& \leq C\varepsilon ^{-(m-\lambda )}e^{-\lambda _{1}t+k\frac{\lambda }{m}%
\gamma t}B_{x_{0}}(R)^{\frac{\lambda }{m}}\Vert u_{0}\Vert _{g,p,\lambda
}^{p},  \notag
\end{align}%
where $\gamma =k\frac{\lambda }{m}\left( \frac{1}{4}-c\right) ^{-1}$ with $%
0<c<\frac{1}{4}$. For $t>1$, using $R\leq 1$, we have
\begin{equation*}
I_{2}^{p}\leq \int_{B_{x_{0}}(R)}r(x,x_{0})^{-(m-\lambda
)}\int_{M\smallsetminus B_{x}(\varepsilon R)}r(x,x_{0})^{-\lambda
}G(t,x,y)|u_{0}(y)|^{p}\,dy\,dx.
\end{equation*}%
Proceeding analogously to the estimates above, we then obtain
\begin{equation}
I_{2}^{p}\leq C\varepsilon ^{-(m-\lambda )}e^{-\lambda _{1}t+k\frac{\lambda
}{m}\gamma t}B_{x_{0}}(R)^{\frac{\lambda }{m}}\Vert u_{0}\Vert _{g,p,\lambda
}^{p}.  \label{eq-mm-0}
\end{equation}%
In the case $u_{0}\in \mathcal{M}_{p,\lambda }$ (item $(ii)$), we get $%
I_{2}^{p}\leq C\varepsilon ^{-(m-\lambda )}e^{-\lambda _{1}t}R^{{\lambda }%
}\Vert u_{0}\Vert _{p,\lambda }^{p}$.

By \eqref{eq-mm-00} and \eqref{eq-mm-0}, we conclude that for $R\leq 1$
\begin{equation}
\left( B_{x_{0}}(R)^{-\frac{\lambda }{m}}\int_{B_{x_{0}}(R)}|u(t,x)|^{p}\,dx%
\right) ^{\frac{1}{p}}\leq Ce^{-\alpha _{p}t+k\frac{\lambda }{pm}\gamma
t}\Vert u_{0}\Vert _{g,p,\lambda },  \label{eq-mm-0p}
\end{equation}%
where $\alpha _{p}=\min \left\{ \frac{4\delta _{m}(p-1)}{p^{2}},\frac{%
\lambda _{1}}{p}\right\} $. Analogously, for $u_{0}\in \mathcal{M}%
_{p,\lambda }$, we have
\begin{equation*}
\left( R^{-\lambda }\int_{B_{x_{0}}(R)}|u(t,x)|^{p}\,dx\right) ^{\frac{1}{p}%
}\leq Ce^{-\alpha _{p}t}\Vert u_{0}\Vert _{p,\lambda }.
\end{equation*}

For $R>1$, fix $\gamma _{g}>0$ and split the integral as
\begin{align*}
\left( \int_{B_{x_{0}}(R)}|u(t,x)|^{p}\,dx\right) ^{\frac{1}{p}}& \leq
\left( \int_{B_{x_{0}}(R)}\bigg|\int_{M}G(t,x,y)u_{0}(y)\,dy\bigg|%
^{p}\,dx\right) ^{\frac{1}{p}} \\
& \leq \left( \int_{B_{x_{0}}(R)}\bigg|\int_{M}G(t,x,y)u_{0}(y)\chi
_{B_{x}(\gamma _{g}t)}(y)\,dy\bigg|^{p}\,dx\right) ^{\frac{1}{p}}+ \\
& \hspace{2cm}+\left( \int_{B_{x_{0}}(R)}\bigg|\int_{M\smallsetminus
B_{x}(\gamma _{g}t)}G(t,x,y)u_{0}(y)\,dy\bigg|^{p}\,dx\right) ^{\frac{1}{p}}
\\
& =I_{3}+I_{4}.
\end{align*}%
For the first integral $I_{3}$, we use the $L^{p}$ results (see \cite[%
Theorem 4.1 and Proposition 4.6]{Pierfelice2017}) and the volume comparison %
\eqref{comp} (see \cite[Lemma 2.1]{li}) to estimate
\begin{align}
I_{3}\leq \Vert e^{t\Delta _{g}}u_{0}\chi _{B_{x_{0}}(R+\gamma _{g}t)}\Vert
_{p}& \leq Ce^{-\frac{4\delta _{m}(p-1)}{p^{2}}t}\Vert u_{0}\chi
_{B_{x_{0}}(R+\gamma _{g}t)}\Vert _{p}  \notag \\
& \leq Ce^{-\frac{4\delta _{m}(p-1)}{p^{2}}t}|B_{x_{0}}(R+\gamma _{g}t)|^{%
\frac{\lambda }{mp}}\Vert u_{0}\Vert _{g,p,\lambda }  \notag \\
& \leq Ce^{-\frac{4\delta _{m}(p-1)}{p^{2}}t}|B_{x_{0}}(R)|^{\frac{\lambda }{%
mp}}e^{k\frac{\lambda }{mp}\gamma _{g}t}\Vert u_{0}\Vert _{g,p,\lambda }.
\label{eq-mm1}
\end{align}%
In the case $u_{0}\in \mathcal{M}_{p,\lambda },$ this estimate simplifies to
$I_{3}\leq Ce^{-\frac{4\delta _{m}(p-1)}{p^{2}}t}R^{\frac{\lambda }{p}%
}\left( 1+\gamma t\right) ^{\frac{\lambda }{p}}\Vert u_{0}\Vert _{p,\lambda
} $.

For the second integral $I_{4}$, note that the heat kernel can be estimate
by a function $F(t,r(x,y))$ that depends on the geodesic distance $r(x,y)$.
Using this estimate and H\"{o}lder's inequality, we obtain
\begin{align*}
I_{4}^{p}& \leq \int_{B_{x_{0}}(R)}\int_{M\smallsetminus B_{x}(\gamma
_{g}t)}G(t,x,y)|u_{0}(y)|^{p}\,dy\,dx \\
& \leq \int_{B_{x_{0}}(R)}\int_{M\smallsetminus B_{x}(\gamma
_{g}t)}F(t,r(x,y))|u_{0}(y)|^{p}\,dy\,dx \\
& =\int_{B_{x_{0}}(R)}\int_{\gamma _{g}t}^{\infty }F(t,r)\,d\rho _{x}(r)\,dx,
\end{align*}%
where $\rho _{x}(r)=\int_{B_{x}(r)}|u_{0}(z)|^{p}\,dz$. It then follows that
\begin{align*}
I_{4}^{p}& \leq \int_{B_{x_{0}}(R)}\left[ F(t,r)\rho _{x}(r)\bigg|_{\gamma
_{g}t}^{\infty }+\int_{\gamma _{g}t}^{\infty }\left( -dF(t,r)\right) \,\rho
_{x}(r)\,dr\right] \,dx \\
& \leq \int_{\gamma _{g}t}^{\infty }\left( -dF(t,r)\right)
\int_{B_{x_{0}}(R)}\rho _{x}(r)\,dx\,dr.
\end{align*}%
Now, note that
\begin{align*}
\int_{B_{x_{0}}(R)}\rho _{x}(r)\,dx&
=\int_{B_{x_{0}}(R)}\int_{B_{x}(r)}|u_{0}(z)|^{p}\,dz\,dx \\
& =\int_{B_{x_{0}}(R)}\int_{M}|u_{0}(z)|^{p}\chi _{B_{x}(r)}(z)\,dz\,dx \\
& =\int_{M}|u_{0}(z)|^{p}\int_{B_{x_{0}}(R)}\chi _{B_{x}(r)}(z)\,dx\,dz,
\end{align*}%
where $\chi _{B_{x}(r)}$ is the characteristic function of $B_{x}(r)$.
Observe also that $\int_{B_{x_{0}}(R)}\chi
_{B_{x}(r)}(z)\,dx=|B_{x_{0}}(R)\cap B_{z}(r)|$, because for fixed $z\in
B_{x_{0}}(R)$, we have $\chi _{B_{x}(r)}(z)=1$ if and only if $r(z,x)<r$
with $x\in B_{x_{0}}(R)$. Moreover, we can write
\begin{equation*}
|B_{x_{0}}(R)\cap B_{z}(r)|=|B_{x_{0}}(R)\cap B_{z}(r)|\chi
_{B_{x_{0}}(R+r)}(z).
\end{equation*}%
Thus, using \eqref{comp} (see \cite[Lemma 2.1]{li}), we obtain
\begin{align*}
\int_{B_{x_{0}}(R)}\rho _{x}(r)\,dx& \leq
\int_{M}|B_{z}(r)||u_{0}(z)|^{p}\chi _{B_{x_{0}}(R+r)}(z),dz \\
& \leq Cr^{n}e^{kr}\int_{B_{x_{0}}(R+r)}|u_{0}(z)|^{p},dz \\
& \leq Cr^{n}e^{kr}|B_{x_{0}}(R+r)|^{\frac{\lambda }{m}}\Vert u_{0}\Vert
_{g,p,\lambda }^{p} \\
& \leq Cr^{n}e^{kr}|B_{x_{0}}(R)|^{\frac{\lambda }{m}}e^{k\frac{\lambda }{m}%
r}\Vert u_{0}\Vert _{g,p,\lambda }^{p} \\
& \leq Cr^{n}e^{k\left( 1+\frac{\lambda }{m}\right) r}|B_{x_{0}}(R)|^{\frac{%
\lambda }{m}}\Vert u_{0}\Vert _{g,p,\lambda }^{p}.
\end{align*}%
where $n\in \lbrack 0,m]$. For $u_{0}\in \mathcal{M}_{p,\lambda }$ (item $%
(ii)$)$,$ this estimate becomes
\begin{equation*}
\int_{B_{x_{0}}(R)}\rho _{x}(r)\,dx\leq Cr^{n}\left( 1+r\right) ^{\lambda
}e^{kr}R^{{\lambda }}\Vert u_{0}\Vert _{p,\lambda }^{p}.
\end{equation*}

In this way, taking $0<c<\frac{1}{4}$ and $\gamma _{g}=k\left( 1+\frac{%
\lambda }{m}\right) \left( \frac{1}{4}-c\right) ^{-1}\geq \gamma ,$ we can
estimate
\begin{align*}
|B_{x_{0}}(R)|^{-\frac{\lambda }{m}}I_{4}^{p}& \leq C\Vert u_{0}\Vert
_{g,p,\lambda }^{p}\int_{\gamma _{g}t}^{\infty }r^{n}e^{k\left( 1+\frac{%
\lambda }{m}\right) r}\left( -dF(t,r)\right) \,dr \\
& \leq Cd(t)^{-\frac{m}{2}}e^{-\lambda _{1}t}\Vert u_{0}\Vert _{g,p,\lambda
}^{p}\int_{\gamma _{g}t}^{\infty }r^{n}e^{k\left( 1+\frac{\lambda }{m}%
\right) r}d\left( -e^{-\frac{r^{2}}{4t}}\right) \,dr \\
& \leq Cd(t)^{-\frac{m}{2}}t^{-1}e^{-\lambda _{1}t}\Vert u_{0}\Vert
_{g,p,\lambda }^{p}\int_{\gamma _{g}t}^{\infty }r^{n+1}e^{-\frac{r^{2}}{4t}%
+k\left( 1+\frac{\lambda }{m}\right) r}\,dr \\
& \leq Cd(t)^{-\frac{m}{2}}t^{-1}e^{-\lambda _{1}t}\Vert u_{0}\Vert
_{g,p,\lambda }^{p}\int_{0}^{\infty }r^{n+1}e^{-c\frac{r^{2}}{t}}\,dr \\
& \leq Cd(t)^{-\frac{m}{2}}t^{\frac{n}{2}}e^{-\lambda _{1}t}\Vert u_{0}\Vert
_{g,p,\lambda }^{p}\int_{0}^{\infty }s^{n+1}e^{-cs^{2}}\,ds \\
& \leq Cd(t)^{-\frac{m}{2}}t^{\frac{n}{2}}e^{-\lambda _{1}t}\Vert u_{0}\Vert
_{g,p,\lambda }^{p}.
\end{align*}%
For $u_{0}\in \mathcal{M}_{p,\lambda }$, the argument is analogous by taking
$\gamma =k\left( \frac{1}{4}-c\right) ^{-1}$. Moreover, if $t\leq 1,$ we
take $n=m,$ which yields
\begin{equation}
|B_{x_{0}}(R)|^{-\frac{\lambda }{m}}I_{4}^{p}\leq Ce^{-{\lambda _{1}}t}\Vert
u_{0}\Vert _{g,p,\lambda }^{p}.  \label{eq-mm2}
\end{equation}%
If $t>1,$ we take $n=0$, yielding the same bound%
\begin{equation}
|B_{x_{0}}(R)|^{-\frac{\lambda }{m}}I_{4}^{p}\leq Ce^{-\lambda _{1}t}\Vert
u_{0}\Vert _{g,p,\lambda }^{p}.  \label{eq-mm3}
\end{equation}%
By \eqref{eq-mm1}, \eqref{eq-mm2} and \eqref{eq-mm3}, for $R>1$ we obtain
\begin{equation}
\left( B_{x_{0}}(R)^{-\frac{\lambda }{m}}\int_{B_{x_{0}}(R)}|u(t,x)|^{p}\,dx%
\right) ^{\frac{1}{p}}\leq Ce^{-\alpha _{p}t+k\frac{\lambda }{mp}\gamma
_{g}t}\Vert u_{0}\Vert _{g,p,\lambda },  \label{eq-mm-1p}
\end{equation}%
where $\alpha _{p}=\min \left\{ \frac{4\delta _{m}(p-1)}{p^{2}},\frac{%
\lambda _{1}}{p}\right\} $. Finally, combining \eqref{eq-mm-0p} and %
\eqref{eq-mm-1p}, we deduce%
\begin{equation*}
\Vert u\Vert _{g,p,\lambda }\leq Ce^{-\alpha _{p}t+k\frac{\lambda }{mp}%
\gamma _{g}t}\Vert u_{0}\Vert _{g,p,\lambda }.
\end{equation*}%
In the case where $u_{0}\in \mathcal{M}_{p,\lambda }$ (item $(ii)$), we get
\begin{equation*}
\Vert u\Vert _{p,\lambda }\leq C\left( 1+\gamma t\right) ^{\frac{\lambda }{p}%
}e^{-\alpha _{p}t}\Vert u_{0}\Vert _{p,\lambda }.
\end{equation*}%
\fin

Now, combining the estimates established in Theorems \ref{th-mi} and \ref%
{th-mm}, we obtain the boundedness properties of the heat semigroup as an
operator from $\mathcal{M}_{p,\lambda }^{g}$ to $\mathcal{M}_{q,\lambda
}^{g} $.

\begin{thr}[Dispersive: $\mathcal{M}_{p,\protect\lambda }^{g}\rightarrow
\mathcal{M}_{q,\protect\lambda }^{g}$]
\label{th-d} Let $(M,g)$ be a complete, simply connected, non-compact, $m$%
-dimensional Riemannian manifold without boundary. Let $1\leq p\leq q<\infty
$ and $\lambda \in \lbrack 0,m)$ and denote by $u_{0}$ the initial data for
the heat equation.

\begin{enumerate}
\item[$(i)$] Assume that $-K(m-1)\leq \func{Ric}\leq -c_{0}<0$ and $\kappa
\leq 0,$ and let $u_{0}\in \mathcal{M}_{p,\lambda }^{g}(\Gamma (TM))$. Then,
for all $t>0$, the following estimate holds:
\begin{equation*}
\Vert e^{t\Delta _{g}}u_{0}\Vert _{g,q,\lambda }\leq C\left[ d(t)^{-\frac{m}{%
2}}t^{\frac{\lambda }{2}}\right] ^{\frac{1}{p}-\frac{1}{q}}e^{-\alpha
_{p,q}t+k\frac{\lambda }{mq}\gamma _{g}t}\Vert u_{0}\Vert _{g,p,\lambda },
\end{equation*}%
where $d(t)=\min \{1,t\}$, $\alpha _{p,q}=\min \left\{ \frac{4\delta
_{m}(p-1)}{pq},\frac{\lambda _{1}}{q}\right\} ,$ $k=\sqrt{K}(m-1),$ $\gamma
_{g}=k\left( 1+\frac{\lambda }{m}\right) \left( \frac{1}{4}-c\right) ^{-1}$
with $0<c<\frac{1}{4}$, and $C=C(m,c_{0},K)$.

\item[$(ii)$] Assume that $\func{Ric}\geq -K(m-1)$ and $\kappa \leq 0$, and
let $u_{0}\in \mathcal{M}_{p,\lambda }(\Gamma (TM)).$ Then, for all $t>0$,
the following estimate holds:
\begin{equation*}
\Vert e^{t\Delta _{g}}u_{0}\Vert _{q,\lambda }\leq C\left[ d(t)^{-\frac{m}{2}%
}t^{\frac{\lambda }{2}}\right] ^{\frac{1}{p}-\frac{1}{q}}\left( 1+\gamma
t\right) ^{\frac{\lambda }{q}}e^{-a_{p,q}t}\Vert u_{0}\Vert _{p,\lambda },
\end{equation*}%
where $\gamma =k\frac{\lambda }{m}\left( \frac{1}{4}-c\right) ^{-1}$ with $%
0<c<\frac{1}{4}$ and $C=C(m,K)$.
\end{enumerate}
\end{thr}

\noindent {\textbf{Proof.}} Firstly, note the interpolation-type estimates
\begin{equation}
\Vert u\Vert _{g,q,\lambda }\leq \Vert u\Vert _{\infty }^{\frac{q-p}{q}%
}\Vert u\Vert _{g,p,\lambda }^{\frac{p}{q}}\text{ and }\Vert u\Vert
_{q,\lambda }\leq \Vert u\Vert _{\infty }^{\frac{q-p}{q}}\Vert u\Vert
_{p,\lambda }^{\frac{p}{q}}.  \label{eq-pq}
\end{equation}%
Denoting $u=e^{t\Delta _{g}}u_{0}$ and using the estimates in Theorem \ref%
{th-mi}, it follows that
\begin{align*}
\Vert u(t)\Vert _{\infty }& \leq Cd(t)^{-\frac{m}{2p}}t^{\frac{\lambda }{2p}%
}e^{-\frac{1}{p}\left[ \lambda _{1}-k\frac{\lambda }{m}\gamma \right]
t}\Vert u_{0}\Vert _{g,p,\lambda }, \\
\Vert u(t)\Vert _{\infty }& \leq Cd(t)^{-\frac{m}{2p}}t^{\frac{\lambda }{2p}%
}e^{-\frac{\lambda _{1}}{p}t}\Vert u_{0}\Vert _{p,\lambda }.
\end{align*}%
In turn, Theorem \ref{th-mm} gives the estimates
\begin{align*}
\Vert u\Vert _{g,p,\lambda }& \leq Ce^{-\alpha _{p}t+k\frac{\lambda }{mp}%
\gamma _{g}t}\Vert u_{0}\Vert _{g,p,\lambda }, \\
\Vert u\Vert _{p,\lambda }& \leq C\left( 1+\gamma t\right) ^{\frac{\lambda }{%
p}}e^{-a_{p}t}\Vert u_{0}\Vert _{p,\lambda }.
\end{align*}%
Combining (\ref{eq-pq}) with the above estimates, we arrive at
\begin{align*}
\Vert u\Vert _{g,q,\lambda }& \leq C\left[ d(t)^{-\frac{m}{2}}t^{\frac{%
\lambda }{2}}\right] ^{\frac{1}{p}-\frac{1}{q}}e^{-\alpha _{p,q}t+k\frac{%
\lambda }{mq}\gamma _{g}t}\Vert u_{0}\Vert _{g,p,\lambda }, \\
\Vert u\Vert _{q,\lambda }& \leq C\left[ d(t)^{-\frac{m}{2}}t^{\frac{\lambda
}{2}}\right] ^{\frac{1}{p}-\frac{1}{q}}\left( 1+\gamma t\right) ^{\frac{%
\lambda }{q}}e^{-a_{p,q}t}\Vert u_{0}\Vert _{p,\lambda },
\end{align*}%
as desired.\fin


\section{Heat equation with Bochner Laplacian}

\label{s4}

In this section, we analyze the semigroup $e^{t(\overrightarrow{\Delta }+r)}$%
, where the operator $\overrightarrow{\Delta }$ denotes the Bochner
Laplacian defined as
\begin{equation*}
\overrightarrow{\Delta }u=-\nabla ^{\ast }\nabla u,
\end{equation*}%
and $\nabla ^{\ast }$ is the formal adjoint of $\nabla $ in $L^{2}$. Using
the gradient estimate for the heat kernel associated with the Hodge
Laplacian \eqref{es-ho}, we derive smoothing estimates for both $e^{t\Delta
_{g}}$ and $e^{t(\overrightarrow{\Delta }+r)}$. For this purpose, we adapt
some ideas found in Pierfelice \cite{Pierfelice2017} to the Morrey setting.

\subsection{Dispersive estimates}

\label{s4.1} Once again, the geometric assumption on the Ricci curvature
plays a crucial role in comparing operators. Indeed, by \cite[Lemma 4.5]%
{Pierfelice2017}, if $\func{Ric}\leq -c_{0}$ with $c_{0}>0$, then for any $%
u_{0}\in \mathcal{C}_{b}^{\infty }(\Gamma (TM))$ it follows that
\begin{equation}
|e^{t(\overrightarrow{\Delta }+r)}u_{0}|\leq e^{t({\Delta }%
_{g}-c_{0})}|u_{0}|.  \label{es-ric}
\end{equation}%
Furthermore, due the classical Weitzenb\"{o}ck identity, with $r(u)=\func{Ric%
}(u,\cdot )^{\sharp },$ it holds that%
\begin{equation}
\overrightarrow{\Delta }u=\Delta _{H}u+r(u).  \label{weit}
\end{equation}%
This identity allows us to relate the Bochner Laplacian to the Hodge
Laplacian, defined by $\Delta _{H}=-(dd^{\ast }+d^{\ast }d),$ where $d$
denotes the exterior derivative and $d^{\ast }$ its formal adjoint. For $%
u\in \Gamma (TM)$, note that $\Delta _{H}u=(\Delta _{H}u^{\flat })^{\sharp }$
and thus $e^{t(\overrightarrow{\Delta }+r)}=e^{t(\Delta _{H}+2r)}$.

Now, we can combine Theorem \ref{th-d} with \eqref{es-ric} to obtain the
desired bounds for $e^{t(\overrightarrow{\Delta }+r)}$.

\begin{thr}[Dispersive: $\mathcal{M}_{p,\protect\lambda }^{g}\rightarrow
\mathcal{M}_{q,\protect\lambda }^{g}$]
\label{th-cd} Let $(M,g)$ be a complete, simply connected, non-compact, $m$%
-dimensional Riemannian manifold without boundary. Let $1\leq p\leq q<\infty
$ and $\lambda \in \lbrack 0,m)$ and denote by $u_{0}$ the initial data for
the heat equation.

\begin{enumerate}
\item[$(i)$] Assume that $-K(m-1)\leq \func{Ric}\leq -c_{0}<0$ and $\kappa
\leq 0,$ and let $u_{0}\in \mathcal{M}_{p,\lambda }^{g}(\Gamma (TM))$. Then,
for all $t>0$, the following estimate holds:
\begin{equation*}
\Vert e^{t(\overrightarrow{\Delta }+r)}u_{0}\Vert _{g,q,\lambda }\leq C\left[
d(t)^{-\frac{m}{2}}t^{\frac{\lambda }{2}}\right] ^{\frac{1}{p}-\frac{1}{q}%
}e^{-c_{0}t-\alpha _{p,q}t+k\frac{\lambda }{mq}\gamma _{g}t}\Vert u_{0}\Vert
_{g,p,\lambda },
\end{equation*}%
where $d(t)=\min \{1,t\}$, $\alpha _{p,q}=\min \left\{ \frac{4\delta
_{m}(p-1)}{pq},\frac{\lambda _{1}}{q}\right\} ,$ $k=\sqrt{K}(m-1),$ $\gamma
_{g}=k\left( 1+\frac{\lambda }{m}\right) \left( \frac{1}{4}-c\right) ^{-1}$
with $0<c<\frac{1}{4}$, and $C=C(m,c_{0},K)$.

\item[$(ii)$] Assume that $-K(m-1)\leq \func{Ric}\leq -c_{0}<0$, and let $%
u_{0}\in \mathcal{M}_{p,\lambda }(\Gamma (TM)).$ Then, for all $t>0$, the
following estimate holds:
\begin{equation*}
\Vert e^{t(\overrightarrow{\Delta }+r)}u_{0}\Vert _{q,\lambda }\leq C\left[
d(t)^{-\frac{m}{2}}t^{\frac{\lambda }{2}}\right] ^{\frac{1}{p}-\frac{1}{q}%
}\left( 1+\gamma t\right) ^{\frac{\lambda }{q}}e^{-c_{0}t-a_{p,q}t}\Vert
u_{0}\Vert _{p,\lambda },
\end{equation*}%
where $\gamma =k\frac{\lambda }{m}\left( \frac{1}{4}-c\right) ^{-1}$ with $%
0<c<\frac{1}{4}$ and $C=C(m,K)$.
\end{enumerate}
\end{thr}

\subsection{Smoothing estimates}

\label{s4.2} Having established the dispersive estimates in Morrey spaces,
the next objective is to derive the corresponding smoothing results
(gradient estimates). To this end, we consider the following definition of
bounded geometry (see \cite{Buttig},\cite{Shubin1992}).

\begin{df}
Let $(M,g)$ be a Riemannian manifold. We say that $M$ has bounded geometry
if the following conditions hold:

\begin{enumerate}
\item[$(i)$] The injectivity radius is uniformly positive, that is, $%
r_{inj}=\inf_{x\in M}r_{x}>0$, where $r_{x}$ denotes the injectivity radius
at $x$ defined via the exponential map $\exp _{x}$;

\item[$(ii)$] The covariant derivatives of the Riemann curvature tensor $%
\mathcal{R}$ are uniformly bounded: for every $i\in \mathbb{Z}_{\geq 0},$
there exists a constant $L_{i}$ such that $|\nabla ^{i}\mathcal{R}|\leq
L_{i} $ on $M$.
\end{enumerate}
\end{df}

By Buttig and Eichhorn \cite{Buttig}, if $M$ has bounded geometry, then
there exist constants $C_{1},C_{2}>0$, depending on $l,i,j$, such that for
all $x,y\in M$ and $t\in (0,\infty )$,
\begin{equation}
\bigg|\left( \partial _{t}\right) ^{l}\nabla ^{i}\nabla ^{j}H(t,x,y)\bigg|%
\leq C_{1}t^{-\frac{m}{2}-\frac{i+j}{2}-l}e^{-C_{2}\frac{r^{2}}{t}},
\label{es-ho}
\end{equation}%
where $H$ denotes the integral kernel of $e^{t\Delta _{H}}$. Using %
\eqref{es-ho}, we are going to estimate the norm of $\nabla e^{t\Delta _{g}}$
and $\nabla e^{t(\overrightarrow{\Delta }+r)}$ through Duhamel's formula.
The exponential decay of the norm of $\nabla e^{t\Delta _{g}}$ will play a
crucial role in establishing the boundedness of the Riesz transform (see
Theorem \ref{th-ri}).

\begin{thr}[Smoothing: $\mathcal{M}_{p,\protect\lambda }^{g}\rightarrow
\mathcal{M}_{q,\protect\lambda }^{g}$]
\label{es-s} Let $(M,g)$ be a complete, simply connected, non-compact, $m$%
-dimensional Riemannian manifold without boundary. Let $1\leq p\leq q<\infty
$ and $\lambda \in \lbrack 0,m)$ and denote by $u_{0}$ the initial data for
the heat equation.

\begin{enumerate}
\item[$(i)$] Assume that $M$ has bounded geometry, $-K(m-1)\leq \func{Ric}%
\leq -c_{0}<0$, $\kappa \leq 0,$ and let $u_{0}\in \mathcal{M}_{p,\lambda
}^{g}(\Gamma (TM))$. Then, for all $t>0$, the following estimate holds:
\begin{equation*}
\Vert \nabla e^{t\Delta _{g}}u_{0}\Vert _{g,q,\lambda }\leq C\left[ d(t)%
\right] ^{-\frac{1}{2}-\frac{m}{2}\left( \frac{1}{p}-\frac{1}{q}\right) }t^{%
\frac{\lambda }{2}\left( \frac{1}{p}-\frac{1}{q}\right) }e^{-\alpha _{p,q}t+k%
\frac{\lambda }{mq}\gamma _{g}t}\Vert u_{0}\Vert _{g,p,\lambda },
\end{equation*}%
where $d(t)=\min \{1,t\}$, $\alpha _{p,q}=\min \left\{ \frac{4\delta
_{m}(p-1)}{pq},\frac{\lambda _{1}}{q}\right\} ,$ $k=\sqrt{K}(m-1),$ $\gamma
_{g}=k\left( 1+\frac{\lambda }{m}\right) \left( \frac{1}{4}-c\right) ^{-1}$
with $0<c<\frac{1}{4}$, and $C=C(m,c_{0},K)$.

\item[$(ii)$] Assume that $M$ has bounded geometry, $\func{Ric}\geq -K(m-1),$
$\kappa \leq 0,$ and let $u_{0}\in \mathcal{M}_{p,\lambda }(\Gamma (TM))$.
Then, for all $t>0$, the following estimate holds:
\begin{equation*}
\Vert \nabla e^{t\Delta _{g}}u_{0}\Vert _{q,\lambda }\leq C\left[ d(t)\right]
^{-\frac{1}{2}-\frac{m}{2}\left( \frac{1}{p}-\frac{1}{q}\right) }t^{\frac{%
\lambda }{2}\left( \frac{1}{p}-\frac{1}{q}\right) }\left( 1+\gamma t\right)
^{\frac{\lambda }{q}}e^{-a_{p,q}t}\Vert u_{0}\Vert _{p,\lambda },
\end{equation*}%
where $\gamma =k\frac{\lambda }{m}\left( \frac{1}{4}-c\right) ^{-1}$ with $%
0<c<\frac{1}{4}$ and $C=C(m,K)$.
\end{enumerate}
\end{thr}

\noindent {\textbf{Proof.}} We prove the estimate in the case $p=q$. For $%
0<t\leq 1,$ we claim that
\begin{equation}
\Vert \nabla e^{t\Delta _{g}}u_{0}^{\flat }\Vert _{g,p,\lambda }\leq \dfrac{C%
}{\sqrt{t}}\Vert u_{0}\Vert _{g,p,\lambda }.  \label{eq-na}
\end{equation}%
Indeed, noting that
\begin{equation*}
\nabla e^{t\Delta _{g}}u_{0}^{\flat }=\int_{M}\nabla _{x}H(t,x,y)\wedge \ast
u_{0}^{\flat }(y)\,dy,
\end{equation*}%
and then using \eqref{es-ho}, we obtain
\begin{equation}
|\nabla e^{t\Delta _{g}}u_{0}^{\flat }|\leq C_{1}t^{-\frac{1}{2}}\int_{M}t^{-%
\frac{m}{2}}e^{-C_{2}\frac{r^{2}}{t}}|u_{0}(y)|\,dy.
\label{aux-proof-grad-1}
\end{equation}%
Combining estimate (\ref{aux-proof-grad-1}) with H\"{o}lder's inequality and
the arguments used in Theorem \ref{th-mm}, we can deduce \eqref{eq-na} for $%
0<t\leq 1$.

For the case $t>1$, we first use the semigroup property to write
\begin{equation*}
\nabla e^{t\Delta _{g}}=\nabla e^{\Delta _{g}}\left( e^{(t-1)\Delta
_{g}}\right) .
\end{equation*}%
So, applying estimate (\ref{eq-na}) with $t=1$ together with Theorem \ref%
{th-mm}, we obtain
\begin{align}
\Vert \nabla e^{t\Delta _{g}}u_{0}^{\flat }\Vert _{g,p,\lambda }&
=\left\Vert \nabla e^{\Delta _{g}}\left( e^{(t-1)\Delta _{g}}u_{0}^{\flat
}\right) \right\Vert _{g,p,\lambda }  \notag \\
& \leq C\left\Vert e^{(t-1)\Delta _{g}}u_{0}^{\flat }\right\Vert
_{g,p,\lambda }  \notag \\
& \leq Ce^{-\alpha _{p}t+k\frac{\lambda }{mp}\gamma _{g}t}\left\Vert
u_{0}\right\Vert _{g,p,\lambda }.  \label{aux-proof-grad-2}
\end{align}

For $p<q$, we can apply Theorem \ref{th-d} together with the semigroup
decomposition
\begin{equation*}
\nabla e^{t\Delta _{g}}u_{0}^{\flat }=\nabla e^{\frac{t}{2}\Delta
_{g}}\left( e^{\frac{t}{2}\Delta _{g}}u_{0}^{\flat }\right) .
\end{equation*}%
This splitting allows us to combine the dispersive estimate for $e^{\frac{t}{%
2}\Delta _{g}}$, acting from $\mathcal{M}_{p,\lambda }^{g}$ to $\mathcal{M}%
_{q,\lambda }^{g}$, with the gradient bound (\ref{aux-proof-grad-2})
established above, and then obtain the desired bound for $\nabla e^{t\Delta
_{g}}u_{0}$ as an operator from $\mathcal{M}_{p,\lambda }^{g}$ to $\mathcal{M%
}_{q,\lambda }^{g}.$

Lastly, we remark that the proof for the case $u_{0}\in \mathcal{M}%
_{p,\lambda }$ (item $(ii)$) proceeds in an analogous manner and is left to
the reader.\qquad {}\fin

We now turn to the smoothing (gradient) estimates for the semigroup $e^{t(%
\overrightarrow{\Delta }+r)}$. These will be derived using the Weitzenb\"{o}%
ck identity \eqref{weit} together with Duhamel's formula.

\begin{thr}[Smoothing: $\mathcal{M}_{p,\protect\lambda }^{g}\rightarrow
\mathcal{M}_{q,\protect\lambda }^{g}$]
\label{th-s-ric} Let $(M,g)$ be a complete, simply connected, non-compact, $%
m $-dimensional Riemannian manifold without boundary. Let $1\leq p\leq
q<\infty $ and $\lambda \in \lbrack 0,m)$ and denote by $u_{0}$ the initial
data for the heat equation.

\begin{enumerate}
\item[$(i)$] Assume that $M$ has bounded geometry, $-K(m-1)\leq \func{Ric}%
\leq -c_{0}<0$, $\kappa \leq 0,$ and let $u_{0}\in \mathcal{M}_{p,\lambda
}^{g}(\Gamma (TM))$. Then, for all $t>0$, the following estimate holds:
\begin{equation*}
\Vert \nabla e^{t(\overrightarrow{\Delta }+r)}u_{0}\Vert _{g,q,\lambda }\leq
C\left[ d(t)\right] ^{-\frac{1}{2}-\frac{m}{2}\left( \frac{1}{p}-\frac{1}{q}%
\right) }t^{\frac{\lambda }{2}\left( \frac{1}{p}-\frac{1}{q}\right)
}e^{-c_{0}t-\alpha _{p,q}t+k\frac{\lambda }{mq}\gamma _{g}t}\Vert u_{0}\Vert
_{g,p,\lambda },
\end{equation*}%
where $d(t)=\min \{1,t\}$, $\alpha _{p,q}=\min \left\{ \frac{4\delta
_{m}(p-1)}{pq},\frac{\lambda _{1}}{q}\right\} ,$ $k=\sqrt{K}(m-1),$ $\gamma
_{g}=k\left( 1+\frac{\lambda }{m}\right) \left( \frac{1}{4}-c\right) ^{-1}$
with $0<c<\frac{1}{4}$, and $C=C(m,c_{0},K)$.

\item[$(ii)$] Assume that $M$ has bounded geometry, $-K(m-1)\leq \func{Ric}%
\leq -c_{0}<0$, $\kappa \leq 0$, and $u_{0}\in \mathcal{M}_{p,\lambda
}(\Gamma (TM)).$ Then, for all $t>0$, the following estimate holds:
\begin{equation*}
\Vert \nabla e^{t(\overrightarrow{\Delta }+r)}u_{0}\Vert _{q,\lambda }\leq C%
\left[ d(t)\right] ^{-\frac{1}{2}-\frac{m}{2}\left( \frac{1}{p}-\frac{1}{q}%
\right) }t^{\frac{\lambda }{2}\left( \frac{1}{p}-\frac{1}{q}\right) }\left(
1+\gamma t\right) ^{\frac{\lambda }{q}}e^{-c_{0}t-a_{p,q}t}\Vert u_{0}\Vert
_{p,\lambda },
\end{equation*}%
where $\gamma =k\frac{\lambda }{m}\left( \frac{1}{4}-c\right) ^{-1}$ with $%
0<c<\frac{1}{4}$ and $C=C(m,c_{0},K)$.
\end{enumerate}
\end{thr}

\noindent {\textbf{Proof.}} We begin with the case $p=q$. Recall first that%
\begin{equation*}
\nabla e^{t(\overrightarrow{\Delta }+r)}=\nabla e^{t(\Delta _{H}+2r)}.
\end{equation*}%
Also, for $0<t\leq 1,$ we can estimate
\begin{equation*}
\Vert \nabla e^{t\Delta _{H}}u^{\flat }\Vert _{g,p,\lambda }\leq \dfrac{C}{%
\sqrt{t}}\Vert u_{0}\Vert _{g,p,\lambda }.
\end{equation*}%
Using Duhamel's formula, we have
\begin{equation*}
\nabla e^{t(\Delta _{H}+2r)}u_{0}^{\flat }=\nabla e^{t\Delta
_{H}}u_{0}^{\flat }+\int_{0}^{t}\nabla e^{(t-s)\Delta _{H}}\left(
2r(e^{s(\Delta _{H}+2r)}u_{0}^{\flat })\right) \,ds,
\end{equation*}%
and then%
\begin{align}
\Vert \nabla e^{t(\Delta _{H}+2r)}u_{0}^{\flat }\Vert _{g,p,\lambda }& \leq
\Vert \nabla e^{t\Delta _{H}}u_{0}^{\flat }\Vert _{g,p,\lambda }+  \notag \\
& \hspace{0.7cm}+\int_{0}^{t}\bigg\|\nabla e^{(t-s)\Delta _{H}}\left(
2r(e^{s(\Delta _{H}+2r)}u_{0}^{\flat })\right) \bigg\|_{g,p,\lambda }\,ds
\notag \\
& \leq \dfrac{C}{\sqrt{t}}\Vert u_{0}\Vert _{g,p,\lambda }+C\int_{0}^{t}%
\dfrac{1}{\sqrt{t-s}}\Vert r(e^{s(\Delta _{H}+2r)}u_{0}^{\flat })\Vert
_{g,p,\lambda }\,ds.  \label{aux-proof-30}
\end{align}%
By the boundedness assumption on the Ricci curvature, it follows that
\begin{align*}
\Vert r(e^{s(\Delta _{H}+2r)}u_{0}^{\flat })\Vert _{g,p,\lambda }& \leq
C\Vert e^{s(\Delta _{H}+2r)}u_{0}^{\flat }\Vert _{g,p,\lambda } \\
& \leq C\Vert u_{0}\Vert _{g,p,\lambda }.
\end{align*}%
Considering this estimate into (\ref{aux-proof-30}), we obtain
\begin{equation}
\Vert \nabla e^{t(\Delta _{H}+2r)}u_{0}^{\flat }\Vert _{g,p,\lambda }\leq
\dfrac{C}{\sqrt{t}}\Vert u_{0}\Vert _{g,p,\lambda }.  \label{aux-proof-31}
\end{equation}%
For $t>1$, using the fact
\begin{equation*}
\nabla e^{t(\overrightarrow{\Delta }+r)}=\nabla e^{\overrightarrow{\Delta }%
+r}\left( e^{(t-1)(\overrightarrow{\Delta }+r)}\right) ,
\end{equation*}%
estimate (\ref{aux-proof-31}) with $t=1,$ and Theorem \ref{th-cd}, we can
estimate
\begin{align}
\Vert \nabla e^{t(\overrightarrow{\Delta }+r)}u_{0}\Vert _{g,p,\lambda }&
=\left\Vert \nabla e^{\overrightarrow{\Delta }+r}\left( e^{(t-1)(%
\overrightarrow{\Delta }+r)}\right) u_{0}\right\Vert _{g,p,\lambda }  \notag
\\
& \leq C\left\Vert e^{(t-1)(\overrightarrow{\Delta }+r)}u_{0}\right\Vert
_{g,p,\lambda }  \notag \\
& \leq Ce^{-c_{0}t-\alpha _{p}t+k\frac{\lambda }{mp}\gamma _{g}t}\left\Vert
u_{0}\right\Vert _{g,p,\lambda }.  \label{aux-proof-32}
\end{align}%
When $p<q$, we can invoke Theorem \ref{th-cd} in conjunction with the
semigroup decomposition%
\begin{equation*}
\nabla e^{t(\overrightarrow{\Delta }+r)}u_{0}=\nabla e^{\frac{t}{2}(%
\overrightarrow{\Delta }+r)}\left( e^{\frac{t}{2}(\overrightarrow{\Delta }%
+r)}u_{0}\right) .
\end{equation*}%
This decomposition enables us to combine the dispersive estimate for $e^{%
\frac{t}{2}(\overrightarrow{\Delta }+r)}$ (see Theorem \ref{th-cd}), which
maps $\mathcal{M}_{p,\lambda }^{g}$ to $\mathcal{M}_{q,\lambda }^{g}$ with
the gradient bound \ref{aux-proof-32} derived above. In this way, we obtain
the desired operator bound for $\nabla e^{t(\overrightarrow{\Delta }+r)}$,
actiong from $\mathcal{M}_{p,\lambda }^{g}$ to $\mathcal{M}_{q,\lambda }^{g}$%
.

Finally, we note that the proof for the case $u_{0}\in \mathcal{M}%
_{p,\lambda }$ (item $(ii)$) follows similarly and is left to the reader. %
\fin

\begin{rem}
\hspace{0.1cm} \label{r-sm}

\begin{enumerate}
\item[$(i)$] Under the hypotheses of Theorem \ref{th-s-ric}, for $\lambda $
sufficiently small, there exist $C>0$ and $\beta _{p,q}>0$ such that
\begin{equation*}
\Vert \nabla e^{t(\overrightarrow{\Delta }+r)}u_{0}\Vert _{g,q,\lambda }\leq
C\left[ d(t)\right] ^{-\frac{1}{2}-\frac{m}{2}\left( \frac{1}{p}-\frac{1}{q}%
\right) }t^{\frac{\lambda }{2}\left( \frac{1}{p}-\frac{1}{q}\right)
}e^{-\beta _{p,q}t}\Vert u_{0}\Vert _{g,p,\lambda },\hspace{0.3cm}\text{for
all }t>0;
\end{equation*}%
while, for each $\lambda \in \lbrack 0,m),$ we have the estimate
\begin{equation*}
\Vert \nabla e^{t(\overrightarrow{\Delta }+r)}u_{0}\Vert _{q,\lambda }\leq C%
\left[ d(t)\right] ^{-\frac{1}{2}-\frac{m}{2}\left( \frac{1}{p}-\frac{1}{q}%
\right) }t^{\frac{\lambda }{2}\left( \frac{1}{p}-\frac{1}{q}\right) }\left(
1+\gamma t\right) ^{\frac{\lambda }{q}}e^{-\beta _{p,q}t}\Vert u_{0}\Vert
_{p,\lambda },\text{ for all }t>0.
\end{equation*}

\item[$(ii)$] With appropriate modifications, the estimates derived in this
section can also be established for the Morrey spaces $\mathcal{M}%
_{p,\lambda }^{K}$ and $\mathcal{M}_{p,\lambda }^{e}$, as defined in
Definition \ref{morrey}.

\item[$(iii)$] By combining the comparison between the semigroups associated
with the Hodge and Beltrami Laplacians (see \eqref{hb}) with the pointwise
estimate of the heat kernel (see \eqref{es-ho}), we can derive analogous
versions of the estimates established in this section for the semigroup $%
e^{t\Delta _{H}} $ and its gradient $\nabla e^{t\Delta _{H}}$ associated
with the Hodge operator.\vspace{0.5cm}
\end{enumerate}
\end{rem}

\section{Riesz transform}

\label{s5}

In this section, we study the Riesz transform $\nabla (-\Delta _{g})^{-\frac{%
1}{2}}$ under the hypotheses previously introduced, relying on the estimates
already established for the associated semigroup. As in the Euclidean
setting, on manifolds without boundary the corresponding Leray--Helmholtz
projection $\mathbb{P}$ can be analyzed through the Riesz transform. In
particular, in the context of the Navier-Stokes equations (\ref{ns-2}), we
are naturally led to consider the operator
\begin{equation*}
Bu=-2\,\func{grad}(-\Delta _{g})^{-1}\func{div}(r(u)),
\end{equation*}%
as well as the projection operator
\begin{equation*}
\mathbb{P}v=v+\func{grad}(-\Delta _{g})^{-1}\func{div}(v).
\end{equation*}%
Then, the boundedness of the Riesz transform will play a fundamental role.
To establish this property, we need the formal identity (see \cite%
{Strichartz} and \cite{Auscher2004})
\begin{equation}
\nabla (-\Delta _{g})^{-\frac{1}{2}}=\Gamma \left( \dfrac{1}{2}\right)
^{-1}\int_{0}^{\infty }t^{-\frac{1}{2}}\nabla e^{t\Delta _{g}}\,dt.
\label{ri}
\end{equation}%
Notice that the boundedness of the Riesz transform $\nabla (-\Delta _{g})^{-%
\frac{1}{2}},$ together with the Ricci curvature condition, implies the
boundedness of both $B$ and $\mathbb{P}$. Indeed, for the projection
operator, we can write
\begin{align*}
\mathbb{P}& =I+\func{grad}(-\Delta _{g})^{-1}\func{div} \\
& =I+\func{grad}(-\Delta _{g})^{-\frac{1}{2}}(-\Delta _{g})^{-\frac{1}{2}}%
\func{div}.
\end{align*}%
Thus, to control $\mathbb{P}$, it suffices to understand the operators $%
(-\Delta _{g})^{-\frac{1}{2}}\func{div}$ and $\func{grad}(-\Delta _{g})^{-%
\frac{1}{2}}.$ In turn, observe that estimating these two operators is
equivalent to studying the Riesz transform $\nabla (-\Delta _{g})^{-\frac{1}{%
2}}$. In fact, we have that
\begin{align*}
(-\Delta _{g})^{-\frac{1}{2}}\func{div}(u)& =\Gamma \left( \dfrac{1}{2}%
\right) ^{-1}\int_{0}^{\infty }t^{-\frac{1}{2}}e^{t\Delta _{g}}\func{div}%
(u)\,dt \\
& =\Gamma \left( \dfrac{1}{2}\right) ^{-1}\int_{0}^{\infty }t^{-\frac{1}{2}%
}\int_{M}G(t,x,y)\func{div}(u(y))\,dy\,dt \\
& =-\Gamma \left( \dfrac{1}{2}\right) ^{-1}\int_{0}^{\infty }t^{-\frac{1}{2}%
}\int_{M}\nabla _{y}G(t,x,y)\cdot u(y)\,dy\,dt
\end{align*}%
and
\begin{align*}
\nabla (-\Delta _{g})^{-\frac{1}{2}}u& =\Gamma \left( \dfrac{1}{2}\right)
^{-1}\int_{0}^{\infty }t^{-\frac{1}{2}}\nabla e^{t\Delta _{g}}u\,dt \\
& =\Gamma \left( \dfrac{1}{2}\right) ^{-1}\int_{0}^{\infty }t^{-\frac{1}{2}%
}\int_{M}\nabla _{x}G(t,x,y)u(y)\,dy\,dt.
\end{align*}%
Therefore, to get the boundedness of $\mathbb{P}$, it is enough to obtain
the boundedness of the Riesz transform, and similarly for the operator $B$.

The following theorem establishes the boundedness of $\nabla (-\Delta
_{g})^{-\frac{1}{2}}$. The result can also be adapted to the spaces $%
\mathcal{M}_{p,\lambda }^{K}$ and $\mathcal{M}_{p,\lambda }^{e}$.

\begin{thr}
\label{th-ri}Let $(M,g)$ be a complete, simply connected, non-compact, $m$%
-dimensional Riemannian manifold without boundary.

\begin{enumerate}
\item[$(i)$] If $M$ has bounded geometry, $-K(m-1)\leq \func{Ric}\leq
-c_{0}<0,$ and $\kappa \leq 0$, then the Riesz transform $\nabla (-\Delta
_{g})^{-\frac{1}{2}}:\mathcal{M}_{p,\lambda }^{g}(M)\rightarrow \mathcal{M}%
_{p,\lambda }^{g}(\Gamma (TM))$ is bounded for every $1<p<\infty $ and every
sufficiently small $\lambda \in \lbrack 0,m)$; that is, there exists a
constant $C>0$ such that
\begin{equation*}
\Vert \nabla (-\Delta _{g})^{-\frac{1}{2}}u\Vert _{g,p,\lambda }\leq C\Vert
u\Vert _{g,p,\lambda },\text{ \ \ \ }\forall u\in \mathcal{M}_{p,\lambda
}^{g}(M).
\end{equation*}

\item[$(ii)$] If $M$ has bounded geometry, $\func{Ric}\geq -K(m-1)$, $K>0,$
and $\kappa \leq 0$, then the Riesz transform $\nabla (-\Delta _{g})^{-\frac{%
1}{2}}:\mathcal{M}_{p,\lambda }(M)\rightarrow \mathcal{M}_{p,\lambda
}(\Gamma (TM))$ is bounded for every $1<p<\infty $ and $\lambda \in \lbrack
0,m)$; that is, there exists a constant $C>0$ such that
\begin{equation*}
\Vert \nabla (-\Delta _{g})^{-\frac{1}{2}}u\Vert _{p,\lambda }\leq C\Vert
u\Vert _{p,\lambda },\text{ \ \ \ }\forall u\in \mathcal{M}_{p,\lambda }(M).
\end{equation*}
\end{enumerate}
\end{thr}

\noindent {\textbf{Proof.}} Splitting the integral expression \eqref{ri}, we
can write
\begin{align*}
\nabla (-\Delta _{g})^{-\frac{1}{2}}u=\Gamma \left( \dfrac{1}{2}\right) ^{-1}%
\bigg[& \int_{0}^{\infty }t^{-\frac{1}{2}}\int_{r(x,y)\leq 1}\nabla
_{x}G(t,x,y)u(y)\,dx\,dt+ \\
& +\int_{0}^{1}t^{-\frac{1}{2}}\int_{r(x,y)\geq 1}\nabla
_{x}G(t,x,y)u(y)\,dx\,dt+ \\
& +\int_{1}^{\infty }t^{-\frac{1}{2}}\int_{r(x,y)\geq 1}\nabla
_{x}G(t,x,y)u(y)\,dx\,dt\bigg] \\
& :=I_{1}+I_{2}+I_{3}.
\end{align*}%
For the first integral, we can use the $L^{p}$ boundedness \cite[Corollaire
1.2]{Lohou}. Then, for $R>1,$ we obtain
\begin{align}
\int_{B_{x_{0}}(R)}|I_{1}|^{p}& \leq C\Vert \nabla (-\Delta _{g})^{-\frac{1}{%
2}}u\chi _{B_{x_{0}}(R+1)}\Vert _{p}^{p}  \notag \\
& \leq C|B_{x_{0}}(R+1)|^{\frac{\lambda }{m}}\Vert u\Vert _{g,p,\lambda }^{p}
\notag \\
& \leq C|B_{x_{0}}(R)|^{\frac{\lambda }{m}}\Vert u\Vert _{g,p,\lambda }^{p}.
\label{eq-r1}
\end{align}%
For $R\leq 1$, we apply Tonelli's theorem to estimate%
\begin{align*}
\bigg(\int_{B_{x_{0}}(R)}& |I_{1}|^{p}\bigg)^{\frac{1}{p}}\leq \\
& \leq C\Vert \nabla (-\Delta _{g})^{-\frac{1}{2}}u\chi
_{B_{x_{0}}(2R)}\Vert _{p}+C\left[ \int_{B_{x_{0}}(R)}\left(
\int_{0}^{\infty }\int_{R\leq r(x,y)\leq 1}t^{-\frac{m}{2}-1}e^{-C_{2}\frac{%
r^{2}}{t}}|u(y)|\,dy\,dt\right) ^{p}\,dx\right] ^{\frac{1}{p}} \\
& \leq C|B_{x_{0}}(2R)|^{\frac{\lambda }{mp}}\Vert u\Vert _{g,p,\lambda }+C%
\left[ \int_{B_{x_{0}}(R)}\left( \int_{R\leq r(x,y)\leq 1}\left(
\int_{0}^{\infty }t^{-\frac{m}{2}-1}e^{-C_{2}\frac{r^{2}}{t}}\,dt\right)
|u(y)|\,dy\right) ^{p}\,dx\right] ^{\frac{1}{p}} \\
& \leq C|B_{x_{0}}(R)|^{\frac{\lambda }{mp}}\Vert u\Vert _{g,p,\lambda }+C%
\left[ \int_{B_{x_{0}}(R)}\left( \int_{R\leq r(x,y)\leq
1}r(x,y)^{-m}|u(y)|\,dy\right) ^{p}\,dx\right] ^{\frac{1}{p}}.
\end{align*}%
The integral in the last line can be treated as follows. Let $p^{\prime }$
be the conjugate exponent of $p$ and choose $\alpha ,\beta >0$ such that
\begin{equation*}
\dfrac{\alpha }{p}+\dfrac{\beta }{p^{\prime }}=m\text{ with }\,\alpha
>\lambda ,\,\beta >m.
\end{equation*}%
Then
\begin{align*}
& \int_{R\leq r(x,y)\leq 1}r(x,y)^{-\left( \frac{\alpha }{p}+\frac{\beta }{%
p^{\prime }}\right) }|u(y)|\,dy \\
& \leq C\left( \int_{R\leq r(x,y)\leq 1}r(x,y)^{-\beta }\,dy\right) ^{\frac{1%
}{p^{\prime }}}\left( \int_{R\leq r(x,y)\leq 1}r(x,y)^{-\alpha
}|u(y)|^{p}\,dy\right) ^{\frac{1}{p}} \\
& \leq CR^{\frac{m-\beta }{p^{\prime }}}R^{\frac{\lambda -\alpha }{p}}\Vert
u\Vert _{g,p,\lambda }=CR^{-\frac{m-\lambda }{p}}\Vert u\Vert _{g,p,\lambda
}.
\end{align*}%
Hence, for $R\leq 1,$ we obtain
\begin{align}
\bigg(\int_{B_{x_{0}}(R)}|I_{1}|^{p}\bigg)^{\frac{1}{p}}& \leq
C|B_{x_{0}}(R)|^{\frac{\lambda }{mp}}\Vert u\Vert _{g,p,\lambda }+C\left[
\int_{B_{x_{0}}(R)}\left( R^{-\frac{m-\lambda }{p}}\Vert u\Vert
_{g,p,\lambda }\right) ^{p}\,dx\right] ^{\frac{1}{p}}  \notag \\
& \leq C|B_{x_{0}}(R)|^{\frac{\lambda }{mp}}\Vert u\Vert _{g,p,\lambda }+CR^{%
\frac{m}{p}-\frac{m-\lambda }{p}}\Vert u\Vert _{g,p,\lambda }  \notag \\
& \leq C|B_{x_{0}}(R)|^{\frac{\lambda }{mp}}\Vert u\Vert _{g,p,\lambda }.
\label{eq-r2}
\end{align}%
For the second integral $I_{2}$, in view of the heat kernel estimate %
\eqref{es-ho}, we can estimate
\begin{align*}
|I_{2}|& \leq \int_{0}^{1}t^{-\frac{m}{2}-1}\int_{r(x,y)\geq 1}e^{-C_{2}%
\frac{r^{2}}{t}}|u(y)|\,dy\,dt \\
& \leq \int_{0}^{1}t^{-\frac{m}{2}-1}e^{-\frac{C}{t}}\int_{r(x,y)\geq 1}e^{-C%
\frac{r^{2}}{t}}|u(y)|\,dy\,dt \\
& \leq \int_{0}^{1}t^{-\frac{m}{2}-1}e^{-\frac{C}{t}}\,dt\int_{r(x,y)\geq
1}e^{-C{r^{2}}}|u(y)|\,dy \\
& \leq C\int_{M}e^{-C{r(x,y)^{2}}}|u(y)|\,dy.
\end{align*}%
Using the strategy of Theorem \ref{th-mm} in the estimate%
\begin{equation*}
\int_{B_{x_{0}}(R)}|I_{2}|^{p}\leq \int_{B_{x_{0}}(R)}\left( \int_{M}e^{-C{%
r(x,y)^{2}}}|u(y)|\,dy\right) ^{p}\,dx,
\end{equation*}%
we get
\begin{equation}
\Vert I_{2}\Vert _{g,p,\lambda }\leq C\Vert u\Vert _{g,p,\lambda }.
\label{eq-r3}
\end{equation}%
For the third term $I_{3},$ we have
\begin{equation*}
\Vert I_{3}\Vert _{g,p,\lambda }\leq C\int_{1}^{\infty }t^{-\frac{1}{2}%
}\Vert \nabla e^{t\Delta _{g}}u\Vert _{g,p,\lambda }\,dt.
\end{equation*}%
Using the smoothing estimate of Theorem \ref{es-s} and Remark \ref{r-sm}, it
follows that
\begin{align}
\int_{1}^{\infty }t^{-\frac{1}{2}}\Vert \nabla e^{t\Delta _{g}}u\Vert
_{g,p,\lambda }\,dt& \leq C\left[ \int_{1}^{\infty }e^{-\beta _{p,p}t}\,dt%
\right] \Vert u(x)\Vert _{g,p,\lambda }  \notag \\
& \leq C\Vert u(x)\Vert _{g,p,\lambda }.  \label{eq-r4}
\end{align}%
Finally, combining the estimates \eqref{eq-r1}, \eqref{eq-r2}, \eqref{eq-r3}%
, and \eqref{eq-r4}, we conclude
\begin{equation*}
\Vert \nabla (-\Delta _{g})^{-\frac{1}{2}}u\Vert _{g,p,\lambda }\leq C\Vert
u\Vert _{g,p,\lambda },
\end{equation*}%
as requested. The proof for the case $u\in \mathcal{M}_{p,\lambda }$ (item $%
(ii)$) is similar and is left to the reader.\fin

\section{Navier-Stokes equations}

\label{s6}

The ultimate goal is to establish the well-posedness of the principal
problem, namely the Navier-Stokes equations, by combining the estimates we
have derived in the previous section for the associated semigroups and the
Riesz transform. For simplicity, we present the results in the Morrey space $%
\mathcal{M}_{p,\lambda }^{g}$, as the cases of $\mathcal{M}_{p,\lambda }$, $%
\mathcal{M}_{p,\lambda }^{K}$, and $\mathcal{M}_{p,\lambda }^{e}$ are
analogous. To implement a contraction argument, we will employ a well-known
variant of the classical Banach fixed point theorem for abstract equations
(see, e.g., \cite[Lemma 3.7]{fer}).

\begin{lem}
\label{fix} Let $X$ be a Banach space, $\mathcal{N}$ $:X\times X\rightarrow
X $ a bilinear operator, and $\mathcal{T}$ $:X\rightarrow X$ a linear
operator such that, for all $u,v\in X$, there exist constants $0\leq C_{1}<1$
and $C_{2}>0$ satisfying
\begin{equation*}
\Vert \mathcal{N}(u,v)\Vert _{X}\leq C_{2}\Vert u\Vert _{X}\Vert v\Vert _{X}%
\text{ and }\Vert \mathcal{T}u\Vert _{X}\leq C_{1}\Vert u\Vert _{X}.
\end{equation*}%
Then, for every $u_{1}\in X$ such that
\begin{equation*}
\Vert u_{1}\Vert _{X}\leq \varepsilon <\frac{(1-C_{1})^{2}}{4C_{2}},
\end{equation*}%
the sequence defined by
\begin{equation*}
u_{n+1}=u_{1}+\mathcal{N}(u_{n},u_{n})+\mathcal{T}u_{n}
\end{equation*}%
converges to the unique solution of
\begin{equation*}
u=u_{1}+\mathcal{N}(u,u)+\mathcal{T}u\text{ \ \ \ with }\Vert u\Vert
_{X}\leq \frac{2\varepsilon }{1-C_{1}}.
\end{equation*}
\end{lem}

For the well-posedness results, we will consider the hypotheses on the
manifold and the restriction that $\lambda \in \lbrack 0,m)$ is sufficiently
small so that the exponential decay required for the fixed point method is
meaningful. With this in mind, we introduce the following definition for the
results that follow.

\begin{df}
\label{df-m} Let $(M,g)$ be a complete, simply connected, non-compact, $m$%
-dimensional Riemannian manifold without boundary, satisfying the following
assumptions:

\begin{enumerate}
\item[$(i)$] $M$ has bounded geometry;

\item[$(ii)$] Ricci curvature bounds:\subitem$(ii)(a)$ $\func{Ric}\geq
-K(m-1) $, $K>0$; \subitem$(ii)(b)$ $\func{Ric}\leq -c_{0}$, $c_{0}>0$;

\item[$(iii)$] Negative sectional curvature: $\kappa \leq 0$;

\item[$(iv)$] Smallness on $\lambda $: the parameter $\lambda $ is
sufficiently small so that
\begin{equation*}
k^{2}\frac{\lambda }{mp}\left( 1+\frac{\lambda }{m}\right) \left( \frac{1}{4}%
-c\right) ^{-1}\leq \min \left\{ \frac{4\delta _{m}(p-1)}{p^{2}},\frac{%
\lambda _{1}}{p}\right\} ,
\end{equation*}
where $k=\sqrt{K}(m-1),$ $0<c<\frac{1}{4},$ and $\delta _{m}\geq \frac{1}{4}%
\left( c_{0}^{2}+(m-1)(m-2)\kappa ^{\ast }\right) $ with $\kappa ^{\ast
}=\sup_{M}\kappa .$
\end{enumerate}
\end{df}

\subsection{A general class of non-compact manifolds}

\label{s6.2}In this section, we analyze the local-in-time well-posedness of
the Navier-Stokes equations on manifolds that are more general than the
Einstein case. Moreover, we also study the Navier-Stokes system with
modified viscosity.

We consider the more general framework introduced in Definition \ref{df-m}.
In this setting, the operator $B$ does not necessarily vanish. This
additional term creates the difficulty to argue the global-in-time
well-posedness; however, we can obtain the local-in-time well-posedness.

Let us consider the equations
\begin{equation}
\begin{cases}
\partial _{t}u-\overrightarrow{\Delta }u=-\mathbb{P}(\func{div}(u\otimes
u))-Bu+r(u), \\
\func{div}(u)=0, \\
u(0,\cdot )=u_{0}.%
\end{cases}
\label{ns-b}
\end{equation}%
Notice that the operator $(\overrightarrow{\Delta }+r-B)$ does not commute
with the operator $\mathbb{P}$ on $M$. For this reason, we choose to treat
the term $(B-r)u$ together with the projection, in view of the abstract
formulation of the equation given in Lemma \ref{fix}. In view of Duhamel's
principle, system (\ref{ns-b}) is formally equivalent to the integral
equation%
\begin{equation}
u(t)=e^{t\overrightarrow{\Delta }}u_{0}-\int_{0}^{t}e^{(t-\tau )%
\overrightarrow{\Delta }}\mathbb{P}(\func{div}(u\otimes u))(\tau )\,d\tau
-\int_{0}^{t}e^{(t-\tau )\overrightarrow{\Delta }}(B-r)u(\tau )\,d\tau .
\label{int-eqb}
\end{equation}%
Solutions of (\ref{int-eqb}) are called mild solutions of system (\ref{ns-b})

Let $p,q\in \lbrack 1,\infty )$ and $\lambda \in \lbrack 0,m)$ satisfy $%
m-\lambda \leq p<q$, and let $T\in (0,\infty ]$ be an existence time.
Consider the space $\mathcal{C}_{b}(0,T),Z)$ of the bounded continuous
functions from $(0,T)$ to $Z$. \ We seek mild solutions of system (\ref{ns-b}%
) in the functional space
\begin{equation}
CX_{T}:=CM_{T}\cap X_{T}  \label{space-th-local-1}
\end{equation}%
with $CM_{T}=\mathcal{C}_{b}((0,T),\mathcal{M}_{p,\lambda }^{g}(\Gamma
(TM))) $ and
\begin{equation}
X_{T}=\bigg\{u\in L_{loc}^{\infty }((0,T),\mathcal{M}_{q,\lambda
}^{g}(\Gamma (TM)))\,\bigg|\,\left[ d(t)^{\frac{m}{2}}t^{-\frac{\lambda }{2}}%
\right] ^{\left( \frac{1}{p}-\frac{1}{q}\right) }e^{\beta t}\Vert u(t,\cdot
)\Vert _{g,q,\lambda }\in L^{\infty }(0,T)\bigg\},  \label{space1}
\end{equation}%
where the parameter $\beta $ is chosen according to the large-time decay
rate in the smoothing estimate (see Remark \ref{r-sm}). The space $CX_{T}$
is a Banach space when endowed with the norm $\Vert \cdot \Vert
_{CX_{T}}=\Vert \cdot \Vert _{CM_{T}}+\Vert \cdot \Vert _{X_{T}}$ where
\begin{equation}
\Vert u\Vert _{CM_{T}}=\sup_{0<t<{T}}\Vert u(t,\cdot )\Vert _{g,p,\lambda }%
\text{ and }\Vert u\Vert _{X_{T}}=\sup_{0<t<{T}}\left\{ \left[ d(t)^{\frac{m%
}{2}}t^{-\frac{\lambda }{2}}\right] ^{\left( \frac{1}{p}-\frac{1}{q}\right)
}e^{\beta t}\Vert u(t,\cdot )\Vert _{g,q,\lambda }\right\} \text{. }
\label{norm1}
\end{equation}

\begin{thr}[Local well-posedness]
\label{th-local} Let $(M,g)$ be a manifold as in Definition \ref{df-m},
except for condition $(iv)$. Suppose that $p,q\in \lbrack 1,\infty )$ and $%
\lambda \in \lbrack 0,m)$ satisfy $m-\lambda \leq p<q.$ Assume that $%
u_{0}\in \mathcal{M}_{p,\lambda }^{g}(\Gamma (TM))$ with $\func{div}%
(u_{0})=0 $ and $\limsup_{T\rightarrow 0}\Vert e^{t\overrightarrow{\Delta }%
}u_{0}\Vert _{X_{T}}<\varepsilon $ for some sufficiently small $\varepsilon
>0$. Then, there exists $T_{0}>0$ such that the Navier-Stokes equations (\ref%
{ns-b}) admit a unique mild solution $u$ belonging to the space $CX_{T_{0}}$%
. Moreover, the solution $u\rightharpoonup u_{0}$ in the sense of
distributions on $M$ as $t\rightarrow 0^{+}$.
\end{thr}

\noindent {\textbf{Proof.}} In view of the mild formulation (\ref{int-eqb}),
we can write
\begin{equation*}
u=u_{1}+\mathcal{N}(u,u)(t)+\mathcal{T}(u)(t),
\end{equation*}%
where $u_{1}=e^{t\overrightarrow{\Delta }}u_{0},$
\begin{equation*}
\mathcal{T}(u)(t)=-\int_{0}^{t}e^{(t-\tau )\overrightarrow{\Delta }%
}(B-r)u(\tau )\,d\tau ,
\end{equation*}%
and
\begin{equation*}
\mathcal{N}(u,u)(t)=-\int_{0}^{t}e^{(t-\tau )\overrightarrow{\Delta }}%
\mathbb{P}(\func{div}(u\otimes u))(\tau )\,d\tau .
\end{equation*}

Next, note that the operator $\overrightarrow{\Delta }$ has the property of
commuting with the projection $\mathbb{P}$ as long as $\partial
M=\varnothing $. Consequently, we can first apply the smoothing estimates
and then invoke the boundedness of $\mathbb{P}$ (see Theorem \ref{th-ri}).
Setting%
\begin{equation*}
\mathcal{S}(u,u)(t)=a_{1}\mathcal{N}(u,u)(t)+a_{2}\mathcal{T}(u)(t),\text{
for }a_{i}\in \{0,1\},
\end{equation*}%
we estimate
\begin{align*}
\Vert \mathcal{S}(u,u)& (t)\Vert _{g,q,\lambda }\leq
C_{2}a_{1}\int_{0}^{t}[d(t-\tau )]^{-\left( \frac{1}{2}+\frac{m}{2q}\right)
}(t-\tau )^{\frac{\lambda }{2q}}e^{-\beta (t-\tau )}\Vert u\otimes u(\tau
)\Vert _{g,q/2,\lambda }\,d\tau \\
& +C_{1}a_{2}\int_{0}^{t}e^{-\beta (t-\tau )}\Vert u(\tau )\Vert
_{g,q,\lambda }\,d\tau \\
\leq & \,C_{2}a_{1}\int_{0}^{t}[d(t-\tau )]^{-\left( \frac{1}{2}+\frac{m}{2q}%
\right) }(t-\tau )^{\frac{\lambda }{2q}}e^{-\beta (t-\tau )}\left( \left[
d(\tau )^{-\frac{m}{2}}\tau ^{\frac{\lambda }{2}}\right] ^{\left( \frac{1}{p}%
-\frac{1}{q}\right) }\right) ^{2}\,d\tau \Vert u\Vert _{X_{T}}^{2} \\
& +C_{1}a_{2}\int_{0}^{t}\left[ d(\tau )^{-\frac{m}{2}}\tau ^{\frac{\lambda
}{2}}\right] ^{\left( \frac{1}{p}-\frac{1}{q}\right) }e^{-\beta t}\,d\tau
\Vert u\Vert _{X_{T}} \\
\leq & C_{2}e^{-\beta t}a_{1}\int_{0}^{t}[d(t-\tau )]^{-\left( \frac{1}{2}+%
\frac{m}{2q}\right) }(t-\tau )^{\frac{\lambda }{2q}}\left[ d(\tau )^{-\frac{m%
}{2}}\tau ^{\frac{\lambda }{2}}\right] ^{2\left( \frac{1}{p}-\frac{1}{q}%
\right) }e^{-\beta \tau }\,d\tau \Vert u\Vert _{X_{T}}^{2} \\
& +C_{1}a_{2}\int_{0}^{t}\left[ d(\tau )^{-\frac{m}{2}}\tau ^{\frac{\lambda
}{2}}\right] ^{\left( \frac{1}{p}-\frac{1}{q}\right) }e^{-\beta t}\,d\tau
\Vert u\Vert _{X_{T}}.
\end{align*}%
Using the condition $q>p\geq m-\lambda ,$ we arrive at
\begin{equation}
\Vert \mathcal{S}(u,u)\Vert _{X_{T}}\leq C_{4}(T)a_{1}\Vert u\Vert
_{X_{T}}^{2}+C_{3}(T)a_{2}\Vert u\Vert _{X_{T}}\text{, for }a_{i}\in \{0,1\}.
\label{aux-proof-th-local-1}
\end{equation}%
Note that we can take $C_{3}(T_{0})<1$, because
\begin{equation*}
C_{1}\int_{0}^{T_{0}}\left[ d(\tau )^{-\frac{m}{2}}\tau ^{\frac{\lambda }{2}}%
\right] ^{\left( \frac{1}{p}-\frac{1}{q}\right) }e^{-\beta t}\,d\tau <1,%
\text{ for some }T_{0}>0.
\end{equation*}%
Then, by assumption $\lim_{T\rightarrow 0}\Vert e^{t\overrightarrow{\Delta }%
}u_{0}\Vert _{X_{T}}<\varepsilon ,$ estimate (\ref{aux-proof-th-local-1}),
and Lemma \ref{fix}, we obtain the local-well-posedness in the space $%
X_{T_{0}}$.

It remains to show that the unique local solution satisfies $u\in CM_{T_{0}}=%
\mathcal{C}_{b}((0,T_{0}),\mathcal{M}_{p,\lambda }^{g}(\Gamma (TM)))$ for
some $T_{0}>0$. For that, we estimate the $CM_{T}$-component of the norm of $%
CX_{T}$. Choosing $\frac{1}{r}=\frac{1}{p}+\frac{1}{q}$ and using H\"{o}%
lder's inequality (\ref{holder}), we obtain
\begin{align*}
\Vert \mathcal{S}& (u,u)(t)\Vert _{g,p,\lambda }\leq
C_{2}a_{1}\int_{0}^{t}[d(t-\tau )]^{-\left( \frac{1}{2}+\frac{m}{2}\left(
\frac{1}{r}-\frac{1}{p}\right) \right) }(t-\tau )^{\frac{\lambda }{2}\left(
\frac{1}{r}-\frac{1}{p}\right) }e^{-\beta (t-\tau )}\Vert u\otimes u(\tau
)\Vert _{g,r,\lambda }\,d\tau \\
& \hspace{1cm}+C_{1}a_{2}\int_{0}^{t}e^{-\beta (t-\tau )}\Vert u(\tau )\Vert
_{g,p,\lambda }\,d\tau \\
& \leq C_{2}a_{1}\Vert u\Vert _{X_{T}}\Vert u\Vert
_{CM_{T}}\int_{0}^{t}[d(t-\tau )]^{-\left( \frac{1}{2}+\frac{m}{2q}\right)
}(t-\tau )^{\frac{\lambda }{2q}}[d(\tau )^{-\frac{m}{2}}\tau ^{\frac{\lambda
}{2}}]^{\left( \frac{1}{p}-\frac{1}{q}\right) }e^{-\beta t}\,d\tau \\
& \hspace{1cm}+C_{1}a_{2}\int_{0}^{t}e^{-\beta (t-\tau )}\Vert u(\tau )\Vert
_{g,p,\lambda }\,d\tau \\
& \leq C_{6}(T)a_{1}\Vert u\Vert _{X_{T}}\Vert u\Vert
_{CM_{T}}+C_{5}(T)a_{2}\Vert u\Vert _{CM_{T}}\text{, for }a_{i}\in \{0,1\}.
\end{align*}%
As before, by possibly reducing $T_{0}$, we may ensure that $C_{5}(T_{0})<1$%
. Moreover, by Lemma \ref{fix} and (\ref{aux-proof-th-local-1}), the
solution $u$ is obtained as the limit of the sequence $\{u_{n}\}_{n\in
\mathbb{N}}$ defined by
\begin{equation*}
u_{n+1}(t)=u_{1}(t)+\mathcal{N}(u_{n},u_{n})(t)+\mathcal{T}u_{n}(t),
\end{equation*}%
where
\begin{equation*}
\Vert u_{n}\Vert _{X_{T}}\leq \frac{2\varepsilon }{1-C_{3}(T_{0})}\text{ for
}n\geq 1.
\end{equation*}%
So, we can estimate
\begin{align*}
\Vert u_{n+1}(t)-u_{n}(t)\Vert _{g,p,\lambda }& \leq \Vert \mathcal{N}%
(u_{n},u_{n})(t)-\mathcal{N}(u_{n-1},u_{n-1})(t)\Vert _{g,p,\lambda }+\Vert
\mathcal{T}\left( u_{n}(t)-u_{n-1}(t)\right) \Vert _{g,p,\lambda } \\
& \leq C_{6}(T_{0})\Vert u_{n}-u_{n-1}\Vert _{CM_{T_{0}}}\left( \Vert
u_{n}\Vert _{X_{T_{0}}}+\Vert u_{n-1}\Vert _{X_{T_{0}}}\right)
+C_{5}(T_{0})\Vert u_{n}-u_{n-1}\Vert _{CM_{T_{0}}} \\
& \leq \left[ \frac{4\varepsilon }{1-C_{3}(T_{0})}C_{6}(T_{0})+C_{5}(T_{0})%
\right] \Vert u_{n}-u_{n-1}\Vert _{CM_{T_{0}}}.
\end{align*}%
By reducing $\varepsilon >0$ and $T_{0}>0$ (if necessary), we ensure that
\begin{equation*}
C_{9}(T_{0})=\frac{4\varepsilon }{1-C_{3}(T_{0})}C_{6}(T_{0})+C_{5}(T_{0})<1.
\end{equation*}
It follows that $\{u_{n}\}_{n\in \mathbb{N}}$ is a Cauchy sequence in $%
CM_{T_{0}}$and hence the limit $u\in CM_{T_{0}}$, as desired.

To conclude, we observe that time-continuity for $t\in (0,T_{0})$ and the
convergence $u\rightharpoonup u_{0}$ in $\mathcal{D}^{\prime }(M)$ as $%
t\rightarrow 0^{+}$ follow from standard arguments in the theory of mild
solutions and abstract semigroups (see, e.g., \cite{Fer2},\cite{Pazy}). We
omit the details and leave them to the reader.\fin

\subsection{The modified case with viscosity}

\label{s.visc} The problem with the viscosity \eqref{ns1} can be modified as
follows
\begin{equation*}
\begin{cases}
\partial _{t}u+\nabla _{u}u+\dfrac{1}{\rho }\func{grad}\,p=\mathcal{L}_{\nu
_{1},\nu _{2}}(u), \\
\func{div}(u)=0, \\
u(0,\cdot )=u_{0},%
\end{cases}%
\end{equation*}%
where $\mathcal{L}_{\nu _{1},\nu _{2}}(u)=\nu _{1}\overrightarrow{\Delta }%
u+\nu _{2}r(u)$. If we have sufficient control over $\nu _{1}$ and $\nu _{2}$%
, then the well-posedness of the problem can be established. Moreover, we
can rewrite
\begin{equation*}
\mathcal{L}_{\nu _{1},\nu _{2}}(u)=\nu _{1}{\Delta }_{H}u+(\nu _{1}+\nu
_{2})r(u),
\end{equation*}%
so that, in the special case $\nu _{2}=-\nu _{1}$, we have $\mathcal{L}_{\nu
_{1},\nu _{2}}(u)=\nu _{1}{\Delta }_{H}u$, which corresponds to the modified
Navier--Stokes equations (see \cite{Czubak}).

Indeed, let us consider, for example, the case $\nu _{1}=\nu $ and $\nu
_{2}=1$. Then we obtain
\begin{equation}  \label{nsv}
\begin{cases}
\partial _{t}u-{\nu }\overrightarrow{\Delta }u-r(u)+B(u)=-\mathbb{P}(\nabla
_{u}u), \\
\func{div}(u)=0, \\
u(0,\cdot )=u_{0}.%
\end{cases}%
\end{equation}%
Next, consider the modified Stokes equations
\begin{equation*}
\begin{cases}
\partial _{t}u-\left( {\nu }\overrightarrow{\Delta }u+r(u)-B(u)\right) =0,
\\
\func{div}(u)=0, \\
u(0,\cdot )=u_{0}.%
\end{cases}%
\end{equation*}%
In order to obtain the well-posedness results, we need the dispersive
estimates for this particular case. This is addressed in the following lemma.

\begin{lem}[Dispersive: $\mathcal{M}_{p,\protect\lambda }^{g}\rightarrow
\mathcal{M}_{q,\protect\lambda }^{g}$]
\label{l-v1} Let $(M,g)$ be a complete, simply connected, non-compact, $m$%
-dimensional Riemannian manifold without boundary. Let $1\leq p\leq q<\infty
$ and $\lambda \in \lbrack 0,m)$ and denote by $u_{0}$ the initial data for
the heat equation.

\begin{enumerate}
\item[$(i)$] Assume that $-K(m-1)\leq \func{Ric}\leq -c_{0}<0$, $\kappa \leq
0,$ and let $u_{0}\in \mathcal{M}_{p,\lambda }^{g}(\Gamma (TM))$. Then, for
all $t>0$, the following estimate holds:%
\begin{equation*}
\Vert e^{t(\nu \overrightarrow{\Delta }+r-B)}u_{0}\Vert _{g,q,\lambda }\leq C%
\left[ d(t)^{-\frac{m}{2}}t^{\frac{\lambda }{2}}\right] ^{\left( \frac{1}{p}-%
\frac{1}{q}\right) }e^{-\sigma _{\nu }t+Ct}\Vert u_{0}\Vert _{g,p,\lambda },
\end{equation*}%
where $d(t)=\min \{1,t\}$, $\sigma _{\nu }=c_{0}+\nu \left( \alpha _{p,q}+k%
\frac{\lambda }{mq}\gamma _{g}\right) $, $\alpha _{p,q}=\min \left\{ \frac{%
4\delta _{m}(p-1)}{pq},\frac{\lambda _{1}}{q}\right\} ,$ $k=\sqrt{K}(m-1)$, $%
\gamma _{g}=k\left( 1+\frac{\lambda }{m}\right) \left( \frac{1}{4}-c\right)
^{-1}$ with $0<c<\frac{1}{4}$, and $C=C(m,c_{0},K)$.

\item[$(ii)$] Assume that $-K(m-1)\leq \func{Ric}\leq -c_{0}<0$, $\kappa
\leq 0,$ and let $u_{0}\in \mathcal{M}_{p,\lambda }(\Gamma (TM))$. Then, for
all $t>0$, the following estimate holds:
\begin{equation*}
\Vert e^{t(\nu \overrightarrow{\Delta }+r-B)}u_{0}\Vert _{g,q,\lambda }\leq C%
\left[ d(t)^{-\frac{m}{2}}t^{\frac{\lambda }{2}}\right] ^{\left( \frac{1}{p}-%
\frac{1}{q}\right) }e^{-\sigma _{\nu }t+Ct}\Vert u_{0}\Vert _{p,\lambda },
\end{equation*}%
where $\sigma _{\nu }=\widetilde{c}_{0}+\nu \alpha _{p,q}$ with $0<%
\widetilde{c}_{0}<c_{0}$, and $C=C(m,c_{0},K).$
\end{enumerate}
\end{lem}

\noindent {\textbf{Proof.}} Since $B:\mathcal{M}_{p,\lambda }^{g}\rightarrow
\mathcal{M}_{p,\lambda }^{g}$ is bounded, using Duhamel's formula with $%
w(t)=e^{t(\nu \overrightarrow{\Delta }+r-B)}u_{0},$ we obtain
\begin{equation*}
w(t)=e^{t({\nu }\overrightarrow{\Delta }+r)}u_{0}+\int_{0}^{t}e^{(t-\tau )({%
\nu }\overrightarrow{\Delta }+r)}Bw(\tau )\,d\tau .
\end{equation*}%
In Theorem\ \ref{th-cd}, we analyzed the semigroup $e^{t(\overrightarrow{%
\Delta }+r)}.$ By proceeding in an analogous manner, we can derive
corresponding estimates for $e^{t({\nu }\overrightarrow{\Delta }+r)}$, which
is essentially the same semigroup (up to a time rescaling), and consequently
obtain
\begin{align*}
\Vert w(t)\Vert _{g,q,\lambda }& \leq \Vert e^{t({\nu }\overrightarrow{%
\Delta }+r)}u_{0}\Vert _{g,q,\lambda }+\int_{0}^{t}\Vert e^{(t-\tau )({\nu }%
\overrightarrow{\Delta }+r)}Bw(\tau )\Vert _{g,q,\lambda }\,d\tau \\
& \leq Ce^{-\sigma _{\nu }t}\Vert u_{0}\Vert _{g,q,\lambda
}+C\int_{0}^{t}e^{-\sigma _{\nu }(t-\tau )}\Vert w(\tau )\Vert _{g,q,\lambda
}\,d\tau ,
\end{align*}%
where $\sigma _{\nu }=c_{0}+\nu \left( \alpha _{p,q}+k\frac{\lambda }{mq}%
\gamma _{g}\right) $ with $\alpha _{p,q}=\min \left\{ \frac{4\delta _{m}(p-1)%
}{pq},\frac{\lambda _{1}}{q}\right\} $, $\gamma _{g}=k\left( 1+\frac{\lambda
}{m}\right) \left( \frac{1}{4}-c\right) ^{-1}$, $k=\sqrt{K}(m-1)$, $0<c<%
\frac{1}{4}$, and $C=C(m,c_{0},K)$. This yields
\begin{equation*}
e^{\sigma _{\nu }t}\Vert w(t)\Vert _{g,q,\lambda }\leq C\Vert u_{0}\Vert
_{g,q,\lambda }+C\int_{0}^{t}e^{\sigma _{\nu }\tau }\Vert w(\tau )\Vert
_{g,q,\lambda }\,d\tau ,
\end{equation*}%
and, by Gr\"{o}nwall's inequality, it follows that
\begin{equation}
\Vert e^{t(\nu \overrightarrow{\Delta }+r-B)}u_{0}\Vert _{g,q,\lambda }\leq
Ce^{-\sigma _{\nu }t+Ct}\Vert u_{0}\Vert _{g,q,\lambda }.  \label{eq-v}
\end{equation}%
Next, for $p\leq q$ and $0<t\leq 2,$ we can estimate
\begin{align*}
\Vert w(t)\Vert _{g,q,\lambda }& \leq \Vert e^{t({\nu }\overrightarrow{%
\Delta }+r)}u_{0}\Vert _{g,p,\lambda }+\int_{0}^{t}\Vert e^{(t-\tau )({\nu }%
\overrightarrow{\Delta }+r)}Bw(\tau )\Vert _{g,q,\lambda }\,d\tau \\
& \leq C\left[ d(t)^{-\frac{m}{2}}t^{\frac{\lambda }{2}}\right] ^{\frac{1}{p}%
-\frac{1}{q}}e^{-\sigma _{\nu }t}\Vert u_{0}\Vert _{g,p,\lambda
}+C\int_{0}^{t/2}\left[ d(t-\tau )^{-\frac{m}{2}}(t-\tau )^{\frac{\lambda }{2%
}}\right] ^{\frac{1}{p}-\frac{1}{q}}e^{-\sigma _{\nu }(t-\tau )}\Vert w(\tau
)\Vert _{g,p,\lambda }\,d\tau \\
& \hspace{0.3cm}+C\int_{t/2}^{t}\Vert w(\tau )\Vert _{g,q,\lambda }\,d\tau \\
& \leq C\left[ d(t)^{-\frac{m}{2}}t^{\frac{\lambda }{2}}\right] ^{\frac{1}{p}%
-\frac{1}{q}}\Vert u_{0}\Vert _{g,p,\lambda }+C\left[ d(t)^{-\frac{m}{2}}t^{%
\frac{\lambda }{2}}\right] ^{\frac{1}{p}-\frac{1}{q}}\int_{0}^{t/2}\Vert
w(\tau )\Vert _{g,p,\lambda }\,d\tau +C\int_{t/2}^{t}\Vert w(\tau )\Vert
_{g,q,\lambda }\,d\tau .
\end{align*}%
By using \eqref{eq-v} in the second term of the last line above, we arrive
at
\begin{equation*}
\Vert w(t)\Vert _{g,q,\lambda }\leq C\left[ d(t)^{-\frac{m}{2}}t^{\frac{%
\lambda }{2}}\right] ^{\frac{1}{p}-\frac{1}{q}}\Vert u_{0}\Vert
_{g,p,\lambda }+C\int_{t/2}^{t}\Vert w(\tau )\Vert _{g,q,\lambda }\,d\tau .
\end{equation*}%
Setting $W(t)=\left[ d(t)^{-\frac{m}{2}}t^{\frac{\lambda }{2}}\right]
^{-\left( \frac{1}{p}-\frac{1}{q}\right) }\Vert w(t)\Vert _{g,q,\lambda }$ ,
it follows that
\begin{align*}
W(t)& \leq C\Vert u_{0}\Vert _{g,p,\lambda }+C\left[ d(t)^{-\frac{m}{2}}t^{%
\frac{\lambda }{2}}\right] ^{-\left( \frac{1}{p}-\frac{1}{q}\right)
}\int_{t/2}^{t}\left[ d(\tau )^{-\frac{m}{2}}\tau ^{\frac{\lambda }{2}}%
\right] ^{\left( \frac{1}{p}-\frac{1}{q}\right) }W(\tau )\,d\tau \\
& \leq C\Vert u_{0}\Vert _{g,p,\lambda }+C\int_{0}^{t}W(\tau )\,d\tau .
\end{align*}%
By Gr\"{o}nwall's inequality, we obtain for $0<t\leq 2$
\begin{equation}
\Vert e^{t(\nu \overrightarrow{\Delta }+r-B)}u_{0}\Vert _{g,q,\lambda }\leq C%
\left[ d(t)^{-\frac{m}{2}}t^{\frac{\lambda }{2}}\right] ^{\frac{1}{p}-\frac{1%
}{q}}\Vert u_{0}\Vert _{g,p,\lambda }.  \label{eq-v1}
\end{equation}%
Next, for $t>1,$ we can use the semigroup property in order to estimate
\begin{align}
\Vert e^{t({\nu }\overrightarrow{\Delta }+r)}u_{0}\Vert _{g,q,\lambda }&
=\Vert e^{({\nu }\overrightarrow{\Delta }+r)}\left( e^{(t-1)({\nu }%
\overrightarrow{\Delta }+r)}u_{0}\right) \Vert _{g,q,\lambda }  \notag \\
& \leq C\Vert e^{(t-1)({\nu }\overrightarrow{\Delta }+r)}u_{0}\Vert
_{g,p,\lambda }  \notag \\
& \leq C\left[ d(t)^{-\frac{m}{2}}t^{\frac{\lambda }{2}}\right] ^{\frac{1}{p}%
-\frac{1}{q}}e^{-\sigma _{\nu }t+Ct}\Vert u_{0}\Vert _{g,p,\lambda }.
\label{eq-v2}
\end{align}%
Hence, by \eqref{eq-v1} and \eqref{eq-v2}, we get
\begin{equation*}
\Vert e^{t(\nu \overrightarrow{\Delta }+r-B)}u_{0}\Vert _{g,q,\lambda }\leq C%
\left[ d(t)^{-\frac{m}{2}}t^{\frac{\lambda }{2}}\right] ^{\frac{1}{p}-\frac{1%
}{q}}e^{-\sigma _{\nu }t+Ct}\Vert u_{0}\Vert _{g,p,\lambda }.
\end{equation*}%
Finally, we point out that the proof for the case $u_{0}\in \mathcal{M}%
_{p,\lambda }$ (item $(ii)$) follows similarly.\fin

Next, we derive the smoothing (gradient) estimates associated with the
semigroup $e^{t(\nu \overrightarrow{\Delta }+r-B)}$.

\begin{lem}[Smoothing: $\mathcal{M}_{p,\protect\lambda }^{g}\rightarrow
\mathcal{M}_{q,\protect\lambda }^{g}$]
\label{l-v2}Let $(M,g)$ be a complete, simply connected, non-compact, $m$%
-dimensional Riemannian manifold without boundary. Let $1\leq p\leq q<\infty
$ and $\lambda \in \lbrack 0,m)$ and denote by $u_{0}$ the initial data for
the heat equation.

\begin{enumerate}
\item[$(i)$] Assume that $M$ has bounded geometry, $-K(m-1)\leq \func{Ric}%
\leq -c_{0}<0$, $\kappa \leq 0,$ and let $u_{0}\in \mathcal{M}_{p,\lambda
}^{g}(\Gamma (TM))$. Then, for all $t>0$, the following estimate holds:
\begin{equation*}
\Vert \nabla e^{t(\nu \overrightarrow{\Delta }+r-B)}u_{0}\Vert _{g,q,\lambda
}\leq C\left[ d(t)\right] ^{-\frac{1}{2}-\frac{m}{2}\left( \frac{1}{p}-\frac{%
1}{q}\right) }t^{\frac{\lambda }{2}\left( \frac{1}{p}-\frac{1}{q}\right)
}e^{-\sigma _{\nu }t+Ct}\Vert u_{0}\Vert _{g,p,\lambda },
\end{equation*}%
where $d(t)=\min \{1,t\}$, $\sigma _{\nu }=c_{0}+\nu \left( \alpha _{p,q}+k%
\frac{\lambda }{mq}\gamma _{g}\right) $, $\alpha _{p,q}=\min \left\{ \frac{%
4\delta _{m}(p-1)}{pq},\frac{\lambda _{1}}{q}\right\} ,$ $k=\sqrt{K}(m-1)$,
and $\gamma _{g}=k\left( 1+\frac{\lambda }{m}\right) \left( \frac{1}{4}%
-c\right) ^{-1}$ with $0<c<\frac{1}{4}$, and $C=C(m,c_{0},K)$.

\item[$(ii)$] Assume that $M$ has bounded geometry, $-K(m-1)\leq \func{Ric}%
\leq -c_{0}<0$, $\kappa \leq 0,$ and let $u_{0}\in \mathcal{M}_{p,\lambda
}(\Gamma (TM))$. Then, for all $t>0$, the following estimate holds:
\begin{equation*}
\Vert \nabla e^{t(\nu \overrightarrow{\Delta }+r-B)}u_{0}\Vert _{g,q,\lambda
}\leq C\left[ d(t)\right] ^{-\frac{1}{2}-\frac{m}{2}\left( \frac{1}{p}-\frac{%
1}{q}\right) }t^{\frac{\lambda }{2}\left( \frac{1}{p}-\frac{1}{q}\right)
}e^{-\sigma _{\nu }t+Ct}\Vert u_{0}\Vert _{p,\lambda },
\end{equation*}%
where $\sigma _{\nu }=\widetilde{c}_{0}+\nu \alpha _{p,q}$ with $0<%
\widetilde{c}_{0}<c_{0}$, and $C=C(m,c_{0},K)$.
\end{enumerate}
\end{lem}

\noindent {\textbf{Proof.}} Invoking Duhamel's formula once more, we get
\begin{equation*}
\nabla w(t)=\nabla e^{t({\nu }\overrightarrow{\Delta }+r)}u_{0}+\int_{0}^{t}%
\nabla e^{(t-\tau )({\nu }\overrightarrow{\Delta }+r)}Bw(\tau )\,d\tau .
\end{equation*}%
Then, for $0<t<2,$ we can estimate
\begin{align*}
\Vert \nabla w(t)\Vert _{g,q,\lambda }& \leq \Vert \nabla e^{t({\nu }%
\overrightarrow{\Delta }+r)}u_{0}\Vert _{g,q,\lambda }+\int_{0}^{t}\Vert
\nabla e^{(t-\tau )({\nu }\overrightarrow{\Delta }+r)}Bw(\tau )\Vert
_{g,q,\lambda }\,d\tau \\
& \leq C\frac{1}{\sqrt{t}}\Vert u_{0}\Vert _{g,q,\lambda }+\int_{0}^{t}\frac{%
1}{\sqrt{t-\tau }}\Vert w(\tau )\Vert _{g,q,\lambda }\,d\tau .
\end{align*}%
Combining this with (\ref{eq-v}), we deduce that
\begin{equation}
\Vert \nabla e^{t({\nu }\overrightarrow{\Delta }+r)}u_{0}\Vert _{g,q,\lambda
}\leq C\frac{1}{\sqrt{t}}\Vert u_{0}\Vert _{g,q,\lambda }.  \label{eq-sv1}
\end{equation}%
For $t>1,$ we use the semigroup property to obtain%
\begin{align}
\Vert \nabla e^{t({\nu }\overrightarrow{\Delta }+r)}u_{0}\Vert _{g,q,\lambda
}& =\Vert \nabla e^{({\nu }\overrightarrow{\Delta }+r)}\left( e^{(t-1)({\nu }%
\overrightarrow{\Delta }+r)}u_{0}\right) \Vert _{g,q,\lambda }  \notag \\
& \leq C\Vert e^{(t-1)({\nu }\overrightarrow{\Delta }+r)}u_{0}\Vert
_{g,q,\lambda }  \notag \\
& \leq Ce^{-\sigma _{\nu }t+Ct}\Vert u_{0}\Vert _{g,q,\lambda }.
\label{eq-sv2}
\end{align}%
Therefore, it follows from \eqref{eq-sv1} and \eqref{eq-sv2} that
\begin{equation}
\Vert \nabla e^{t({\nu }\overrightarrow{\Delta }+r)}u_{0}\Vert _{g,q,\lambda
}\leq C\left[ d(t)\right] ^{-\frac{1}{2}}e^{-\sigma _{\nu }t+Ct}\Vert
u_{0}\Vert _{g,q,\lambda }\text{, for all }t>0\text{.}  \label{eq-sv3}
\end{equation}%
Next, using the semigroup property together with \eqref{eq-sv3} and Lemma %
\ref{l-v1}, we are led to conclude
\begin{align*}
\Vert \nabla e^{t({\nu }\overrightarrow{\Delta }+r)}u_{0}\Vert _{g,q,\lambda
}& =\Vert \nabla e^{\frac{t}{2}({\nu }\overrightarrow{\Delta }+r)}\left( e^{%
\frac{t}{2}({\nu }\overrightarrow{\Delta }+r)}u_{0}\right) \Vert
_{g,q,\lambda } \\
& \leq C\left[ d(t)\right] ^{-\frac{1}{2}}e^{-\sigma _{\nu }\frac{t}{2}+C%
\frac{t}{2}}\Vert e^{\frac{t}{2}({\nu }\overrightarrow{\Delta }%
+r)}u_{0}\Vert _{g,q,\lambda } \\
& \leq C\left[ d(t)\right] ^{-\frac{1}{2}-\frac{m}{2}\left( \frac{1}{p}-%
\frac{1}{q}\right) }t^{\frac{\lambda }{2}\left( \frac{1}{p}-\frac{1}{q}%
\right) }e^{-\sigma _{\nu }t+Ct}\Vert u_{0}\Vert _{g,p,\lambda },
\end{align*}%
as desired.\fin

With the previous semigroups estimates in hands, we are in position to
establish a well-posedness result for system (\ref{nsv}). First, in light of
Duhamel's principle, system (\ref{nsv}) can be formally rewritten in its
mild formulation%
\begin{equation}
u(t)=e^{t(\nu \overrightarrow{\Delta }+r-B)}u_{0}-\int_{0}^{t}e^{(t-\tau
)(\nu \overrightarrow{\Delta }+r-B)}\mathbb{P}(\func{div}(u\otimes u))(\tau
)\,d\tau .  \label{mildv}
\end{equation}

Next let $p,q,s\in \lbrack 1,\infty )$ and $\lambda \in \lbrack 0,m)$
satisfy $m-\lambda \leq p<q,s$, and recall that $\mathcal{C}_{b}(0,\infty
),Z)$ stands for the space of bounded continuous functions from $(0,\infty )$
to $Z$. \ We investigate solutions of (\ref{mildv}) (mild solutions) in the
functional space
\begin{equation}
CX_{\infty }^{\rho }=CM_{\infty }\cap X_{\infty }^{\rho },
\label{space-th-visc-1}
\end{equation}%
where
\begin{equation*}
CM_{\infty }=\mathcal{C}_{b}((0,{\infty }),\mathcal{M}_{p,\lambda
}^{g}(\Gamma (TM)))
\end{equation*}%
and
\begin{align*}
X_{\infty }^{\rho }=\bigg\{& u\in L_{loc}^{\infty }((0,{\infty }),\mathcal{M}%
_{q,\lambda }^{g}(\Gamma (TM))), \\
& \nabla u\in L_{loc}^{\infty }((0,{\infty }),\mathcal{M}_{\widetilde{q}%
,\lambda }^{g}(\Gamma (TM)))\cap L_{loc}^{\infty }((0,{\infty }),\mathcal{M}%
_{s,\lambda }^{g}(\Gamma (TM))) \\
& \,\bigg|\,[d(t)^{\frac{m}{2}}t^{-\frac{\lambda }{2}}]^{\left( \frac{1}{p}-%
\frac{1}{q}\right) }e^{\rho t}\Vert u(t)\Vert _{g,q,\lambda }+\left[ d(t)%
\right] ^{\frac{1}{2}+\frac{m}{2}\left( \frac{1}{p}-\frac{1}{s}\right) }t^{-%
\frac{\lambda }{2}\left( \frac{1}{p}-\frac{1}{s}\right) }e^{\rho t}\Vert
\nabla u(t)\Vert _{g,s,\lambda }\in L^{\infty }(0,{\infty })\bigg\},
\end{align*}%
where the parameter $\rho >0$ is chosen according to the large-time decay
rate in the smoothing estimate of Lemma \ref{l-v2}. The space $CX_{\infty
}^{\rho }$ is Banach when equipped with the norm $\Vert \cdot \Vert
_{CX_{\infty }^{\rho }}=\Vert \cdot \Vert _{CM_{\infty }}+\Vert \cdot \Vert
_{X_{\infty }^{\rho }}$ where $\Vert u\Vert _{CM_{\infty }}=\sup_{0<t<\infty
}\Vert u(t)\Vert _{g,p,\lambda }$ and
\begin{equation}
\Vert u\Vert _{X_{\infty }^{\rho }}=\sup_{0<t<\infty }\bigg\{[d(t)^{\frac{m}{%
2}}t^{-\frac{\lambda }{2}}]^{\left( \frac{1}{p}-\frac{1}{q}\right) }e^{\rho
t}\Vert u(t)\Vert _{g,q,\lambda }+\left[ d(t)\right] ^{\frac{1}{2}+\frac{m}{2%
}\left( \frac{1}{p}-\frac{1}{s}\right) }t^{-\frac{\lambda }{2}\left( \frac{1%
}{p}-\frac{1}{s}\right) }e^{\rho t}\Vert \nabla u(t)\Vert _{g,s,\lambda }%
\bigg\}.  \label{norm2}
\end{equation}

Our global well-posedness result for (\ref{nsv}) reads as follows.

\begin{thr}[Global well-posedness]
\label{th-visc} Let $(M,g)$ be a manifold as in Definition \ref{df-m}, and
let $p,q,s\in \lbrack 1,\infty )$ and $\lambda \in \lbrack 0,m)$ satisfy $%
m-\lambda \leq p<q,s.$ Assume that $u_{0}\in \mathcal{M}_{p,\lambda
}^{g}(\Gamma (TM))$ with $\func{div}(u_{0})=0$. Then, there exists $%
\varepsilon >0$ such that the modified Navier-Stokes system (\ref{nsv})
admit a unique mild solution $u\in CX_{\infty }^{\rho }$ satisfying $\Vert
u\Vert _{CX_{\infty }^{\rho }}\leq 2\widetilde{C}\varepsilon $, for some
constant $\widetilde{C}>0$, provided that $\Vert u_{0}\Vert _{g,p,\lambda
}\leq \varepsilon $. Moreover, the solution $u\rightharpoonup u_{0}$ in the
sense of distributions on $M$ as $t\rightarrow 0^{+}$.
\end{thr}

\noindent {\textbf{Proof.}} The integral equation (\ref{mildv}) can be
written as
\begin{equation}
u(t)=u_{1}+\mathcal{N}_{\nu }(u,u)(t),  \label{aux-proof-mild-2000}
\end{equation}%
where $u_{1}=e^{t(\nu \overrightarrow{\Delta }+r-B)}u_{0}$ and
\begin{equation*}
\mathcal{N}_{\nu }(u,u)(t)=-\int_{0}^{t}e^{(t-\tau )(\nu \overrightarrow{%
\Delta }+r-B)}\mathbb{P}(\func{div}(u\otimes u))(\tau )\,d\tau .
\end{equation*}

By combining estimate \eqref{eq-v} with the boundedness of $\mathbb{P}$ (see
Theorem \ref{th-ri}), and using the assumptions $m-\lambda \leq p<q,s<\infty
$ together with $\frac{1}{r}=\frac{1}{q}+\frac{1}{s}$, we obtain
\begin{align*}
\Vert \mathcal{N}_{\nu }(u,u)& (t)\Vert _{g,q,\lambda }\leq
C\int_{0}^{t}[d(t-\tau )]^{-\frac{m}{2}\left( \frac{1}{r}-\frac{1}{q}\right)
}(t-\tau )^{\frac{\lambda }{2}\left( \frac{1}{r}-\frac{1}{q}\right)
}e^{-\rho (t-\tau )}\Vert \func{div}(u\otimes u)(\tau )\Vert _{g,r,\lambda
}\,d\tau \\
& \leq C\int_{0}^{t}[d(t-\tau )]^{-\frac{m}{2}\left( \frac{1}{r}-\frac{1}{q}%
\right) }(t-\tau )^{\frac{\lambda }{2}\left( \frac{1}{r}-\frac{1}{q}\right)
}e^{-\rho (t-\tau )}\Vert u(\tau )\Vert _{g,q,\lambda }\Vert \nabla u(\tau
)\Vert _{g,s,\lambda }\,d\tau \\
& \leq C\Vert u\Vert _{X_{\infty }^{\rho }}^{2}\int_{0}^{t}[d(t-\tau )]^{-%
\frac{m}{2}\left( \frac{1}{r}-\frac{1}{q}\right) }(t-\tau )^{\frac{\lambda }{%
2}\left( \frac{1}{r}-\frac{1}{q}\right) }e^{-\rho (t-\tau )}e^{-2\rho \tau
}[d(\tau )]^{-\frac{1}{2}-\frac{m}{2}\left( \frac{2}{p}-\frac{1}{q}-\frac{1}{%
s}\right) }\tau ^{\frac{\lambda }{2}\left( \frac{2}{p}-\frac{1}{q}-\frac{1}{s%
}\right) }\,d\tau ,
\end{align*}%
which yields
\begin{equation}
\lbrack d(t)^{\frac{m}{2}}t^{-\frac{\lambda }{2}}]^{\left( \frac{1}{p}-\frac{%
1}{q}\right) }e^{\rho t}\Vert \mathcal{N}(u,u)(t)\Vert _{g,q,\lambda }\leq
C\Vert u\Vert _{X_{\infty }^{\rho }}^{2}.  \label{aux-proof-2001}
\end{equation}%
For second term of the norm $\Vert \cdot \Vert _{X_{\infty }^{\rho }}$, we
have
\begin{align*}
\Vert \nabla \mathcal{N}_{\nu }& (u,u)(t)\Vert _{g,s,\lambda }\leq
C\int_{0}^{t}[d(t-\tau )]^{-\frac{1}{2}-\frac{m}{2}\left( \frac{1}{r}-\frac{1%
}{s}\right) }(t-\tau )^{\frac{\lambda }{2}\left( \frac{1}{r}-\frac{1}{s}%
\right) }e^{-\rho (t-\tau )}\Vert \func{div}(u\otimes u)(\tau )\Vert
_{g,r,\lambda }\,d\tau \\
& \leq C\int_{0}^{t}[d(t-\tau )]^{-\frac{1}{2}-\frac{m}{2}\left( \frac{1}{r}-%
\frac{1}{s}\right) }(t-\tau )^{\frac{\lambda }{2}\left( \frac{1}{r}-\frac{1}{%
q}\right) }e^{-\rho (t-\tau )}\Vert u(\tau )\Vert _{g,q,\lambda }\Vert
\nabla u(\tau )\Vert _{g,s,\lambda }\,d\tau \\
& \leq C\Vert u\Vert _{X_{\infty }^{\rho }}^{2}\int_{0}^{t}[d(t-\tau )]^{{-%
\frac{1}{2}-\frac{m}{2}\left( \frac{1}{r}-\frac{1}{s}\right) }}(t-\tau )^{%
\frac{\lambda }{2}\left( \frac{1}{r}-\frac{1}{s}\right) }e^{-\rho (t-\tau
)}e^{-2\rho \tau }[d(\tau )]^{\left( -\frac{1}{2}+\frac{2}{p}-\frac{1}{q}-%
\frac{1}{s}\right) }\tau ^{\frac{\lambda }{2}\left( \frac{2}{p}-\frac{1}{q}-%
\frac{1}{s}\right) }\,d\tau ,
\end{align*}%
and then
\begin{equation}
\left[ d(t)\right] ^{{\frac{1}{2}+\frac{m}{2}\left( \frac{1}{r}-\frac{1}{s}%
\right) }}t^{-\frac{\lambda }{2}\left( \frac{1}{r}-\frac{1}{s}\right)
}e^{\rho t}\Vert \nabla \mathcal{N}_{\nu }(u,u)(t)\Vert _{g,s,\lambda }\leq
C\Vert u\Vert _{X_{\infty }^{\rho }}^{2}.  \label{aux-proof-2002}
\end{equation}

Moreover, taking $\frac{1}{\widetilde{r}}=\frac{1}{p}+\frac{1}{{q}}$ and
using H\"{o}lder's inequality (\ref{holder}), we can estimate the norm of $%
CM_{\infty }$ as
\begin{align*}
\Vert \mathcal{N}_{\nu }& (u,u)(t)\Vert _{g,p,\lambda }\leq
C\int_{0}^{t}[d(t-\tau )]^{-\frac{1}{2}-\frac{m}{2}\left( \frac{1}{%
\widetilde{r}}-\frac{1}{p}\right) }(t-\tau )^{\frac{\lambda }{2}\left( \frac{%
1}{\widetilde{r}}-\frac{1}{p}\right) }e^{-\rho (t-\tau )}\Vert u\otimes
u(\tau )\Vert _{g,\widetilde{r},\lambda }\,d\tau \\
& \leq C\Vert u\Vert _{X_{\infty }^{\rho }}\Vert u\Vert _{CM_{\infty
}}e^{-\rho t}\,\int_{0}^{t}[d(t-\tau )]^{\left( \frac{1}{q}+\frac{1}{m}%
\right) }(t-\tau )^{\frac{\lambda }{2q}}[d(\tau )^{-\frac{m}{2}}\tau ^{\frac{%
\lambda }{2}}]^{\left( \frac{1}{p}-\frac{1}{q}\right) }d\tau .
\end{align*}%
This last estimate leads us to%
\begin{equation}
\Vert \mathcal{N}_{\nu }(u,u)\Vert _{CM_{\infty }^{\rho }}\leq C\Vert u\Vert
_{X_{\infty }^{\rho }}\Vert u\Vert _{CM_{\infty }}.  \label{aux-proof-2003}
\end{equation}%
Bringing together the estimates (\ref{aux-proof-2001}), (\ref{aux-proof-2002}%
) and (\ref{aux-proof-2003}), we arrive at
\begin{align*}
\Vert \mathcal{N}_{\nu }(u,u)(t)\Vert _{CX_{\infty }^{\rho }}& =\Vert
\mathcal{N}_{\nu }(u,u)(t)\Vert _{CM_{\infty }}+\Vert \mathcal{N}_{\nu
}(u,u)(t)\Vert _{X_{\infty }^{\rho }} \\
& \leq C\left( \Vert u\Vert _{X_{\infty }^{\rho }}\Vert u\Vert _{CM_{\infty
}}+\Vert u\Vert _{X_{\infty }^{\rho }}^{2}\right) \\
& \leq C\Vert u\Vert _{CM_{\infty }^{\rho }}^{2}.
\end{align*}%
Also, using Lemma \ref{l-v1}, we can estimate the linear part in (\ref%
{aux-proof-mild-2000}) as%
\begin{eqnarray*}
\Vert u_{1}\Vert _{CX_{\infty }^{\rho }} &=&\Vert e^{t(\nu \overrightarrow{%
\Delta }+r-B)}u_{0}\Vert _{CM_{\infty }}+\Vert e^{t(\nu \overrightarrow{%
\Delta }+r-B)}u_{0}\Vert _{X_{\infty }^{\rho }} \\
&\leq &\widetilde{C}\left\Vert u_{0}\right\Vert _{g,p,\lambda }\leq
\widetilde{C}\varepsilon ,
\end{eqnarray*}%
provided that $\left\Vert u_{0}\right\Vert _{g,p,\lambda }\leq \varepsilon .$
The result then follows by choosing $\varepsilon >0$ sufficiently small and
applying Lemma \ref{fix}. \fin

\subsection{Einstein manifolds with negative curvature}

\label{s6.1} If we add the additional hypothesis that $M$ is an Einstein
manifold with $\func{Ric}=\lambda <0$, then
\begin{equation*}
\func{div}(u)=0\Rightarrow \func{div}(r(u))=0,
\end{equation*}
which in turn implies that $Bu=0$ in the Navier-Stokes system. Thus, this
system on $M$ can be written as
\begin{equation}
\begin{cases}
\partial _{t}u-\left( \overrightarrow{\Delta }u+r(u)\right) =-\mathbb{P}(%
\func{div}(u\otimes u)), \\
\func{div}(u)=0, \\
u(0,\cdot )=u_{0}.%
\end{cases}
\label{ns-e}
\end{equation}%
In its mild formulation, this system takes the following form
\begin{equation}
u(t)=e^{t(\overrightarrow{\Delta }+r)}u_{0}-\int_{0}^{t}e^{(t-\tau )(%
\overrightarrow{\Delta }+r)}\mathbb{P}(\func{div}(u\otimes u))(\tau )\,d\tau
\label{mild-e}
\end{equation}

An important example of an Einstein manifold is the hyperbolic space $%
\mathbb{H}^{m}$, and more generally, any space of constant sectional
curvature. Note also that the hyperbolic space satisfies Definition \ref%
{df-m}.

In order to analyze the global well-posedness of (\ref{ns-e}), let $p,q\in
\lbrack 1,\infty )$ and $\lambda \in \lbrack 0,m)$ satisfying $m-\lambda
\leq p<q$ and consider the functional setting $CX_{\infty }=CM_{\infty }\cap
X_{\infty }$, where $CM_{\infty }=\mathcal{C}_{b}((0,{\infty }),\mathcal{M}%
_{p,\lambda }^{g}(\Gamma (TM)))$ and
\begin{equation}
X_{\infty }=\bigg\{u\in L_{loc}^{\infty }((0,{\infty }),\mathcal{M}%
_{q,\lambda }^{g}(\Gamma (TM)))\,\bigg|\,[d(t)^{\frac{m}{2}}t^{-\frac{%
\lambda }{2}}]^{\left( \frac{1}{p}-\frac{1}{q}\right) }e^{\beta t}\Vert
u(t)\Vert _{g,q,\lambda }\in L^{\infty }(0,{\infty })\bigg\}  \label{space3}
\end{equation}%
where again $\beta $ is chosen according to the decay rate in the smoothing
estimate of Remark \ref{r-sm}.

The space $CX_{\infty }$ is endowed with the norm $\Vert \cdot \Vert
_{CX_{\infty }}=\Vert \cdot \Vert _{CM_{\infty }}+\Vert \cdot \Vert
_{X_{\infty }}$ where and
\begin{equation}
\Vert u\Vert _{CM_{\infty }}=\sup_{0<t<\infty }\Vert u(t)\Vert _{g,p,\lambda
}\text{ and }\Vert u\Vert _{X_{\infty }}=\sup_{0<t<{\infty }}\left\{ [d(t)^{%
\frac{m}{2}}t^{-\frac{\lambda }{2}}]^{\left( \frac{1}{p}-\frac{1}{q}\right)
}e^{\beta t}\Vert u(t)\Vert _{g,q,\lambda }\right\}  \label{space-norm3}
\end{equation}

We obtain the following global well-posedness result for system (\ref{ns-e}).

\begin{thr}[Einstein Manifolds-Global Well-posedness]
\label{th-ein} Let $(M,g)$ be a manifold as in Definition \ref{df-m}, and
assume further that $M$ is an Einstein manifold. Let $p,q\in \lbrack
1,\infty )$ and $\lambda \in \lbrack 0,m)$ satisfy $m-\lambda \leq p<q$, and
suppose that $u_{0}\in \mathcal{M}_{p,\lambda }^{g}(\Gamma (TM))$ with $%
\func{div}(u_{0})=0$. Then, there exists $\varepsilon >0$ such that, if $%
\Vert u_{0}\Vert _{g,p,\lambda }\leq \varepsilon ,$ the Navier-Stokes system
(\ref{ns-e}) admits a unique mild solution $u\in CX_{\infty }$ satisfying $%
\Vert u\Vert _{CX_{\infty }}\leq 2\widetilde{C}\varepsilon $, for some
constant $\widetilde{C}>0$. Moreover, $u\rightharpoonup u_{0}$ in the sense
of distributions on $M$ as $t\rightarrow 0^{+}$.
\end{thr}

\noindent {\textbf{Proof. }}We begin by writing the mild formulation (\ref%
{mild-e}) in the form{\textbf{\ }}
\begin{equation}
u(t)=u_{1}+\mathcal{N}(u,u)(t),  \label{aux-proof-mild-3000}
\end{equation}%
where $u_{1}=e^{t(\overrightarrow{\Delta }+r)}u_{0}$ and
\begin{equation*}
\mathcal{N}(u,u)(t)=-\int_{0}^{t}e^{(t-\tau )(\overrightarrow{\Delta }+r)}%
\mathbb{P}(\func{div}(u\otimes u))(\tau )\,d\tau .
\end{equation*}

The operator $\overrightarrow{\Delta }+r$ commutes with the projection $%
\mathbb{P}$ provided that $\partial M=\varnothing .$ Therefore, after
applying the boundedness of $\mathbb{P}$ (see Theorem \ref{th-ri}), we can
use the corresponding smoothing estimates (see Theorem \ref{th-s-ric}) to
obtain
\begin{align}
\Vert \mathcal{N}(u,u)& (t)\Vert _{g,q,\lambda }\leq C\int_{0}^{t}[d(t-\tau
)]^{-\left( \frac{1}{2}+\frac{m}{2q}\right) }(t-\tau )^{\frac{\lambda }{2q}%
}e^{-\beta (t-\tau )}\Vert u\otimes u(\tau )\Vert _{g,q/2,\lambda }\,d\tau
\notag \\
& \leq C\int_{0}^{t}[d(t-\tau )]^{-\left( \frac{1}{2}+\frac{m}{2q}\right)
}(t-\tau )^{\frac{\lambda }{2q}}e^{-\beta (t-\tau )}\left( [d(\tau )^{\frac{m%
}{2}}\tau ^{\frac{\lambda }{2}}]^{\left( \frac{1}{p}-\frac{1}{q}\right)
}e^{-\beta \tau }\right) ^{2}\,d\tau \Vert u\Vert _{X_{\infty }}^{2}  \notag
\\
& \leq Ce^{-\beta t}\int_{0}^{t}[d(t-\tau )]^{-\left( \frac{1}{2}+\frac{m}{2q%
}\right) }(t-\tau )^{\frac{\lambda }{2q}}[d(\tau )^{\frac{m}{2}}\tau ^{\frac{%
\lambda }{2}}]^{2\left( \frac{1}{p}-\frac{1}{q}\right) }e^{-\beta \tau
}\,d\tau \Vert u\Vert _{X_{\infty }}^{2}.  \label{aux-proof-3000}
\end{align}%
For $0<t<1,$ using that $p\geq m-\lambda $, we can estimate the R.H.S. of (%
\ref{aux-proof-3000}) by
\begin{align*}
& Ce^{-\beta t}\left[ \int_{0}^{t}(t-\tau )^{-\left( \frac{1}{2}+\frac{m}{2q}%
\right) +\frac{\lambda }{2q}}\tau ^{-\left( \frac{m-\lambda }{p}-\frac{%
m-\lambda }{q}\right) }e^{-\beta \tau }\,d\tau \right] \Vert u\Vert
_{X_{\infty }}^{2} \\
\leq & C[d(t)^{-\frac{m}{2}}t^{\frac{\lambda }{2}}]^{\left( \frac{1}{p}-%
\frac{1}{q}\right) }e^{-\beta t}\Vert u\Vert _{X_{\infty }}^{2}.
\end{align*}%
For $t\geq 1$, since $q>p\geq m-\lambda $, we obtain a similar estimate,
leading us to
\begin{equation}
\Vert \mathcal{N}(u,u)\Vert _{X_{\infty }}\leq C\Vert u\Vert _{X_{\infty
}}^{2}.  \label{aux-proof-N-1}
\end{equation}%
Moreover, taking $\frac{1}{r}=\frac{1}{p}+\frac{1}{q}$ and using the H\"{o}%
lder's inequality (\ref{holder}), we handle the $CM_{\infty }$-component of
the norm of $CX_{\infty }$ as follows:
\begin{align}
\Vert \mathcal{N}& (u,u)(t)\Vert _{g,p,\lambda }\leq C\int_{0}^{t}[d(t-\tau
)]^{-\frac{1}{2}-\frac{m}{2}\left( \frac{1}{r}-\frac{1}{p}\right) }(t-\tau
)^{\frac{\lambda }{2}\left( \frac{1}{r}-\frac{1}{p}\right) }e^{-\beta
(t-\tau )}\Vert u\otimes u(\tau )\Vert _{g,r,\lambda }\,d\tau  \notag \\
& \leq C\int_{0}^{t}[d(t-\tau )]^{-\frac{1}{2}-\frac{m}{2}\left( \frac{1}{r}-%
\frac{1}{p}\right) }(t-\tau )^{\frac{\lambda }{2}\left( \frac{1}{r}-\frac{1}{%
p}\right) }e^{-\beta (t-\tau )}\Vert u(\tau )\Vert _{g,q,\lambda }\Vert
u(\tau )\Vert _{g,p,\lambda }\,d\tau  \notag \\
& \leq C\Vert u\Vert _{X_{\infty }}\Vert u\Vert _{CM_{\infty
}}\int_{0}^{t}[d(t-\tau )]^{-\frac{1}{2}-\frac{m}{2}\left( \frac{1}{r}-\frac{%
1}{p}\right) }(t-\tau )^{\frac{\lambda }{2}\left( \frac{1}{r}-\frac{1}{p}%
\right) }e^{-\beta (t-\tau )}[d(\tau )^{-\frac{m}{2}}\tau ^{\frac{\lambda }{2%
}}]^{\left( \frac{1}{p}-\frac{1}{q}\right) }e^{-\beta \tau }\,d\tau  \notag
\\
& \leq C\Vert u\Vert _{X_{\infty }}\Vert u\Vert _{CM_{\infty
}}\int_{0}^{t}[c_{(}t-\tau )]^{-\frac{1}{2}-\frac{m}{2}\left( \frac{1}{r}-%
\frac{1}{p}\right) }(t-\tau )^{\frac{\lambda }{2}\left( \frac{1}{r}-\frac{1}{%
p}\right) }[d(\tau )^{-\frac{m}{2}}\tau ^{\frac{\lambda }{2}}]^{\left( \frac{%
1}{p}-\frac{1}{q}\right) }e^{-\beta t}\,d\tau  \notag \\
& \leq C\Vert u\Vert _{X_{\infty }}\Vert u\Vert _{CM_{\infty
}}\int_{0}^{t}[d(t-\tau )]^{-\frac{1}{2}-\frac{m}{2}\left( \frac{1}{r}-\frac{%
1}{p}\right) }(t-\tau )^{\frac{\lambda }{2q}}[d(\tau )^{-\frac{m}{2}}\tau ^{%
\frac{\lambda }{2}}]^{\left( \frac{1}{p}-\frac{1}{q}\right) }e^{-\beta
t}\,d\tau .  \label{aux-proof-3002}
\end{align}%
For $0<t<1$, recalling that $q>p\geq m-\lambda $, the R.H.S. of (\ref%
{aux-proof-3002}) can be estimated by
\begin{align*}
& C\Vert u\Vert _{X_{\infty }}\Vert u\Vert _{CM_{\infty
}}\int_{0}^{t}(t-\tau )^{-\frac{1}{2}-\frac{m-\lambda }{2q}}\tau ^{-\frac{%
m-\lambda }{2}\left( \frac{1}{p}-\frac{1}{q}\right) }e^{-\beta t}\,d\tau \\
\leq & C\Vert u\Vert _{X_{\infty }}\Vert u\Vert _{CM_{\infty
}}\int_{0}^{t}(t-\tau )^{-\frac{1}{2}-\frac{m-\lambda }{2q}}\tau ^{-\frac{1}{%
2}+\frac{m-\lambda }{2q}}e^{-\beta t}\,d\tau \\
\leq & C\Vert u\Vert _{X_{\infty }}\Vert u\Vert _{CM_{\infty }}.
\end{align*}%
Similarly for $t\geq 1$. Thus, we arrive at%
\begin{equation}
\Vert \mathcal{N}(u,u)\Vert _{CM_{\infty }}\leq C\Vert u\Vert _{X_{\infty
}}\Vert u\Vert _{CM_{\infty }}.  \label{aux-proof-N-2}
\end{equation}%
Combining estimates (\ref{aux-proof-N-1}) and (\ref{aux-proof-N-2}) yields%
\begin{align*}
\Vert \mathcal{N}(u,u)\Vert _{CX_{\infty }}& =\Vert \mathcal{N}(u,u)\Vert
_{CM_{\infty }}+\Vert \mathcal{N}(u,u)\Vert _{X_{\infty }} \\
& \leq C\left( \Vert u\Vert _{X_{\infty }}\Vert u\Vert _{CM_{\infty }}+\Vert
u\Vert _{X_{\infty }}^{2}\right) \\
& \leq C\Vert u\Vert _{CX_{\infty }}^{2}.
\end{align*}%
Also, by Theorem \ref{th-cd}, the linear part in (\ref{aux-proof-mild-2000})
satisfies%
\begin{equation*}
\Vert u_{1}\Vert _{CX_{\infty }}=\Vert e^{t(\overrightarrow{\Delta }%
+r)}u_{0}\Vert _{CM_{\infty }}+\Vert e^{t(\nu \overrightarrow{\Delta }%
+r)}u_{0}\Vert _{X_{\infty }}\leq \widetilde{C}\left\Vert u_{0}\right\Vert
_{g,p,\lambda }.
\end{equation*}%
Therefore, considering $u_{0}\in \mathcal{M}_{p,\lambda }^{g}(\Gamma (TM))$
with $\left\Vert u_{0}\right\Vert _{g,p,\lambda }\leq \varepsilon $ for
small-enough $\varepsilon >0$, an application of Lemma \ref{fix} gives the
desired result.\fin

We end this section with a remark highlighting how the above well-posedness
results extend and relate to prior work.

\begin{rem}
\

\begin{enumerate}
\item[$(i)$] The well-posedness result of the previous theorem can be
extended to the space $\mathcal{M}_{p,\lambda }$ without requiring the
smallness of $\lambda $ or the assumption $(ii)(b)$ in Definition \ref{df-m}.

\item[$(ii)$] The well-posedness result for Einstein manifolds with negative
curvature (see Theorem \ref{th-ein}) generalizes the earlier result for
Lebesgue spaces obtained by Pierfelice \cite{Pierfelice2017} in the case of
manifolds with bounded geometry.
\end{enumerate}
\end{rem}

\section{Results on Ricci-flat manifolds}

\label{s7}

For a manifold satisfying $0<\func{Ric}\leq C$, one cannot in general
guarantee either exponential or polynomial decay of the semigroup $%
e^{t(\Delta _{H}+2r)}$ and or its gradient $\nabla e^{t(\Delta _{H}+2r)}.$
However, in the modified case with viscosity treated in Section \ref{s.visc}%
, assuming that the Ricci curvature is positive and bounded allows one to
obtain suitable decay estimates.

In a complementary spirit, in this section we restrict our attention to
Ricci-flat manifolds, that is, manifolds satisfying $\func{Ric}=0$. Notable
instances, previously highlighted in the introduction, include simply
connected Ricci-flat manifolds with bounded geometry or large volume growth,
such as the ALE (asymptotically locally Euclidean) Ricci-flat (see, for
example, \cite{NAKAJIMA1990385}). Additionally, one can consider product
manifolds of the form $\mathbb{R}^{m}\times C$, where $C$ is a simply
connected compact Ricci-flat manifold, encompassing in particular the simply
connected compact Calabi-Yau manifolds $C$.

In this context, the Navier-Stokes equations \eqref{ns1} take the form
\begin{equation}
\begin{cases}
\partial _{t}u+\nabla _{u}u+\dfrac{1}{\rho }\func{grad}\,p=\nu (\Delta
_{H}u^{\flat })^{\sharp }, \\
\func{div}(u)=0, \\
u(0,\cdot )=u_{0},%
\end{cases}
\label{NS-0}
\end{equation}%
and the corresponding mild formulation is given by
\begin{equation}
u(t)=e^{t\nu \Delta _{H}}u_{0}-\int_{0}^{t}e^{(t-\tau )\nu \Delta _{H}}%
\mathbb{P}(\func{div}(u\otimes u))(\tau )\,d\tau .  \label{mild-NS-0}
\end{equation}

Although we do not assume that the injectivity radius is infinite, the
completeness of $M$ ensures that $M=\cup _{R>0}B_{z}(R)$, for any $z\in M$.
Hence, for every integrable function $f$ on $M$, we have the inequality
\begin{equation*}
\int_{M}f\,dV\leq \int_{S_{z}}\int_{0}^{\infty }f(\exp
_{z}(tv))J(t,v)\,dt\,dv\text{, for all }z\in M.
\end{equation*}%
By Bishop's results (see \cite[Theorem III.4.3]{chavel}), if $\func{Ric}\geq
0$, then%
\begin{equation*}
J(t,v)\leq Ct^{m-1}.
\end{equation*}%
Consequently, for a radial function $\phi $, we can write the pushforward
(in polar-coordinate representation)%
\begin{equation*}
\int_{M}\phi (r(x,y))f(y)\,dy\leq \int_{0}^{\infty }\phi
(r)\int_{S_{x}}f(\exp _{x}(rv))J(r,v)\,dv\,dr.
\end{equation*}

By the results of Li-Yau \cite{yau}, if $M$ is a complete Riemannian
manifold without boundary and $\func{Ric}\geq 0$, then the heat kernel $%
G(t,x,y)$ satisfies
\begin{equation*}
G(t,x,y)\leq C_{1}|B_{x}(\sqrt{t})|^{-1}e^{-C_{2}\frac{r(x,y)^{2}}{t}}
\end{equation*}%
and
\begin{equation*}
\nabla _{x}G(t,x,y)\leq C_{1}t^{-\frac{1}{2}}|B_{x}(\sqrt{t})|^{-1}e^{-C_{2}%
\frac{r(x,y)^{2}}{t}}.
\end{equation*}%
Furthermore, by a result of Grigor'yan \cite{Grigoryan2014},
\begin{equation*}
\partial _{t}G(t,x,y)\leq C_{1}t^{-1}|B_{x}(\sqrt{t})|^{-1}e^{-C_{2}\frac{%
r(x,y)^{2}}{t}}.
\end{equation*}%
Assuming $\func{Ric}\geq 0$ and $|B_{x}(R)|\geq CR^{m}$, for $R>0$, then the
previous estimates imply the global Gaussian bounds
\begin{align}
G(t,x,y)& \leq C_{1}t^{-\frac{m}{2}}e^{-C_{2}\frac{r(x,y)^{2}}{t}};
\label{eq-r0} \\
\nabla _{x}G(t,x,y)& \leq C_{1}t^{-\frac{m+1}{2}}e^{-C_{2}\frac{r(x,y)^{2}}{t%
}};  \notag \\
\partial _{t}G(t,x,y)& \leq C_{1}t^{-\frac{m}{2}-1}e^{-C_{2}\frac{r(x,y)^{2}%
}{t}}.  \label{eq-t0}
\end{align}

As before, the first step in obtaining the decay estimates is to bound the
heat semigroup $e^{t\Delta _{g}}:\mathcal{M}_{p,\lambda }^{g}\rightarrow
L^{\infty }$ and $e^{t\Delta _{g}}:\mathcal{M}_{p,\lambda }\rightarrow
L^{\infty }$. Recall that the following pointwise inequality holds (see \cite%
[Lemme 2.3]{Lohou}, \cite{Devyyer2014}):
\begin{equation*}
|e^{t\Delta _{H}}u^{\flat }|\leq e^{t\Delta _{g}}|u|.
\end{equation*}

With these preliminary observations in hand, we can now derive the following
dispersive estimates.

\begin{thr}[Dispersive: $\mathcal{M}_{p,\protect\lambda }^{g}\rightarrow
L^{\infty }$]
\label{th-est-ric-zero-1}Let $(M,g)$ be a complete, simply connected,
non-compact, $m$-dimensional Riemannian manifold without boundary. Let $p\in
\lbrack 1,\infty )$ and $\lambda \in \lbrack 0,m)$ and denote by $u_{0}$ the
initial data for the heat equation.

\begin{enumerate}
\item[$(ii)$] Suppose that $M$ satisfies $\func{Ric}\geq 0$ and has large
volume growth. If $u_{0}\in \mathcal{M}_{p,\lambda }^{g}(\Gamma (TM))$ then
\begin{equation*}
\Vert e^{t\Delta _{g}}u_{0}\Vert _{\infty }\leq Ct^{-\frac{m-\lambda }{2p}%
}\Vert u_{0}\Vert _{g,p,\lambda },\text{ }
\end{equation*}%
for all $t>0,$ where $C>0$ is a constant.

\item[$(ii)$] Suppose that $M$ satisfies $\func{Ric}\geq 0$ and has bounded
geometry. If $u_{0}\in \mathcal{M}_{p,\lambda }(\Gamma (TM))$ then
\begin{equation*}
\Vert e^{t\Delta _{g}}u_{0}\Vert _{\infty }\leq Ct^{-\frac{m-\lambda }{2p}%
}\Vert u_{0}\Vert _{p,\lambda },
\end{equation*}%
for all $t>0,$ where $C>0$ is a constant.
\end{enumerate}
\end{thr}

\noindent {\textbf{Proof.}} It suffices to prove the heat semigroup estimate
for $u_{0}\in \mathcal{M}_{p,\lambda }^{g}(\Gamma (TM))$, since the other
case (item $(ii)$) can be treated similarly. Applying H\"{o}lder's
inequality to the Green representation formula \eqref{green}, we obtain
\begin{equation*}
|u(t,x)|^{p}\leq \int_{M}G(t,x,y)|u_{0}(y)|^{p}\,dy
\end{equation*}%
where $u(t)=e^{t\Delta _{g}}u_{0}$. Since $M$ is complete, we can use the
heat kernel estimate \eqref{eq-r0} and together with the polar coordinate
representation (pushforward measure) to deduce
\begin{align*}
|u(t,x)|^{p}& \leq C_{1}\int_{M}t^{-\frac{m}{2}}e^{-C_{2}\frac{r(x,y)^{2}}{t}%
}|u_{0}(y)|^{p}\,dy \\
& \leq C_{1}\int_{0}^{\infty }t^{-\frac{m}{2}}e^{-C_{2}\frac{r^{2}}{t}%
}\,d\rho (r),
\end{align*}%
where $\rho (r)=\int_{B_{x}(r)}|u_{0}(y)|^{p}\,dy$. Note that, by
integration by parts, we can further handle this expression as follows:
\begin{align*}
|u(t,x)|^{p}& \leq C_{1}\int_{0}^{\infty }t^{-\frac{m}{2}}e^{-C_{2}\frac{%
r^{2}}{t}}\,d\rho (r) \\
& =C_{1}t^{-\frac{m}{2}}\left[ e^{-C_{2}\frac{r^{2}}{t}}\rho (r)\bigg|%
_{0}^{\infty }+C\int_{0}^{\infty }\dfrac{r}{t}e^{-C_{2}\frac{r^{2}}{t}}\rho
(r)\,dr\right] \\
& \leq Ct^{-\frac{m}{2}}\Vert u_{0}\Vert _{g,p,\lambda }^{p}\int_{0}^{\infty
}\dfrac{r^{\lambda +1}}{t}e^{-C_{2}\frac{r^{2}}{t}}\,dr \\
& \leq Ct^{-\frac{m-\lambda }{2}}\Vert u_{0}\Vert _{g,p,\lambda }^{p},
\end{align*}%
which yields the estimate stated in item $(i)$.\fin

In what follows, we derive estimates for the semigroup $e^{t\Delta _{g}}$
acting from $\mathcal{M}_{p,\lambda }^{g}$ to $\mathcal{M}_{q,\lambda }^{g}$
and from $\mathcal{M}_{p,\lambda }$ to $\mathcal{M}_{q,\lambda }$. Note that
the polynomial decay arises from the Euclidean volume growth.

\begin{thr}[Dispersive: $\mathcal{M}_{p,\protect\lambda }^{g}\rightarrow
\mathcal{M}_{q,\protect\lambda }^{g}$]
\label{th-m0} Let $(M,g)$ be a complete, simply connected, non-compact, $m$%
-dimensional Riemannian manifold without boundary. Let $1\leq p\leq q<\infty
$ and $\lambda \in \lbrack 0,m)$ and denote by $u_{0}$ the initial data for
the heat equation.

\begin{enumerate}
\item[$(i)$] Suppose that $M$ satisfies $\func{Ric}\geq 0$ and has large
volume growth. If $u_{0}\in \mathcal{M}_{p,\lambda }^{g}(\Gamma (TM))$ then
\begin{equation*}
\Vert e^{t\Delta _{g}}u_{0}\Vert _{g,q,\lambda }\leq Ct^{-\frac{m-\lambda }{2%
}\left( \frac{1}{p}-\frac{1}{q}\right) }\Vert u_{0}\Vert _{g,p,\lambda },
\end{equation*}%
for all $t>0,$ where $C>0$ is a constant.

\item[$(ii)$] Suppose that $M$ satisfies $\func{Ric}\geq 0$ and has bounded
geometry. If $u_{0}\in \mathcal{M}_{p,\lambda }(\Gamma (TM))$ then
\begin{equation*}
\Vert e^{t\Delta _{g}}u_{0}\Vert _{g,q,\lambda }\leq Ct^{-\frac{m-\lambda }{2%
}\left( \frac{1}{p}-\frac{1}{q}\right) }\Vert u_{0}\Vert _{p,\lambda },
\end{equation*}%
for all $t>0,$ where $C>0$ is a constant.
\end{enumerate}
\end{thr}

\noindent {\textbf{Proof.}} In view of Theorem \ref{th-est-ric-zero-1} and
the interpolation property \eqref{eq-pq} in the case $\func{Ric}=0,$ it
suffices to prove the statement for $p=q$. To this end, applying H\"{o}%
lder's inequality to the Green representation formula \eqref{green}, we
obtain
\begin{equation*}
|u(t,x)|^{p}\leq \int_{M}G(t,x,y)|u_{0}(y)|^{p}\,dy,
\end{equation*}%
where $u(t)=e^{t\Delta _{g}}u_{0}$. Next, using the kernel estimate, we
deduce
\begin{align*}
\int_{B_{x_{0}}(R)}& |u(t,x)|^{p}\,dx\leq
\int_{B_{x_{0}}(R)}\int_{M}G(t,x,y)|u_{0}(y)|^{p}\,dy\,dx \\
& \leq
\int_{B_{x_{0}}(R)}\int_{B_{x_{0}}(2R)}G(t,x,y)|u_{0}(y)|^{p}\,dy\,dx+%
\int_{B_{x_{0}}(R)}\int_{M\smallsetminus
B_{x_{0}}(2R)}G(t,x,y)|u_{0}(y)|^{p}\,dy\,dx \\
& \leq
\int_{B_{x_{0}}(2R)}|u_{0}(y)|^{p}\,dy+C\int_{B_{x_{0}}(R)}r(x,x_{0})^{-k}%
\int_{M\smallsetminus B_{x_{0}}(2R)}r(x,y)^{k}G(t,x,y)|u_{0}(y)|^{p}\,dy\,dx
\\
& \leq B_{x_{0}}(2R)^{\frac{\lambda }{m}}\Vert u_{0}\Vert _{g,p,\lambda
}^{p}+C\int_{B_{x_{0}}(R)}r(x,x_{0})^{-k}\int_{M\smallsetminus
B_{x}(R)}r(x,y)^{k}G(t,x,y)|u_{0}(y)|^{p}\,dy\,dx \\
& \leq CB_{x_{0}}(R)^{\frac{\lambda }{m}}\Vert u_{0}\Vert _{g,p,\lambda
}^{p}+C\int_{B_{x_{0}}(R)}r(x,x_{0})^{-k}\int_{M\smallsetminus
B_{x}(R)}r(x,y)^{k}t^{-\frac{m}{2}}e^{-C_{2}\frac{r(x,y)^{2}}{t}%
}|u_{0}(y)|^{p}\,dy\,dx.
\end{align*}%
Now, proceeding analogously to the proof of Theorem \ref{th-mm}, we arrive
at
\begin{equation*}
\Vert u\Vert _{g,p,\lambda }\leq C\Vert u_{0}\Vert _{g,p,\lambda }.
\end{equation*}%
\fin

By a slight modification of the argument used in the above proof, one can
similarly estimate $\nabla e^{t\Delta _{H}}$.

\begin{thr}[Smoothing: $\mathcal{M}_{p,\protect\lambda }^{g}\rightarrow M_{q,%
\protect\lambda }^{g}$]
\label{th-s0} Let $(M,g)$ be a complete, simply connected, non-compact, $m$%
-dimensional Riemannian manifold without boundary. Let $1\leq p\leq q<\infty
$ and $\lambda \in \lbrack 0,m)$ and denote by $u_{0}$ the initial data for
the heat equation.

\begin{enumerate}
\item[$(i)$] Suppose that $M$ has nonnegative Ricci curvature ($\func{Ric}%
\geq 0$), bounded geometry, and large volume growth. If $u_{0}\in \mathcal{M}%
_{p,\lambda }^{g}(\Gamma (TM))$ then
\begin{equation*}
\Vert \nabla e^{t\Delta _{H}}u_{0}^{\flat }\Vert _{g,q,\lambda }\leq Ct^{-%
\frac{1}{2}-\frac{m-\lambda }{2}\left( \frac{1}{p}-\frac{1}{q}\right) }\Vert
u_{0}\Vert _{g,p,\lambda },
\end{equation*}%
for all $t>0,$ where $C>0$ is a constant.

\item[$(ii)$] Assume that $M$ is Ricci-flat and has bounded geometry. If $%
u_{0}\in \mathcal{M}_{p,\lambda }(\Gamma (TM))$ then
\begin{equation*}
\Vert \nabla e^{t\Delta _{H}}u_{0}^{\flat }\Vert _{g,q,\lambda }\leq Ct^{-%
\frac{1}{2}-\frac{m-\lambda }{2}\left( \frac{1}{p}-\frac{1}{q}\right) }\Vert
u_{0}\Vert _{p,\lambda },
\end{equation*}%
for all $t>0,$ where $C>0$ is a constant.
\end{enumerate}
\end{thr}

\noindent {\textbf{Proof.}} Again, it suffices to consider the case $%
u_{0}\in \mathcal{M}_{p,\lambda }^{g}(\Gamma (TM))$ with $p=q$. By estimate %
\eqref{es-ho}, we have
\begin{equation*}
|\nabla e^{t\Delta _{H}}u_{0}^{\flat }(t,x)|\leq C_{1}\int_{M}t^{-\frac{1+m}{%
2}}e^{-C_{2}\frac{r(x,y)^{2}}{t}}|u_{0}(y)|\,dy,
\end{equation*}%
which, combined with H\"{o}lder's inequality, yields
\begin{align*}
|\nabla e^{t\Delta _{H}}u_{0}^{\flat }(t,x)|^{p}& \leq C\left[
\int_{M}\left( t^{-\frac{m}{2}}e^{-C_{2}\frac{r(x,y)^{2}}{t}}\right) ^{\frac{%
1}{p}+\frac{1}{p^{\prime }}}t^{-\frac{1}{2}}u_{0}(y)\,dy\right] ^{p} \\
& \leq C\left[ \int_{M}t^{-\frac{m}{2}}e^{-C_{2}\frac{r(x,y)^{2}}{t}}t^{-%
\frac{p}{2}}|u_{0}(y)|^{p}\,dy\right] \left[ \int_{M}t^{-\frac{m}{2}%
}e^{-C_{2}\frac{r(x,y)^{2}}{t}}\,dy\right] ^{\frac{p}{p^{\prime }}} \\
& \leq Ct^{-\frac{p+m}{2}}\int_{M}e^{-C_{2}\frac{r(x,y)^{2}}{t}%
}|u_{0}(y)|^{p}\,dy.
\end{align*}%
Integrating over the ball $B_{x_{0}}(R)$, we obtain
\begin{align*}
& \int_{B_{x_{0}}(R)}|\nabla e^{t\Delta _{H}}u_{0}^{\flat
}(t,x)|^{p}\,dx\leq Ct^{-\frac{p+m}{2}}\int_{B_{x_{0}}(R)}\int_{M}e^{-C_{2}%
\frac{r(x,y)^{2}}{t}}|u_{0}(y)|^{p}\,dy\,dx \\
& \leq Ct^{-\frac{p+m}{2}}\left(
\int_{B_{x_{0}}(R)}\int_{B_{x_{0}}(2R)}e^{-C_{2}\frac{r(x,y)^{2}}{t}%
}|u_{0}(y)|^{p}\,dy\,dx+\int_{B_{x_{0}}(R)}\int_{M\smallsetminus
B_{x_{0}}(2R)}e^{-C_{2}\frac{r(x,y)^{2}}{t}}|u_{0}(y)|^{p}\,dy\,dx\right) \\
& \leq Ct^{-\frac{p+m}{2}}\left( t^{\frac{m}{2}}|B_{x_{0}}(R)|^{\frac{%
\lambda }{m}}\Vert u_{0}\Vert _{g,p,\lambda
}^{p}+\int_{B_{x_{0}}(R)}r(x,x_{0})^{-k}\int_{M\smallsetminus
B_{x_{0}}(2R)}r(x,y)^{k}e^{-C_{2}\frac{r(x,y)^{2}}{t}}|u_{0}(y)|^{p}\,dy\,dx%
\right) \\
& \leq Ct^{-\frac{p+m}{2}}\left( t^{\frac{m}{2}}|B_{x_{0}}(R)|^{\frac{%
\lambda }{m}}\Vert u_{0}\Vert _{g,p,\lambda
}^{p}+\int_{B_{x_{0}}(R)}r(x,x_{0})^{-k}\int_{M\smallsetminus
B_{x}(R)}r(x,y)^{k}e^{-C_{2}\frac{r(x,y)^{2}}{t}}|u_{0}(y)|^{p}\,dy\,dx%
\right) .
\end{align*}%
Now, we can proceed as in the proof of Theorem \ref{th-mm} to conclude that
\begin{equation*}
\Vert \nabla e^{t\Delta _{H}}u_{0}^{\flat }(t,x)\Vert _{g,p,\lambda }\leq
Ct^{-\frac{1}{2}}\Vert u_{0}\Vert _{g,p,\lambda }.
\end{equation*}%
\fin

In the setting of Ricci-flat manifolds, the boundedness of the Riesz
transform follows as a special case of more general results, in analogy with
the Euclidean case $\mathbb{R}^{m}$ (see \cite{Kato1992}).

\begin{thr}
\label{th-r0} Let $(M,g)$ be a complete, simply connected, non-compact, $m$%
-dimensional Riemannian manifold without boundary. Let $S:M\times
M\rightarrow \Gamma (TM)$ satisfy
\begin{equation*}
|S(x,y)|\leq Cr(x,y)^{-m},
\end{equation*}%
where $r(x,y)$ denotes the Riemannian distance between $x$ and $y.$ Define
the operator $T:L^{p}(M)\rightarrow L^{p}(\Gamma (TM))$ by%
\begin{equation*}
Tf(x)=\int_{M}S(x,y)f(y)\,dy,
\end{equation*}%
and assume that $T$ is bounded. Under these assumptions, we have:

\begin{enumerate}
\item[$(i)$] If $M$ is Ricci-flat with large volume growth, then $T:\mathcal{%
M}_{p,\lambda }^{g}(M)\rightarrow \mathcal{M}_{p,\lambda }^{g}(\Gamma (TM))$
is bounded;

\item[$(ii)$] If $M$ is Ricci-flat with bounded geometry, then $T:\mathcal{M}%
_{p,\lambda }(M)\rightarrow \mathcal{M}_{p,\lambda }(\Gamma (TM))$ is
bounded.
\end{enumerate}
\end{thr}

\noindent {\textbf{Proof.}} We prove the result for $u_{0}\in \mathcal{M}%
_{p,\lambda }^{g}(M)$, as the case in item $(ii)$ follows by a similar
argument. For each $\rho >0$, define $S_{\rho }$ by
\begin{equation*}
S_{\rho }(x,y)=%
\begin{cases}
S(x,y), & \text{if}\hspace{0.3cm}r(x,y)<\rho , \\
0, & \text{if}\hspace{0.3cm}r(x,y)\geq \rho .%
\end{cases}%
\end{equation*}%
For $f\in \mathcal{M}_{p,\lambda }^{g}$, set
\begin{equation*}
g^{\prime }(x)=\int_{M}S_{\rho }(x,y)f(y)\,dy\text{ }\ \ \text{and \ }%
g^{\prime \prime }(x)=\int_{M}(S-S_{\rho })(x,y)f(y)\,dy.
\end{equation*}

First, we will prove that $g^{\prime \prime }\in L^{\infty }$. Let $%
p^{\prime }$ be the conjugate exponent of $p.$ Choose $\alpha ,\beta >0$
such that
\begin{equation*}
\dfrac{\alpha }{p}+\dfrac{\beta }{p^{\prime }}=m,\,\text{\ \ }\alpha
>\lambda ,\text{ and }\beta >m.
\end{equation*}%
Then, we can estimate
\begin{align*}
|g^{\prime \prime }(x)|& =\bigg|\int_{M}(S-S_{\rho })(x,y)f(y)\,dy\bigg| \\
& \leq \int_{M}|(S-S_{\rho })(x,y)||f(y)|\,dy \\
& \leq C\int_{r(x,y)\geq \rho }r(x,y)^{-\left( \frac{\alpha }{p}+\frac{\beta
}{p^{\prime }}\right) }|f(y)|\,dy \\
& \leq C\left( \int_{r(x,y)\geq \rho }r(x,y)^{-\beta }\,dy\right) ^{\frac{1}{%
p^{\prime }}}\left( \int_{r(x,y)\geq \rho }r(x,y)^{-\alpha
}|f(y)|^{p}\,dy\right) ^{\frac{1}{p}} \\
& \leq C\rho ^{\frac{m-\beta }{p^{\prime }}}\rho ^{\frac{\lambda -\alpha }{p}%
}\Vert f\Vert _{g,p,\lambda }=C\rho ^{-\frac{m-\lambda }{p}}\Vert f\Vert
_{g,p,\lambda }.
\end{align*}%
Thus, we obtain that $g^{\prime \prime }\in L^{\infty }$ and
\begin{equation*}
\left( \int_{B_{x_{0}}(R)}|g^{\prime \prime }(y)|^{p}\,dy\right) ^{\frac{1}{p%
}}\leq \Vert g^{\prime \prime }\Vert _{\infty }R^{\frac{m}{p}}\leq CR^{\frac{%
m}{p}}\rho ^{-\frac{m-\lambda }{p}}\Vert f\Vert _{g,p,\lambda }.
\end{equation*}%
For the term $g^{\prime }$, we have
\begin{align*}
\left( \int_{B_{x_{0}}(R)}|g^{\prime }(x)|^{p}\,dx\right) ^{\frac{1}{p}}&
=\left( \int_{B_{x_{0}}(R)}\left\vert \int_{M}S_{\rho }(x,y)\overline{f}%
(y)\,dy\right\vert ^{p}\,dx\right) ^{\frac{1}{p}} \\
& \leq \left\Vert \int_{M}S_{\rho }(x,y)\overline{f}(y)\,dy\right\Vert _{p},
\end{align*}%
where
\begin{equation*}
\overline{f}(x)=%
\begin{cases}
f(x),\, & \text{if}\hspace{0.3cm}x\in B(x_{0},R+\rho ), \\
0,\, & \text{if}\hspace{0.3cm}x\notin B(x_{0},R+\rho ).%
\end{cases}%
\end{equation*}%
Using the hypothesis that $T:L^{p}(M)\rightarrow L^{p}(M)$ is bounded, we
obtain
\begin{align*}
\left( \int_{B_{x_{0}}(R)}|g^{\prime }(x)|^{p}\,dx\right) ^{\frac{1}{p}}&
\leq \Vert T\overline{f}\Vert _{p} \\
& \leq C\Vert \overline{f}\Vert _{p} \\
& \leq C\left( \int_{B_{x_{0}}(R+\rho )}|f(x)|^{p}\,dx\right) ^{\frac{1}{p}}
\\
& \leq C|B_{x_{0}}(R+\rho )|^{\frac{\lambda }{mp}}\Vert f\Vert _{g,p,\lambda
} \\
& \leq C(R+\rho )^{\frac{\lambda }{p}}\Vert f\Vert _{g,p,\lambda }.
\end{align*}%
Since $Tf=g^{\prime }+g^{\prime \prime }$, it follows that%
\begin{equation*}
\left( \int_{B_{x_{0}}(R)}|Tf(x)|^{p}\,dx\right) ^{\frac{1}{p}}\leq C\left[
R^{\frac{m}{p}}\rho ^{-\frac{m-\lambda }{p}}+(R+\rho )^{\frac{\lambda }{p}}%
\right] \Vert f\Vert _{g,p,\lambda }.
\end{equation*}%
Choosing $\rho =R$, we finally arrive at
\begin{equation*}
\left( \int_{B_{x_{0}}(R)}|Tf(x)|^{p}\,dx\right) ^{\frac{1}{p}}\leq CR^{%
\frac{\lambda }{p}}\Vert f\Vert _{g,p,\lambda }.
\end{equation*}%
\fin

\begin{rem}
\label{Riesz-Lp-bounded-1000}\ Note that the Riesz transform is a special
case of Theorem \ref{th-r0}, since when the Ricci curvature is non-negative,
it is known to be $L^{p}$-bounded; see \cite{Bakry1987}, \cite%
{CoulhonDuong2003}.
\end{rem}

Let $q\in \lbrack 1,\infty )$ and $\lambda \in \lbrack 0,m)$ satisfy $%
m-\lambda \leq q$. We can look for mild solutions of system (\ref{NS-0}) in
the functional space $CX_{\infty }^{g}=CM_{\infty }^{g}\cap X_{\infty }^{g}$%
, where $CM_{\infty }^{g}=\mathcal{C}_{b}((0,{\infty }),\mathcal{M}%
_{p,\lambda }^{g}(\Gamma (TM)))$ and
\begin{equation*}
X_{\infty }^{g}=\bigg\{u\in L_{loc}^{\infty }((0,{\infty }),\mathcal{M}%
_{q,\lambda }^{g}(\Gamma (TM)))\,\bigg|\,t^{\frac{1}{2}-\frac{m-\lambda }{2q}%
}\Vert u(t)\Vert _{g,q,\lambda }\in L^{\infty }(0,{\infty })\bigg\}.
\end{equation*}%
The space $CX_{\infty }^{g}$ is equipped with the norm $\Vert \cdot \Vert
_{CX_{\infty }^{g}}=\Vert u\Vert _{CM_{\infty }^{g}}+\Vert u\Vert
_{X_{\infty }^{g}}$, where
\begin{equation*}
\Vert u\Vert _{CM_{\infty }^{g}}=\sup_{0<t<{\infty }}\Vert u(t)\Vert
_{g,q,\lambda }\text{ \ and \ }\Vert u\Vert _{X_{\infty }^{g}}=\sup_{0<t<{%
\infty }}t^{\frac{1}{2}-\frac{m-\lambda }{2q}}\Vert u(t)\Vert _{g,q,\lambda
}.
\end{equation*}%
Similarly,we define the spaces $CX_{\infty },$ $CM_{\infty }$, and $%
X_{\infty }$ by replacing $\mathcal{M}_{p,\lambda }^{g}$ and $\Vert \cdot
\Vert _{g,p,\lambda }$ by $\mathcal{M}_{p,\lambda }$ and $\Vert \cdot \Vert
_{p,\lambda },$ respectively.

We are now in a position to state our well-posedness result in the
Ricci-flat setting.

\begin{thr}[Well-posedness - Ricci-flat Manifolds]
\label{th-w0}Let $(M,g)$ be a complete, simply connected, non-compact, $m$%
-dimensional Riemannian manifold without boundary. Suppose further that $M$
is Ricci-flat.

\begin{enumerate}
\item[$(i)$] Assume that $M$ has bounded geometry and large volume growth,
and let $u_{0}\in \mathcal{M}_{p,\lambda }^{g}(\Gamma (TM))$ with $%
p=m-\lambda $ and $\func{div}(u_{0})=0.$ Then, the Navier-Stokes system (\ref%
{NS-0}) admits a unique mild solution $u\in CX_{\infty }^{g}$ satisfying $%
\Vert u\Vert _{CX_{\infty }^{g}}\leq 2\widetilde{C}\varepsilon $, for some
constant $\widetilde{C}>0$. Moreover, $u\rightharpoonup u_{0}$ in the sense
of distributions on $M$ as $t\rightarrow 0^{+}$.

\item[$(ii)$] Assume that $M$ has bounded geometry, and let $u_{0}\in
\mathcal{M}_{p,\lambda }(\Gamma (TM))$ with $p=m-\lambda $ and $\func{div}%
(u_{0})=0.$ There exists $\varepsilon >0$ and $\widetilde{C}>0$ such that if
$\Vert u_{0}\Vert _{p,\lambda }\leq \varepsilon $, then the Navier-Stokes
system (\ref{NS-0}) admits a unique mild solution $u\in CX_{\infty }$
satisfying $\Vert u\Vert _{CX_{\infty }}\leq 2\widetilde{C}\varepsilon .$
Moreover, $u\rightharpoonup u_{0}$ in the sense of distributions on $M$ as $%
t\rightarrow 0^{+}$.
\end{enumerate}
\end{thr}

\textbf{Proof. }The result follows from the estimates established in
Theorems \ref{th-m0}, \ref{th-s0} and \ref{th-r0}, combined with a
contraction mapping argument similar to those employed in the proofs of the
results in previous section. We omit the details for brevity, leaving them
to the reader.\fin

\begin{rem}
\label{Rem-th-w0-1} In the results of this section, the assumption that the
manifold is simply connected can be replaced by the weaker requirement that
it be orientable, due to a classical geometric result (see, e.g., \cite[%
Theorem 15.43]{Lee2013}), which ensures the necessary integration
properties. Moreover, the results can be adapted to compact geometries with
the dimension $m\geq 2$. These considerations allow us to cover additional
examples such as the cylinder $\mathbb{R}\times \mathbb{T}_{flat}^{m_{2}}$ ($%
m_{2}\geq 2$), the products $\mathbb{R}^{m_{1}}\times \mathbb{T}%
_{flat}^{m_{2}}$ ($m_{1}\geq 2,$ $m_{2}\geq 1$), and the flat-torus $\mathbb{%
T}_{flat}^{m}$ ($m\geq 2$), among others.
\end{rem}

\noindent

\paragraph*{\textbf{Acknowledgments.}}

V. Chaves-Santos was supported by CAPES (Finance Code $001$), Brazil. L. C.
F. Ferreira was supported by CNPq (grant 312484/2023-2), Brazil.

\noindent

\paragraph*{\textbf{Conflict of Interest Statement.}}

The authors declare that there are no conflicts of interest to disclose.

\noindent

\paragraph*{\textbf{Data availability statement.}}

This manuscript has no associated data.

\end{document}